\documentclass[12pt]{amsart}

\usepackage[foot]{amsaddr}
\usepackage[all]{xy}

\theoremstyle{plain}
\newtheorem{thm}{Theorem}[section]
\newtheorem{theorem}[thm]{Theorem}

\newtheorem{lemma}[thm]{Lemma}
\newtheorem{corollary}[thm]{Corollary}
\newtheorem{proposition}[thm]{Proposition}
\newtheorem*{cor4.18}{Corollary \ref{cor4.18}}

\theoremstyle{definition}
\newtheorem{remark}[thm]{Remark}

\newtheorem{notation}[thm]{Notation}
\newtheorem{definition}[thm]{Definition}

\newtheorem{assumption}[thm]{Assumption}

\newtheorem{question}[thm]{Question}

\newtheorem*{acknowledgements}{Acknowledgements}
\numberwithin{equation}{section}

\newcommand{\sB}{{\mathcal B}}
\newcommand{\sC}{{\mathcal C}}

\newcommand{\sM}{{\mathcal M}}

\newcommand{\sP}{{\mathcal P}}

\newcommand{\sT}{{\mathcal T}}

\newcommand{\C}{{\mathbb C}}

\newcommand{\BP}{{\mathbb P}}

\newcommand{\rd}{{\rm d}}

\def\Gr{\mathop{\rm Gr}\nolimits}

\def\Sym{\mathop{\rm Sym}\nolimits}

\def\Ker{\mathop{\rm Ker}\nolimits}

\begin{document}

\title[Moduli map of second fundamental forms]{Moduli map of second fundamental forms on a nonsingular intersection of two quadrics}
\author[Yewon Jeong]{Yewon Jeong}

\address{Department of Mathematical Sciences, KAIST, Yuseong-gu, Daejeon 305-701, Korea}
\email{jyw57@kaist.ac.kr}

\begin{abstract} 
In \cite{GH}, Griffiths and Harris asked whether a projective complex submanifold of codimension two is determined by the moduli of its second fundamental forms. More precisely, given a nonsingular subvariety $X^n \subset {\mathbb P}^{n+2}$, the second fundamental form $II_{X,x}$ at a point $x \in X$ is a pencil of quadrics on $T_x(X)$, defining a rational map $\mu^X$ from $X$ to a suitable moduli space of pencils of quadrics  on a complex vector space of dimension $n$. The question raised by Griffiths and Harris was whether the image of $\mu^X$ determines $X$. We study this question when $X^n \subset {\mathbb P}^{n+2}$ is a nonsingular intersection of two quadric hypersurfaces of dimension $n >4$. In this case, the second fundamental form $II_{X,x}$ at a general point $x \in X$ is a nonsingular pencil of quadrics. 
Firstly, we prove that  the moduli map $\mu^X$ is dominant over the moduli of nonsingular pencils of quadrics. This gives a negative answer to Griffiths-Harris's question. To remedy the situation, we consider a refined  version  $\widetilde\mu^X$ of the moduli map $\mu^X$, which takes into account the infinitesimal information of $\mu^X$.  Our main result is an affirmative answer in terms of the refined moduli map: 
we prove that the image of  $\widetilde\mu^X$ determines $X$, among  nonsingular intersections of two quadrics.  

\end{abstract}

\maketitle

\section{Introduction}\label{intro}

Given a nondegenerate projective variety (or a complex submanifold) $X \subset \BP^N$ in the complex projective space, the (projective) second fundamental form  $II_{X,x}$ at a nonsingular point $x \in X$ is one of basic projective invariants of $X$. When $c$ is the codimension of $X \subset \BP^N$, the second fundamental form $II_{X,x}$ is a linear system of quadrics of projective dimension at most $c-1$ on the tangent space $T_x(X).$
When $X$ is a hypersurface in $\BP^N$, the second fundamental form at general $x$ is the system generated by a single quadratic form on $T_x(X)$, which is nondegenerate
unless $X$ is ruled in a special way. So for hypersurfaces, the pointwise second fundamental form itself
has no interesting information. When $X$ has codimension 2 in $\BP^N$, the second fundamental form at a point $x$ is a pencil of quadrics on $T_x(X)$. Pencils of quadrics
on a given vector space have nontrivial moduli. So already in codimension 2, the second fundamental form has nontrivial pointwise information.\\

To make it more precise, denote by $\sM_{n}^{\rm PQ}$  the moduli of (nonsingular) pencils of quadrics on a vector space of dimension $n$. (See Definition \ref{n.4} for a precise definition.) Then for a submanifold $X \subset \BP^{n+2}$ of dimension $n\geq3$, there is the moduli map $\mu^X: X^o \to \sM_n^{\rm PQ}$ of second fundamental forms defined on a Zariski open subset $X^o \subset X$ and it is natural  to study the projective geometry of $X$ in terms of the map $\mu^X$. In \cite{GH} p.451, Griffiths and Harris raised the following question.

\begin{question}\label{q.GH} Let $M, M'$ be two submanifolds of codimension 2 in $\BP^{n+2}$ with the
maps $\mu^M: M \to \sM^{\rm PQ}_n$ and $\mu^{M'}: M' \to \sM^{\rm PQ}_n$ arising from the moduli of their second fundamental forms.
Suppose there exists a biholomorphic map $f: M \to M'$ satisfying $\mu^M = \mu^{M'} \circ f$. Does $f$ come from a projective automorphism of $\BP^{n+2}$? \end{question}

 Although this is a very natural question, it seems that there has not been much study on it. In this article, we study this question when the submanifold $X \subset \BP^{n+2}$ is  defined as the intersection of two quadrics. In this simple case, we find that Question \ref{q.GH} has a negative answer:
\begin{theorem}\label{t.1} Let $X \subset \BP^{n+2}$ be a nonsingular intersection of two quadric hypersurfaces with $n\geq3$. Then the morphism $\mu^X: X^o \to \sM^{\rm PQ}_n$ is dominant.
In particular, given any two nonsingular varieties $X, X' \subset \BP^{n+2}$ defined as the intersections of two quadric hypersurfaces, we can always find a biholomorphic map $f:M \to M'$ between some Euclidean open subsets $M \subset X$ and $M' \subset X'$ such that $\mu^M = \mu^{M'} \circ f.$ \end{theorem}

Since there are many choices of $X$ and $X'$ that are not biregular to each other, Theorem \ref{t.1}(=Theorem \ref{theoremA}) gives a negative example for Question \ref{q.GH}. This leads to a  natural problem: how to reformulate Question \ref{q.GH} to have an affirmative answer? In other words, what additional information other than the image of the map $\mu^M$ is needed to determine $M$ up to projective transformation?\\

The result of \cite{HM} gives a partial answer. When $X \subset \BP^{n+2}$ is  a nonsingular intersection of two quadrics  with $n\geq3$, the base locus of the second fundamental form $II_{X,x}$ at $x \in X$ is precisely the VMRT $\sC_x \subset \BP T_x(X)$ consisting of tangent directions of lines on $X$ passing through $x$. So Cartan-Fubini type extension theorem in \cite{HM} gives the following result.

\begin{theorem}\label{t.CF}
For two nonsingular varieties $X, X' \subset \BP^{n+2}$ ($n>3$) defined as the intersections of two quadric hypersurfaces, suppose there exists a biholomorphic map $f:M \to M'$ between connected Euclidean open subsets $M \subset X$ and $M' \subset X'$ such that ${\rm d}_x f: T_x(M) \to T_{f(x)}(M')$ for each $x \in M$ sends the base locus $\sC_x$ of $II_{X,x}$ to the base locus
  $\sC_{f(x)}$ of $II_{X',f(x)}$.
 Then $f$ comes from a projective automorphism of $\BP^{n+2}$. \end{theorem}

 The condition that ${\rm d}_x f: T_x(M) \to T_{f(x)}(M')$ sends the base locus $\sC_x$ of $II_{X,x}$ to the base locus $\sC_{f(x)}$ of $II_{X',f(x)}$  implies $\mu^M = \mu^{M'} \circ f.$ In this sense, Theorem \ref{t.CF} provides a condition strengthening  that of Question \ref{q.GH}, which gives  an affirmative answer. What is unsatisfactory about the condition in Theorem \ref{t.CF} is that it is not formulated in terms of pointwise invariants of $X$.  So the natural problem is to replace it by conditions formulated in terms of pointwise invariants. \\

We will resolve this problem in the following way.
In the setting of Theorem \ref{t.1}, at a general point $x \in X$, the kernel of the derivative d$_x\mu^X$ is a three-dimensional vector subspace in $T_x(X)$, to be  denoted by $\sP_x$. We consider the pair $(II_{X,x},\sP_x)$ at each general point $x\in X$ and define the refined moduli map $\widetilde{\mu}^X: X^{\rm reg}\to \widetilde{\sM}_n^{\rm PQ}$ where we denote by $\widetilde{\sM}_n^{\rm PQ}$ the moduli of pairs of a pencil of quadrics on a vector space of dimension $n$ and a three-dimensional vector subspace in the vector space. This map assigns the projective equivalence class of the pair $(II_{X,x},\sP_x)$ to $x\in X$ and $X^{\rm reg}$ is a dense open subset of $X$ on which the map is well-defined. Then $\widetilde{\mu}^X$ is formulated in terms of pointwise invariants of $X$. We will prove that this information is enough to recognize $X$:

\begin{theorem}\label{t.3}
Let $X, X'$ be two nonsingular varieties in $\BP^{n+2}$ ($n>4$), each of them defined as an intersection of two quadric hypersurfaces. Let $\widetilde{\mu}^X: X^{\rm reg} \to \widetilde{\sM}_n^{\rm PQ}$ and $\widetilde{\mu}^{X'}: (X')^{\rm reg} \to \widetilde{\sM}_n^{\rm PQ}$ be their refined moduli maps of second fundamental forms.  Suppose there exists a biholomorphic map $f:M \to M'$ between connected Euclidean open subsets $M \subset X^{\rm reg}$ and $M' \subset (X')^{\rm reg}$ such that $\widetilde{\mu}^X|_M = \widetilde{\mu}^{X'}|_{M'} \circ f$. Then $f$ comes from a projective automorphism of $\BP^{n+2}$.
\end{theorem}

Theorem \ref{t.3}(=Theorem \ref{theoremB}) says that Question \ref{q.GH} has an affirmative answer for nonsingular intersections of two quadrics with $n>4$ if we replace $\mu^X$ by $\widetilde\mu^X$. The restriction $n>4$ in Theorem \ref{t.3} seems to be fairly strict because $X$ and $\widetilde{\sM}_n^{\rm PQ}$ have the same dimension for $n=4$; if $n\leq3$, then both of ${\sM}_n^{\rm PQ}$ and $\widetilde{\sM}_n^{\rm PQ}$ are trivial, so $X$ can not be characterized by $\widetilde\mu^X$.\\ 

In the course of proving Theorem \ref{t.3}, we obtain the following result on the derivative of the refined moduli map:

\begin{theorem}
Let $X \subset \BP^{n+2}$ be a nonsingular intersection of two quadric hypersurfaces with $n>4$. 
Let $\widetilde{\mu}^X: X^{\rm reg} \to \widetilde{\sM}_n^{\rm PQ}$ be the refined moduli map 
of second fundamental forms on $X$. Denote by $X^{\rm good}$ the subset $\{x\in X^{\rm reg} \mid \Ker(d_x\widetilde{\mu}^X)=0\}\subset X^{\rm reg}$. Then $X^{\rm good}$ is nonempty.
\end{theorem}

To prove Theorem \ref{t.3}, we use a special property of $\sP_x$, which is worth highlighting. 
For a pencil of quadrics $\Phi$ on a vector space $W$ of dimension $n$, we introduce a special class of three-dimensional subspaces of $W$, namely, those `poised by $\Phi$' (See Section 4 for a precise definition). This is  an invariant property and leads to
 a natural correspondence at the level of moduli spaces: 

\begin{cor4.18}\label{intcor}
For a nonsingular pencil of quadrics $\Phi$ on a vector space $W$ of dimension $n>3$, let $S_{\Phi}$ be the set of three-dimensional subspaces poised by $\Phi$. Then there is a correspondence 
$S:=\{([\Phi],[\{\Phi\}\times S_{\Phi}])\in\sM^{\rm PQ}_n\times\widetilde\sM^{\rm PQ}_n\mid [\Phi]\in\sM^{\rm PQ}_n\}$ between $\sM^{\rm PQ}_n$ and $\widetilde\sM^{\rm PQ}_n$. Denote by $p_1$ and $p_2$ the natural projections: 
 \begin{displaymath}
    \xymatrix{
      &  S \ar[dl]_{p_1} \ar[dr]^{p_2}       &   \\
        \sM^{\rm PQ}_n  &  & \widetilde\sM^{\rm PQ}_n }
\end{displaymath}
Then $p_2$ is an embedding and $p_1$ is a submersion having fibers birational to $\BP^{n-1}$.
\end{cor4.18}

A key step in the proof of Theorem \ref{t.3} is to show that 
 the subspace $\sP_x \subset T_x(X)$ is poised by $II_{X,x}$. (See Theorem \ref{t.71}).
 
 \medskip
 It is natural to ask to what extent our results can be generalized to other submanifolds of codimension 2
  in projective space. Do Theorems \ref{t.1} and \ref{t.3} hold, if we replace $X$ by a smooth complete intersection of codimension 2 (or more generally, by a nondegenerate submanifold of codimension 2) in $\BP^{n+2}$?   The  methods we have used do not seem to generalize easily to these general cases. 

\begin{acknowledgements}
This study has been done under the guidance of Jun-Muk Hwang. I would like to thank him for many inspiring discussions and suggestions throughout the research period. I would like to thank Sijong Kwak for his consideration and encouragement. I am grateful to Gary R. Jensen and Emilio Musso for discussions that helped me learn projective differential geometry. A special thank goes to Qifeng Li for many detailed comments on the first draft of the paper. The author is supported by National Researcher Program 2010-0020413 of NRF.

\end{acknowledgements}

\tableofcontents

\begin{notation}
Throughout this paper, we denote by $\BP(V)$ or simply $\BP V$ the projectivization of a complex vector space $V$. We write $[\mathbf v]$ for the point in $\BP V$ corresponding to the line in $V$ generated by a nonzero vector $\mathbf v \in V$. For a linear subspace (resp. a subset) $P$ of $\BP V$, we denote by $\widehat{P}$ the corresponding vector subspace (resp. the affine cone) in $V$. 
\end{notation}


\section{Preliminaries}\label{moduli}
We refer to the references \cite{Reid} and \cite{AvLa}. 

\begin{notation}\label{n.1.1}
Let $W$ be a complex vector space of dimension $n\geq3$ and let $W^*$ be the dual vector space of $W$. We regard the symmetric product $\Sym^2(W^*)$ as the space of quadratic forms (i.e. homogeneous polynomials of degree two) on $W$. For each $\varphi\in\Sym^2(W^*)$, denote by $B_{\varphi}$ the unique symmetric bilinear form on $W$ such that $B_{\varphi}(\mathbf u,\mathbf u)=\varphi(\mathbf u)$ for every $\mathbf u\in W$. In other words, for $\mathbf u,\mathbf v \in W$,
\begin{center}
$B_{\varphi}(\mathbf u,\mathbf v):=\big(\varphi(\mathbf u+\mathbf v)-\varphi(\mathbf u)-\varphi(\mathbf v)\big)/2\,\in\,\C$. 
\end{center}
For $\varphi\in\Sym^2(W^*)$ and $\mathbf u\in W$, we write 
\begin{displaymath}
\mathbf u^{\perp_{\varphi}}:=\{\mathbf v\in W\mid B_{\varphi}(\mathbf u,\mathbf v)=0\} \subset W
\end{displaymath}
for the orthogonal space of $\mathbf u$ in $W$ with respect to the symmetric bilinear form $B_{\varphi}$. For $[\mathbf u] \in \BP (W)$, let $[\mathbf u]^{\perp_{\varphi}}$ indicate $\mathbf u^{\perp_{\varphi}} \subset W$. 
\end{notation}

\begin{definition}\label{d.1.3}
For $\varphi\in\Sym^2(W^*)\setminus\{0\}$, let
\begin{center}
Sing$(\varphi):=\{\mathbf u\in W\mid \mathbf u^{\perp_{\varphi}}=W\} \subset W.$
\end{center}
We say $\varphi$ is $\textit{degenerate}$ if Sing$(\varphi)$ is nontrivial.
\end{definition}

\begin{notation}\label{n.1.4}
If we fix a basis $\{\mathbf w_1, \mathbf w_2, \ldots, \mathbf w_n\}$ of $W$, then each $\varphi\in\Sym^2(W^*)$ corresponds to a symmetric $n\times n$ matrix $\mathbf A^{\varphi}=(\mathbf A^{\varphi}_{ij})$ where
\begin{displaymath}
\mathbf A^{\varphi}_{ij}:=B_{\varphi}(\mathbf w_i,\mathbf w_j).
\end{displaymath}
We denote by  
\begin{center}
$\det(\varphi)$ 
\end{center}
the determinant of $\mathbf A^{\varphi}$. Note that both of $\mathbf A^{\varphi}$ and $\det(\varphi)$ depend on the choice of basis on $W$.
\end{notation}

\begin{proposition}\label{p.1.5}
Let $\varphi\in\Sym^2(W^*)$. With a basis $\{\mathbf w_1, \mathbf w_2, \ldots, \mathbf w_n\}$ of $W$, 
\begin{displaymath}
B_{\varphi}(\mathbf u,\mathbf v)=\sum_{i,j} u_iB_{\varphi}(\mathbf w_i,\mathbf w_j)v_j=\sum_{i,j} u_i\mathbf A^{\varphi}_{ij}v_j
\end{displaymath}
where $\mathbf u=\sum u_i\mathbf w_i\in W$ and $\mathbf v=\sum v_i\mathbf w_i\in W$ with $u_i, v_i \in \C$.
In particular, $\varphi$ is degenerate if and only if $\det(\varphi)=\det(\mathbf A^{\varphi})=0$.    
\end{proposition}
\begin{proof}
Since $B_{\varphi}$ is bilinear, this proposition is given by definitions.
\end{proof}

\begin{definition}\label{d.1.5}
A line in $\BP(\Sym^2(W^*))$ is called a \it pencil of quadratic forms \rm on $W$ or a \it  pencil of quadrics \rm on $W$. 
\end{definition}

\begin{notation}
Let $D \subset \BP(\Sym^2(W^*))$ be the hypersurface consisting of degenerate quadratic forms on $W$, i.e., $D$ is defined by the determinant polynomial of degree $n$ on $\Sym^2(W^*)$. 
\end{notation}

\begin{remark}
The hypersurface $D \subset \BP(\Sym^2(W^*))$ is irreducible and reduced. So a general line in $\BP(\Sym^2(W^*))$ intersects $D$ at $n$ distinct points.
\end{remark}

\begin{definition}\label{d.nonsingular}
A pencil $\Phi \subset \BP(\Sym^2(W^*))$ is $nondegenerate$ if general elements in $\Phi$ are nondegenerate. A nondegenerate pencil $\Phi$ is $nonsingular$ if $\Phi$ intersects $D$ at $n$ distinct points.  
\end{definition}

\begin{definition}\label{d.discrim}
For a pencil $\Phi \subset \BP(\Sym^2(W^*))$, let $\varphi_1,\varphi_2 \in \Sym^2(W^*)$ be two linearly independent quadratic forms in $\widehat{\Phi}\in\Gr(2,\Sym^2(W^*))$. With a basis $\{\mathbf w_1, \mathbf w_2, \ldots, \mathbf w_n\}$ of $W$, the homogeneous polynomial 
\begin{center}
$\det(s\varphi_1-t\varphi_2)$
\end{center}
in $s$ and $t$ is called the $discriminant$ of $\Phi$ with respect to $\varphi_1$ and $\varphi_2$.  
\end{definition}

\begin{remark}
In Definition \ref{d.discrim}, the discriminant polynomial $\det(s\varphi_1-t\varphi_2)$ does not depend on the choice of basis of $W$ up to multiplication by nonzero constants, see $(i)$ of Proposition \ref{p.2}. 
\end{remark}

\begin{proposition}\label{p.3}
With notations in Definition \ref{d.discrim}, the discriminant $\det(s \varphi_1-t\varphi_2)$ is not identically zero if and only if the pencil $\Phi$ is nondegenerate. When the discriminant polynomial is not identically zero, it has no multiple root (or multiple linear factor) if and only if the pencil $\Phi$ is nonsingular. 
\end{proposition}

\begin{proof}
Each element of $\widehat{\Phi}\in\Gr(2,\Sym^2(W^*))$ is expressed in $c_1\varphi_1-c_2\varphi_2$ for some $c_1,c_2\in\C$ and the vanishing of $\det(c_1\varphi_1-c_2\varphi_2)$ means that $c_1\varphi_1-c_2\varphi_2$ is degenerate. Hence, the discriminant polynomial $\det(s \varphi_1-t\varphi_2)$ is not identically zero if and only if general elements in $\Phi$ are nondegenerate. Moreover, $\Phi$ is nonsingular if and only if $\det(s \varphi_1-t\varphi_2)$ is not identically zero and has $n$ distinct roots in $\BP(\C^2)$. 
\end{proof}

\begin{proposition}\label{p.1.6}
Let $\Phi \subset \BP(\Sym^2(W^*))$ be a nonsingular pencil of quadratic forms on $W$ and let $\varphi_1,\varphi_2 \in \Sym^2(W^*)$ be two linearly independent nondegenerate quadratic forms in $\widehat{\Phi}\in\Gr(2,\Sym^2(W^*))$. Then\\ 
(i) the $n$ degenerate elements in $\Phi$ are 
\begin{center}
$ [\varphi_1-\tau_1\varphi_2]$, $[\varphi_1-\tau_2\varphi_2], \ldots , [\varphi_1-\tau_n\varphi_2]$ $\in\Phi$
\end{center}
for some distinct nonzero numbers $\tau_1, \tau_2, \ldots, \tau_n \in \C$;\\
(ii) the roots of the discriminant $\det(s\varphi_1-t\varphi_2)$ of $\Phi$ with respect to $\varphi_1$ and $\varphi_2$ are 
\begin{center}
$[1:\tau_1], [1:\tau_2], \ldots, [1:\tau_n]$ $\in \BP(\C^2)$;
\end{center}
(iii) there is a basis $\{\mathbf w_1', \mathbf w_2', \ldots, \mathbf w_n'\}$ of $W$ such that
\begin{center}
$\varphi_1(\sum z_i \mathbf w_i')=\sum \tau_iz_i^2$ and $\varphi_2(\sum z_i \mathbf w_i')=\sum z_i^2$.
\end{center}
for $\sum z_i \mathbf w_i'\in W$ with $z_i\in \C$. Note that each $\mathbf w_i'$ in (iii) generates \rm{Sing}$(\varphi_1-\tau_i\varphi_2)$.
\end{proposition}

\begin{proof}
Since $\varphi_1$ and $\varphi_2$ are nondegenerate and linearly independent, any degenerate element in $\Phi$ is represented by 
\begin{center}
$\varphi_1-\tau\varphi_2\in\Sym^2(W^*)$ 
\end{center}
for some nonzero number $\tau\in\C$ uniquely. Note that $\varphi_1-\tau\varphi_2$ and $\varphi_1-\tau'\varphi_2$ are linearly independent if $\tau$ and $\tau'$ are distinct numbers. Hence, $(i)$ follows from the nonsingularity of $\Phi$. Since the degeneracy of each $\varphi_1-\tau_i\varphi_2$ is equivalent to the vanishing of $\det(\varphi_1-\tau_i\varphi_2)$, the roots of $\det(s\varphi_1-t\varphi_2)$ are exactly 
\begin{center}
$[1:\tau_1], [1:\tau_2], \ldots, [1:\tau_n]$ $\in \BP(\C^2)$.
\end{center}

To see $(iii)$, let $\mathbf w_i'\in W$ be a nonzero vector satisfying two conditions: 
\begin{center}
$\mathbf w_i'\in\textrm{Sing}(\varphi_1-\tau_i\varphi_2)$ and $\varphi_2(\mathbf w_i')=1$. 
\end{center}
Then $\mathbf w_i'$ and $\mathbf w_j'$ with $i\neq j$ satisfy
\begin{equation}\label{e.1.6.1}
B_{\varphi_1-\tau_i\varphi_2}(\mathbf w_i',\mathbf w_j')=B_{\varphi_1-\tau_j\varphi_2}(\mathbf w_i',\mathbf w_j')=0.
\end{equation}
Since $\varphi_1-\tau_i\varphi_2$ and $\varphi_1-\tau_j\varphi_2$ with $i\neq j$ span $\widehat{\Phi}$, (\ref{e.1.6.1}) implies 
\begin{equation}\label{e.1.6.2}
B_{\varphi_1}(\mathbf w_i',\mathbf w_j')=B_{\varphi_2}(\mathbf w_i',\mathbf w_j')=0.
\end{equation}
Here $\mathbf w_1', \mathbf w_2', \ldots, \mathbf w_n'$ are linearly independent and form a basis of $W$. To obtain a contradiction, suppose 
\begin{center}
$\mathbf w_1'=c_2\mathbf w_2'+c_3\mathbf w_3'+\cdots+c_n\mathbf w_n'$ 
\end{center}
for some $c_2, \ldots, c_n \in \C$. Then $B_{\varphi_2}(\mathbf w_1',\mathbf w_1')=B_{\varphi_2}(\mathbf w_1',\sum_{j> 1} c_j\mathbf w_j')=0$ by (\ref{e.1.6.2}), but this is contrary to the choice of $\mathbf w_1'$. Therefore, $\mathbf w_1', \mathbf w_2', \ldots, \mathbf w_n'$ form a basis of $W$ and satisfy
\begin{center}
$\varphi_1(\sum z_i \mathbf w_i')=\sum \tau_iz_i^2$ and $\varphi_2(\sum z_i \mathbf w_i')=\sum z_i^2$
\end{center}
by the choice of them.
\end{proof}

\begin{definition}\label{d.standard}
We call a basis $\{\mathbf w_1', \mathbf w_2', \ldots, \mathbf w_n'\}$ of $W$ as in $(iii)$ of Proposition \ref{p.1.6} a $standard$ $basis$ of $W$ with respect to the pair $(\varphi_1,\varphi_2)$, or simply a $standard$ $basis$ for $(\varphi_1,\varphi_2)$. 
\end{definition}

\begin{corollary}\label{p.1.6.1}
With notations in Proposition \ref{p.1.6}, a nonsingular pencil $\Phi\subset \BP(\Sym^2(W^*))$ gives a decomposition $W=W_1\oplus W_2\oplus\cdots\oplus W_n$ where $W_i=$ \rm {Sing}$(\varphi_1-\tau_i\varphi_2)\in\Gr(1,W)$. \it In particular, such a  decomposition is independent of the choice of quadratic forms $\varphi_1,\varphi_2\in\widehat\Phi$.
\end{corollary}

\begin{definition}\label{n.1.7}
Let GL$(W)$ be the general linear group on $W$. Then there is a natural GL$(W)$-action on the space $\Sym^2(W^*)$ of quadratic forms on $W$. For $T\in\textrm{GL}(W)$, $\varphi\in\Sym^2(W^*)$, and $\mathbf u\in W$,
\begin{center}
$T.\varphi(\mathbf u):=T^*(\varphi)(\mathbf u)=\varphi(T(\mathbf u))$. 
\end{center}
This action induces an PGL$(W)$-action on the space $\Gr(1,\BP(\Sym^2(W^*)))$ of pencils of quadratic forms on $W$. We say $\Phi_1, \Phi_2 \in \Gr(1,\BP(\Sym^2(W^*)))$ are \textit{projectively equivalent} if 
\begin{center}
$\Phi_2=T^*(\Phi_1):=\{[T^*(\varphi)]\in\BP(\Sym^2(W^*))\mid[\varphi]\in\Phi_1\}$ 
\end{center}
for some $T\in$ GL$(W)$.\end{definition}

\begin{definition}\label{n.4}
A pencil $\Phi\in\Gr(1,\BP(\Sym^2(W^*)))$ is nonsingular in the sense of Definition \ref{d.nonsingular} if and only if $\Phi$ is stable with respect to the PGL$(W)$-action on $\Gr(1,\BP(\Sym^2(W^*)))$ in the sense of geometric invariant theory. (This is part of Theorem 3.1 in \cite{AvLa}.) Thus we write
\begin{center}
$\Gr(1,\BP(\Sym^2(W^*)))^{\rm s}\subset\Gr(1,\BP(\Sym^2(W^*)))$
\end{center}
for the Zariski open subset consisting of nonsingular pencils of quadratic forms. The orbit space
\begin{center}
$\Gr(1,\BP(\Sym^2(W^*)))^{\rm s}\big/{{\rm PGL}(W)}$
\end{center}
can be regarded as a Zariski open subset in the GIT quotient of $\Gr(1,\BP(\Sym^2(W^*)))$ modulo the reductive group PGL$(W)$ \cite{AvLa}. Let
\begin{center}
$\sM^{\rm PQ}_n:=\Gr(1,\BP(\Sym^2(W^*)))^{\rm s}\big/{{\rm PGL}(W)}$. 
\end{center}

If we consider another $n$-dimensional vector space $W'$, any isomorphism $h:W'\to W$ between vector spaces induces an isomorphism $h^*$ between the orbit spaces
\begin{center}
$h^*:\Gr(1,\BP(\Sym^2(W^*)))^{\rm s}\big/{{\rm PGL}(W)}\to\Gr(1,\BP(\Sym^2(W'^*)))^{\rm s}\big/{{\rm PGL}(W')}$
\end{center}
and this isomorphism $h^*$ does not depend on the choice of $h$. Hence $\sM^{\rm PQ}_n$ is defined independently of the choice of $W$ and called the \it moduli of nonsingular pencils of quadratic forms \rm on a complex vector space of dimension $n$. 
\end{definition}

\begin{notation}\label{n.1.8}
We write 
\begin{center}
$[\Phi]\in\sM^{\rm PQ}_n$ 
\end{center}
for the isomorphism class of a nonsingular pencil $\Phi\in\Gr(1,\BP(\Sym^2(W^*)))^{\rm s}$.  
\end{notation}

\begin{remark}
A quadric hypersurface $Q$ in $\BP (W)$ corresponds to an element in $\BP(\Sym^2(W^*))$ and the intersection of two quadric hypersurfaces in $\BP (W)$ corresponds to an element in $\Gr(1,\BP(\Sym^2(W^*)))$. 
\end{remark}

\begin{notation}\label{n.2}
Given an intersection $Y$ of two quadric hypersurfaces in $\BP (W)$, we denote by $\Phi_Y$ the pencil of quadratic forms vanishing on $Y$.
\end{notation} 

The following proposition is part of Proposition 2.1 in \cite{Reid}.

\begin{proposition}\label{p.1}
Let $Y \subset \BP (W)$ be the intersection of two quadric hypersurfaces in $\BP (W)$. Then Y is nonsingular if and only if its pencil $\Phi_Y \subset \BP(\Sym^2(W^*))$ is nonsingular in the sense of Definition \ref{d.nonsingular}. 
\end{proposition}

Proposition \ref{p.1} implies the following theorem.
\begin{theorem}\label{t.1.9}
The moduli of nonsingular intersections of two quadric hypersurfaces in $\BP^{n-1}$ is equivalent to the moduli of nonsingular pencils of quadratic forms on a complex vector space of dimension $n$.
\end{theorem}

On the other hand, there is another moduli space $\sM^{\rm BF}_n$, which is closely related to $\sM^{\rm PQ}_n$.

\begin{notation}\label{n.5}
We denote by $\sB_n$ the space of binary forms (i.e. homogeneous polynomials in two variables) of degree $n$ in $s$ and $t$. Then $\sB_n$ is regarded as a complex vector space of dimension $n+1$ with its standard basis $\{s^nt^0, s^{n-1}t^1, \ldots, s^0t^n\}$. 
\end{notation}

\begin{proposition}\label{p.2}
Let $\Phi\subset\BP(\Sym^2(W^*))$ be a nondegenerate pencil of quadratic forms on $W$. When $\varphi_1,\varphi_2\in\Sym^2(W^*)$ are linearly independent quadratic forms in  $\widehat\Phi\in\Gr(2,\Sym^2(W^*))$, let $[\sigma_1:\tau_1], [\sigma_2:\tau_2], \ldots, [\sigma_n:\tau_n]$ $\in \BP(\C^2)$ be the roots (that may not be distinct) of $\det(s \varphi_1-t\varphi_2)$, i.e., $\det(s \varphi_1-t\varphi_2)$ is a constant multiple of
\begin{center}
$(\tau_1s-\sigma_1t)(\tau_2s-\sigma_2t)\cdots(\tau_ns-\sigma_nt)\in\sB_n$.
\end{center}
Then the roots satisfy the following properties.\\
(i) They are independent of the choice of basis on $W$. Moreover, they are invariant under any isomorphism $h$ from another $n$-dimensional complex vector space $W'$ to $W$. More precisely, the discriminant of the pull-back $h^*(\Phi)$ of $\Phi$ with respect to $h^*(\varphi_1)$ and $h^*(\varphi_2)$ has the same roots of $\det(s \varphi_1-t\varphi_2)$.\\
(ii) They are not uniquely determined by $\Phi$. If we choose $\varphi_1',\varphi_2'\in\widehat\Phi$ instead of $\varphi_1$ and $\varphi_2$, each root $[\sigma_i:\tau_i]\in\BP(\C^2)$ turns into 
\begin{center}
$[c\tau_i+d\sigma_i:a\tau_i+b\sigma_i]\in\BP(\C^2)$ 
\end{center}
for $a,b,c,d\in \C$ such that
\begin{center}
$\varphi_1'=a\varphi_1+b\varphi_2$, $\varphi_2'=c\varphi_1+d\varphi_2$
\end{center}
(with $ad-bc\neq0$).
\end{proposition}

\begin{proof}
We write $s\varphi_1-t\varphi_2$ simply $\varphi$ and consider the symmetric matrix $\mathbf A^{\varphi}$. If we change the basis on $W$, $\mathbf A^{\varphi}$ is changed into $T\mathbf A^{\varphi}T^t$ for some $T\in{\rm GL}(n,\C)$. Since 
\begin{center}$\det(T\mathbf A^{\varphi}T^t)=\det(T)^2\det(s\varphi_1-t\varphi_2)$, 
\end{center}
the roots are invariant under the change of basis. The proof for the rest of $(i)$ is similar. 
And $(ii)$ follows from the fact
\begin{center}
$\det((ct+ds)\varphi_1'-(at+bs)\varphi_2')=\det((ad-bc)(s\varphi_1-t\varphi_2))$.
\end{center}
\end{proof}

\begin{definition}\label{n.61} 
From (ii) of Proposition \ref{p.2}, we define an action of the special linear group $\rm{SL}(2,\C)$ on $\BP(\sB_n)$. Firstly, we define the action on an element 
\begin{center}
$\psi=c_0s^nt^0+c_1s^{n-1}t^{1}+ \cdots +c_ns^{0}t^{n} \in \sB_n\quad(c_i\in\C)$. 
\end{center}
Let us consider a linear decomposition of $\psi$ 
\begin{displaymath}
(\tau_1s-\sigma_1t)(\tau_2s-\sigma_2t)\cdots(\tau_ns-\sigma_nt)
\end{displaymath}
with complex numbers $\sigma_i$ and $\tau_i$. For
\begin{displaymath}
\left[
\begin{array}{cc}
a&b\\
c&d
\end{array}\right] \in \rm{SL}(2,\C),
\end{displaymath}
define
\begin{displaymath}
\left[
\begin{array}{cc}
a&b\\
c&d
\end{array}\right].\,\psi
=
\left[
\begin{array}{cc}
a&b\\
c&d
\end{array}\right].\big((\tau_1s-\sigma_1t)(\tau_2s-\sigma_2t)\cdots(\tau_ns-\sigma_nt)\big)
\end{displaymath}
as the binary form
\begin{displaymath}
\big((a\tau_1+b\sigma_1)s-(c\tau_1+d\sigma_1)t\big)\cdots\big((a\tau_n+b\sigma_n)s-(c\tau_n+d\sigma_n)t\big).
\end{displaymath}
This action induces an $\rm{SL}(2,\C)$-action on $\BP(\sB_n)$. With respect to the action, the class $[\psi]\in\BP(\sB_n)$ is stable (resp. semi-stable) if and only if $\psi$ has no root of multiplicity $\geq n/2$ (resp. $> n/2$) by Proposition 4.1 in \cite{MFK}. 
\end{definition}

\begin{definition}\label{n.6}
Consider the Zariski open subset $\BP(\sB_n)^o\subset\BP(\sB_n)$ of binary forms with no multiple root. The set $\BP(\sB_n)^o$ is contained in the stable locus of $\BP(\sB_n)$ and has the orbit map 
\begin{center}
$q:\BP(\sB_n)^o \to \BP(\sB_n)^o\big/{\rm SL}(2,\C)$. 
\end{center}
Then the orbit space $\BP(\sB_n)^o\big/{\rm SL}(2,\C)$ can be regarded as a Zariski open subset in the GIT quotient of $\BP(\sB_n)$ modulo the reductive group SL$(2,\C)$. We define
\begin{center}
$\sM^{\rm BF}_n:=\BP(\sB_n)^o\big/{\rm SL}(2,\C)$
\end{center}
and call it the \it moduli of binary forms of degree $n$ with no multiple root. \rm 
As in Definition \ref{n.4}, we denote the isomorphism class of $[\psi]\in\BP(\sB_n)$ by $[[\psi]]\in\sM^{\rm BF}_n$.
\end{definition}

\begin{theorem}\label{caniso}
The discriminants of pencils of quadratic forms induce a canonical isomorphism between the moduli spaces $\sM^{\rm PQ}_n$ and $\sM^{\rm BF}_n$ denoted by $\widetilde{\rm D}:\sM^{\rm PQ}_n \to \sM^{\rm BF}_n$. More precisely, $\widetilde{\rm D}([\Phi])=[[\det(s\varphi_1-t\varphi_2)]]$ with notations in Definition \ref{d.discrim}. 
\end{theorem}

\begin{proof}
Proposition \ref{p.2} implies that the two moduli spaces $\sM^{\rm PQ}_n$ and $\sM^{\rm BF}_n$ are in one to one correspondence induced by the discriminants of pencils. This correspondence is indeed an isomorphism between them by Theorem 4.2 of \cite{AvLa}.
\end{proof}

\section{Moduli map of second fundamental forms}\label{pf_thm1}

Let us recall the definitions in \cite{IvLa}, p.76-77. 
\begin{definition}\label{d.secondfu}
Let $V$ be a complex vector space of dimension $n+3$. Given an $n$-dimensional complex submanifold $M\subset\BP V$, its projective Gauss map 
\begin{center}
$\gamma:M\to\Gr(n+1,V)$ 
\end{center}
sends $m\in M$ to the affine tangent space $T_{\mathbf v}(\widehat M)\in\Gr(n+1,V)$ where $\mathbf v\in V$ is a nonzero vector in the line $\widehat{m}\subset V$. Then its derivative d$_m\gamma$ induces an element in $\Sym^2(T^*_m(M))\otimes N_m(M)$, which corresponds to a linear system of quadrics on $T_m(M)$. We call it the $(projective)$ $second$ $fundamental$ $form$ of $M$ at $m$ and denote it by 
\begin{center}
$II_{M,m}$. 
\end{center}
We write 
\begin{center}
$\sC_m\subset\BP T_m(M)$ 
\end{center}
for the base locus of $II_{M,m}$.
\end{definition}

The following is Proposition 3.3.2 of \cite{IvLa}.
\begin{proposition}\label{p.4}
Let $X\subset\BP^N$ be a nonsingular intersection of two quadric hypersurfaces. Then, at each point $x\in X$, the base locus $\sC_x\subset\BP T_x(X)$ of $II_{X,x}$ consists of tangent directions of lines on $X$ passing through $x$. 
\end{proposition}  

\begin{notation}\label{n.7}
Let $X \subset \BP V$ be a nonsingular intersection of two quadric hypersurfaces in $\BP V$ where $V$ is a complex vector space of dimension $n+3$ with $n\geq3$. By Proposition \ref{p.1} and \ref{p.1.6}, we can choose nondegenerate quadratic forms $\varphi_1$ and $\varphi_2$ and a standard basis $\{\mathbf e_1, \mathbf e_2, \cdots, \mathbf e_{n+3}\}$ of $V$ such that $X$ is defined by
\begin{center}
$\varphi_1(\sum z_i\mathbf e_i)=\lambda_1z_1^2+\lambda_2z_2^2+\cdots+\lambda_{n+3}z_{n+3}^2=0$
\end{center} 
and
\begin{center}
$\varphi_2(\sum z_i\mathbf e_i)=z_1^2+z_2^2+\cdots+z_{n+3}^2=0$
\end{center}
with $n+3$ distinct nonzero numbers 
\begin{center}
$\lambda_1, \lambda_2, \ldots, \lambda_{n+3} \in \C$. 
\end{center}
For $x=[\sum x_i \mathbf e_i]\in X$, denote by $\varphi_1|_{T_x(X)}$ (resp. $\varphi_2|_{T_x(X)}$) the quadratic form on $T_x(X)$ induced by $\varphi_1$ (resp. $\varphi_2$). 
As long as we fix the basis $\{\mathbf e_1, \mathbf e_2, \cdots, \mathbf e_{n+3}\}$ of $V$, we denote $\,\sum z_i\mathbf e_i\in V$ (resp. $[\sum z_i\mathbf e_i]\in \BP V$) by $(z_1,z_2,\cdots,z_{n+3})\in V$ (resp. $[z_1:z_2:\cdots:z_{n+3}]\in\BP V$). 
\end{notation}

\begin{proposition}
With Notation \ref{n.7}, the base locus $\sC_x$ of $II_{X,x}$ at $x\in X$ is expressed as follows. Let $v=[v_1:v_2:\cdots:v_{n+3}]$ be another point in $\BP V$. Then the line $l_{x,v} \subset \BP V$ connecting the two points $x$ and $v$ is contained in $X$ if and only if 
\begin{center}
$\varphi_1(x+tv)=0$ and $\varphi_2(x+tv)=0$ for every $t\in\C$.
\end{center}
Since $x\in X$, the conditions above are equivalent to four equalities:
\begin{equation}\label{e.1}
x_1v_1+x_2v_2+\cdots+x_{n+3}v_{n+3}=0,\, \lambda_1x_1v_1+\lambda_2x_2v_2+\cdots+\lambda_{n+3}x_{n+3}v_{n+3}=0,
\end{equation}
\begin{equation}\label{e.2}
v_1^2+v_2^2+\cdots+v_{n+3}^2=0, \,\mathrm{and}\, \lambda_1v_1^2+\lambda_2v_2^2+\cdots+\lambda_{n+3}v_{n+3}^2=0.
\end{equation}
Geometrically (\ref{e.1}) means the line $l_{x,v}$ is tangent to $X$ at $x$. So two equations in (\ref{e.2}) (with (\ref{e.1})) define two quadric hypersurfaces in the projectivized tangent space $\BP T_x(X)$ at $x$. Their intersection is the base locus $\sC_x$ of $II_{X,x}$. 
\end{proposition}

\begin{definition}\label{d.theta}
Let us define the discriminant of second fundamental form $II_{X,x}$ at $x$. We define the orthogonal space $x^{\perp} \subset V$ of x as   
\begin{center}
$x^{\perp_{\varphi_1}}\cap x^{\perp_{\varphi_2}} \subset V$.
\end{center}
If we take an $n$-dimensional vector subspace $W_x \subset x^{\perp}$ complementary to the line $\widehat x \subset x^{\perp}$ (i.e. $x^{\perp}=W_x\oplus\widehat{x}$), then we can identify 
$W_x$ with the tangent space $T_x(X)$ at $x$. So we regard the second fundamental form $II_{X,x}$ as the linear subspace in $\BP(\Sym^2(W^*_x))$ generated by the restrictions $\varphi_1|_{W_x}$ and $\varphi_2|_{ W_x}$. Then, by $(i)$ of Proposition \ref{p.2}, the class 
\begin{center}
$[\det(s\varphi_1|_{W_x}-t\varphi_2|_{W_x})] \in \BP(\sB_n) \cup \{0\}$
\end{center} 
is well-defined and denoted by
\begin{center}
D$(II_{X,x},\varphi_1,\varphi_2)$.
\end{center}
We call it the \textit{discriminant} of $II_{X,x}$ (with respect to $\varphi_1$ and $\varphi_2$).
  \end{definition}

\begin{proposition}\label{p.5.1.1}
The discriminant \rm D$(II_{X,x},\varphi_1,\varphi_2)$ \it of $II_{X,x}$ does not depend on the choice of $W_x$. 
\end{proposition}

\begin{proof}
Consider the quotient space $x^{\perp}/\widehat{x}$. For a coset $\mathbf v+\widehat{x} \in x^{\perp}/\widehat{x}$, define 
\begin{center}
$\widetilde\varphi_1(\mathbf v+\widehat{x}):=\varphi_1(\mathbf v)$ and $\widetilde\varphi_2(\mathbf v+\widehat{x}):=\varphi_2(\mathbf v)$. 
\end{center}
Then $\widetilde\varphi_i$ is well-defined since $B_{\varphi_i}(\mathbf u,\mathbf v)=0$ for every $\mathbf u\in\widehat{x}$, $\mathbf v\in x^{\perp}$, and $i \in \{1,2\}$. So there is a canonical isomorphism 
\begin{center}
$h:W_x \to x^{\perp}/\widehat{x}$ 
\end{center}
that sends $\mathbf w\in W_x$ to $\mathbf w+\widehat{x} \in x^{\perp}/\widehat{x}$ and the pull-back of $\widetilde\varphi_i$ through $h$ is exactly $\varphi_i|_{W_x}$.
Hence, for another subspace $W_x'\subset x^{\perp}$ complementary to $\widehat x$, there is an isomorphism between $W_x'$ and $W_x$ such that the pull-backs of $\varphi_1|_{W_x}$ and $\varphi_2|_{W_x}$ are exactly $\varphi_1|_{W_x'}$ and $\varphi_2|_{W_x'}$ respectively. Then the discriminant D$(II_{X,x},\varphi_1,\varphi_2) \in \BP(\sB_n)$ is well-defined by $(i)$ of Proposition \ref{p.2}.
\end{proof}

\begin{lemma}\label{p55}
For any $x=[x_1:x_2:\cdots:x_{n+3}] \in X$, with Notation \ref{n.7}, at least three $x_i$ must be nonzero. When $x_{n+1},x_{n+2},x_{n+3}$ of $x$ are nonzero, we can take $W_x$ as $x^{\perp} \cap H$ where 
\begin{center}
$H=\{(v_1, v_2, \ldots, v_{n+3}) \in V \mid v_{n+1}=0\}$. 
\end{center}
Then there is a natural basis $\mathbf e_1', \mathbf e_2', \ldots, \mathbf e_{n}'\,$ of $W_x$ such that
\begin{displaymath}
\mathbf e_i'=\mathbf e_i-\frac{(\lambda_{n+3}-\lambda_i)x_i}{(\lambda_{n+3}-\lambda_{n+2})x_{n+2}}\mathbf e_{n+2}-\frac{(\lambda_{n+2}-\lambda_i)x_i}{(\lambda_{n+2}-\lambda_{n+3})x_{n+3}}\mathbf e_{n+3}  
\end{displaymath}
for $1\leq i \leq n$. As a result, 
\begin{equation}\label{e1}
\varphi_1|_{W_x}(\sum z_i\mathbf e_i')=\sum_{i=1}^n\Big(\lambda_{n+2}\frac{(\lambda_{n+3}-\lambda_i)^2x_i^2}{(\lambda_{n+3}-\lambda_{n+2})^2x_{n+2}^2}+\lambda_{n+3}\frac{(\lambda_{n+2}-\lambda_i)^2x_i^2}{(\lambda_{n+3}-\lambda_{n+2})^2x_{n+3}^2}+\lambda_i\Big)z_i^2
\end{equation}
\begin{flushright}
$\displaystyle+\sum_{1\leq i<j\leq n}\Big(\lambda_{n+2}\frac{(\lambda_{n+3}-\lambda_i)(\lambda_{n+3}-\lambda_j)x_ix_j}{(\lambda_{n+3}-\lambda_{n+2})^2x_{n+2}^2}+\lambda_{n+3}\frac{(\lambda_{n+2}-\lambda_i)(\lambda_{n+2}-\lambda_j)x_ix_j}{(\lambda_{n+3}-\lambda_{n+2})^2x_{n+3}^2}\Big)2z_iz_j$
\end{flushright}
and
\begin{equation}\label{e2}
\varphi_2|_{W_x}(\sum z_i\mathbf e_i')=\sum_{i=1}^n\Big(\frac{(\lambda_{n+3}-\lambda_i)^2x_i^2}{(\lambda_{n+3}-\lambda_{n+2})^2x_{n+2}^2}+\frac{(\lambda_{n+2}-\lambda_i)^2x_i^2}{(\lambda_{n+3}-\lambda_{n+2})^2x_{n+3}^2}+1\Big)z_i^2\qquad\qquad\quad
\end{equation}
\begin{flushright}
$\displaystyle+\sum_{1\leq i<j\leq n}\Big(\frac{(\lambda_{n+3}-\lambda_i)(\lambda_{n+3}-\lambda_j)x_ix_j}{(\lambda_{n+3}-\lambda_{n+2})^2x_{n+2}^2}+\frac{(\lambda_{n+2}-\lambda_i)(\lambda_{n+2}-\lambda_j)x_ix_j}{(\lambda_{n+3}-\lambda_{n+2})^2x_{n+3}^2}\Big)2z_iz_j$.
\end{flushright}
\end{lemma}

\begin{proof}
If only one or two $x_i$ are nonzero, to make the sums $\sum x_i^2$ and $\sum \lambda_i x_i^2$ vanish, every $x_i$ must be zero since $\lambda_i$'s are distinct complex numbers. Thus, without loss of generality, we can assume $x_{n+1}x_{n+2}x_{n+3}\neq 0$ and take $W_x$ as in the statement.

\smallskip
On the other hand, $\mathbf v=(v_1, v_2, \ldots, v_{n+3}) \in V$ is contained in $x^{\perp}$ if and only if $\mathbf v$ satisfies  
$\sum x_iv_i=0$ and $\sum \lambda_i x_iv_i=0$.  
Equivalently $\mathbf v \in x^{\perp}$ if and only if
\begin{center}
$\displaystyle v_{n+2}=-\sum_{i=1}^{n+1}\frac{(\lambda_{n+3}-\lambda_i)x_i}{(\lambda_{n+3}-\lambda_{n+2})x_{n+2}}v_i$ and 
$\displaystyle v_{n+3}=-\sum_{i=1}^{n+1}\frac{(\lambda_{n+2}-\lambda_i)x_i}{(\lambda_{n+2}-\lambda_{n+3})x_{n+3}}v_i$
\end{center}
(as long as $x_{n+2}x_{n+3}\neq 0$). Let
\begin{displaymath}
\mathbf e_i':=\mathbf e_i-\frac{(\lambda_{n+3}-\lambda_i)x_i}{(\lambda_{n+3}-\lambda_{n+2})x_{n+2}}\mathbf e_{n+2}-\frac{(\lambda_{n+2}-\lambda_i)x_i}{(\lambda_{n+2}-\lambda_{n+3})x_{n+3}}\mathbf e_{n+3}  
\end{displaymath}
for $1\leq i \leq n$. Then $\mathbf e_1', \mathbf e_2', \ldots, \mathbf e_{n}'\,$ generate $W_x$ and give the formulas (\ref{e1}) and (\ref{e2}) in the statement.
\end{proof}

\begin{proposition}\label{p.5.1}
With Notation \ref{n.7}, the discriminant \rm D\it$(II_{X,x},\varphi_1,\varphi_2)$ at $x$ is represented by the binary form
\begin{equation}\label{e.p.5.1}
\frac{1}{t^2}\sum_{i=1}^{n+3}\lambda_i^2x_i^2\frac{\prod_{j=1}^{n+3}(\lambda_js-t)}{(\lambda_is-t)}\in\sB_n,
\end{equation}
which is equivalent to
\begin{equation}\label{t.1.1}
\sum_{i=1}^{n+1}\Big((\lambda_i-\lambda_{n+2})(\lambda_i-\lambda_{n+3})x_i^2\frac{\prod_{j=1}^{n+1}(\lambda_js-t)}{(\lambda_is-t)}\Big)\in\sB_n
\end{equation}
and equivalent to
\begin{equation}\label{e99}
\sum_{k=0}^n \Big( (-1)^{n-k} \big( \sum_{i=1}^{n+1}(\lambda_i-\lambda_{n+2})(\lambda_i-\lambda_{n+3}) x_i^2 \Gamma_i^k \big) s^kt^{n-k} \Big)
\end{equation}
where $\Gamma_i^0:=1$ and
\begin{displaymath}
\Gamma_i^k:=\sum_{\tiny{
\begin{array}{c}
1\leq j_1<\cdots<j_k\leq n+1\\
i\notin\{j_1,\cdots,j_k\}
\end{array}
}}\lambda_{j_1}\cdots\lambda_{j_k}\,\,
\end{displaymath}
for $1\leq k \leq n$.

\end{proposition}

\begin{proof} 

Firstly, we show that (\ref{e.p.5.1}) equals to (\ref{t.1.1}). Let
\begin{displaymath}
f:=\sum_{i=1}^{n+3}\frac{\lambda_i^2x_i^2}{(\lambda_is-t)}
\end{displaymath}
in (\ref{e.p.5.1}). Note that the rest part $\prod_{j=1}^{n+3}(s-\lambda_jt)$ in (\ref{e.p.5.1}) does not depend on the index $i$. Since $x\in X$, it satisfies $\sum x_i^2=0$ and $\sum \lambda_ix_i^2=0$, which are equivalent to
\begin{equation}\label{e200}
x_{n+2}^2=-\sum_{i=1}^{n+1}\frac{(\lambda_{n+3}-\lambda_i)}{(\lambda_{n+3}-\lambda_{n+2})}x_i^2 \,\,\textrm{and}\,\, 
x_{n+3}^2=-\sum_{i=1}^{n+1}\frac{(\lambda_{n+2}-\lambda_i)}{(\lambda_{n+2}-\lambda_{n+3})}x_i^2.
\end{equation}
By (\ref{e200}),
\begin{center}
$\displaystyle f =\sum_{i=1}^{n+1}\Big(\frac{\lambda_i^2x_i^2}{(\lambda_is-t)}-\frac{(\lambda_{n+3}-\lambda_i)\lambda_{n+2}^2x_i^2}{(\lambda_{n+3}-\lambda_{n+2})(\lambda_{n+2}s-t)}+\frac{(\lambda_{n+2}-\lambda_i)\lambda_{n+3}^2x_i^2}{(\lambda_{n+3}-\lambda_{n+2})(\lambda_{n+3}s-t)}\Big)$\\
$\displaystyle =\sum_{i=1}^{n+1}x_i^2\Big(\frac{\lambda_i^2}{(\lambda_is-t)}-\frac{(\lambda_{n+2}\lambda_{n+3}-\lambda_i\lambda_{n+2}-\lambda_i\lambda_{n+3})t+\lambda_i\lambda_{n+2}\lambda_{n+3}s}{(\lambda_{n+2}s-t)(\lambda_{n+3}s-t)}\Big)$\\
$\displaystyle =\sum_{i=1}^{n+1}x_i^2\Big(\frac{(\lambda_i-\lambda_{n+2})(\lambda_i-\lambda_{n+3})t^2}{(\lambda_is-t)(\lambda_{n+2}s-t)(\lambda_{n+3}s-t)}\Big)$.
\end{center}
This implies the equality between (\ref{e.p.5.1}) and (\ref{t.1.1}). Hence, it is enough to show that $D(II_{X,x},\varphi_1,\varphi_2)$ is represented by (\ref{t.1.1}) since (\ref{e99}) is a rearrangement of (\ref{t.1.1}).\\

We choose $W_x\subset x^{\perp}$ to describe $D(II_{X,x},\varphi_1,\varphi_2)$ first. Without lose of generality, suppose that $x_{n+1},x_{n+2},x_{n+3}$ of $x$ are nonzero and take $W_x$ 
as in Lemma \ref{p55}. 
Let $\mathbf M$ be the symmetric matrix corresponding to the quadratic form $s\varphi_1|_{W_x}-t\varphi_2|_{W_x}$. Throughout the rest of this proof, we will directly show that $\det(\mathbf M)=\det(s \varphi_1|_{W_x}-t\varphi_2|_{W_x})$ is a constant multiple of (\ref{t.1.1}) by applying elementary row and column operations to $\mathbf M$ several times.\\ 

Each $(i,j)$-entry $\mathbf M_{ij}$ of $\mathbf M$ is defined as
\begin{center}
$s \varphi_1|_{W_x}(\mathbf e_i',\mathbf e_j')-t\varphi_2|_{W_x}(\mathbf e_i',\mathbf e_j')$. 
\end{center}
From (\ref{e1}) and (\ref{e2}), the $(i,j)$-entry $\mathbf M_{ij}$ with $i \neq j$ is equal to
\begin{displaymath}
(\lambda_{n+2}s-t)\frac{(\lambda_{n+3}-\lambda_i)(\lambda_{n+3}-\lambda_j)x_ix_j}{(\lambda_{n+3}-\lambda_{n+2})^2x_{n+2}^2}+(\lambda_{n+3}s-t)\frac{(\lambda_{n+2}-\lambda_i)(\lambda_{n+2}-\lambda_j)x_ix_j}{(\lambda_{n+3}-\lambda_{n+2})^2x_{n+3}^2}.
\end{displaymath}
If $i=j$, 
\begin{displaymath}
\mathbf M_{ii}=(\lambda_{n+2}s-t)\frac{(\lambda_{n+3}-\lambda_i)^2x_i^2}{(\lambda_{n+3}-\lambda_{n+2})^2x_{n+2}^2}+(\lambda_{n+3}s-t)\frac{(\lambda_{n+2}-\lambda_i)^2x_i^2}{(\lambda_{n+3}-\lambda_{n+2})^2x_{n+3}^2}+(\lambda_is-t).
\end{displaymath}
Thus $\mathbf M_{ij}$ is decomposed into three parts, i.e. $\mathbf M_{ij}=A_{ij}+B_{ij}+C_{ij}$ where 
\begin{equation}\label{e3}
A_{ij}:=(\lambda_{n+2}s-t)\frac{(\lambda_{n+3}-\lambda_i)(\lambda_{n+3}-\lambda_j)x_ix_j}{(\lambda_{n+3}-\lambda_{n+2})^2x_{n+2}^2},
\end{equation}
\begin{equation}\label{e5}
B_{ij}:=(\lambda_{n+3}s-t)\frac{(\lambda_{n+2}-\lambda_i)(\lambda_{n+2}-\lambda_j)x_ix_j}{(\lambda_{n+3}-\lambda_{n+2})^2x_{n+3}^2},
\end{equation}
and
\begin{center}
$C_{ii}:=(\lambda_is-t)$. ($C_{ij}=0$ if $i\neq j$.) 
\end{center}
Then
\begin{displaymath}
\mathbf M=\left[
\begin{array}{cccc}
A_{11}+B_{11}+C_{11}&A_{12}+B_{12}&\cdots&A_{1n}+B_{1n}\\
A_{21}+B_{21}&A_{22}+B_{22}+C_{22}&\cdots&A_{2n}+B_{2n}\\
\vdots&\vdots&\ddots&\vdots\\
A_{n1}+B_{n1}&A_{n2}+B_{n2}&\cdots&A_{nn}+B_{nn}+C_{nn}\\
\end{array}\right].
\end{displaymath}

Next purpose is omitting most of $A_{ij}$ and $B_{ij}$ terms by subtracting appropriate linear combination of the first and second rows from other rows. We will work under the assumption that at least two among $x_1, x_2, \ldots, x_n$ are nonzero. If all $x_i$ for $1\leq i\leq n$ (or except one) are zero, most of the terms in $M$ vanish and all computations become easier. Let us assume $x_1\neq0$ and $x_2\neq0$.\\  

We apply row operations to $M$ first. From (\ref{e3}), 
\begin{displaymath}
A_{ij}=\Big(\frac{(\lambda_{n+3}-\lambda_i)x_i}{(\lambda_{n+3}-\lambda_1)x_1}\Big)A_{1j}
\end{displaymath} 
for any $i,j$ and the number $\displaystyle\frac{(\lambda_{n+3}-\lambda_i)x_i}{(\lambda_{n+3}-\lambda_1)x_1}$ is independent of $j$. So we can remove all $A_{ij}$ for $i\geq2$ by subtracting constant multiple of the first row from other rows. The resulting matrix is 
\begin{displaymath}
\mathbf M':=\left[
\begin{array}{ccccc}
A_{11}+B_{11}+C_{11}&A_{12}+B_{12}&A_{13}+B_{13}&\cdots&A_{1n}+B_{1n}\\
B'_{21}+C'_{21}&B'_{22}+C_{22}&B'_{23}&\cdots&B'_{2n}\\
B'_{31}+C'_{31}&B'_{32}&B'_{33}+C_{33}&\cdots&B'_{3n}\\
\vdots&\vdots&\vdots&\ddots&\vdots\\
B'_{n1}+C'_{n1}&B'_{n2}&B'_{n3}&\cdots&B'_{nn}+C_{nn}\\
\end{array}\right]
\end{displaymath}
where
\begin{equation}\label{e4}
B_{ij}':=B_{ij}-\Big(\frac{(\lambda_{n+3}-\lambda_i)x_i}{(\lambda_{n+3}-\lambda_1)x_1}\Big)B_{1j}=(\lambda_{n+3}s-t)\frac{(\lambda_1-\lambda_i)(\lambda_{n+2}-\lambda_j)x_ix_j}{(\lambda_{n+3}-\lambda_{n+2})(\lambda_{n+3}-\lambda_1)x_{n+3}^2}
\end{equation}
and 
\begin{displaymath}
C'_{i1}:=-\Big(\frac{(\lambda_{n+3}-\lambda_i)x_i}{(\lambda_{n+3}-\lambda_1)x_1}\Big)(\lambda_1s-t)
\end{displaymath}
for $i\geq2$. Similarly, by the relations
\begin{displaymath}
B_{ij}'=\Big(\frac{(\lambda_1-\lambda_i)x_i}{(\lambda_1-\lambda_2)x_2}\Big)B_{2j}' 
\end{displaymath}
and
\begin{displaymath}
B_{1j}=\Big(\frac{(\lambda_{n+2}-\lambda_1)(\lambda_{n+3}-\lambda_1)x_1}{(\lambda_{n+3}-\lambda_{n+2})(\lambda_1-\lambda_2)x_2}\Big)B_{2j}' 
\end{displaymath}
\smallskip
from (\ref{e4}) and (\ref{e5}), we also remove $B_{ij}'$ (for $i \geq 3$) and $B_{1j}$ by subtracting constant multiple of the second row from other rows. The resulting matrix is  
\begin{displaymath}
\mathbf M'':=\left[
\begin{array}{cccccc}
A_{11}+C'_{11}&A_{12}+C'_{12}&A_{13}&A_{14}&\cdots&A_{1n}\\
B'_{21}+C'_{21}&B'_{22}+C_{22}&B'_{23}&B'_{24}&\cdots&B'_{2n}\\
C''_{31}&C''_{32}&C_{33}&0&\cdots&0\\
C''_{41}&C''_{42}&0&C_{44}&\cdots&0\\
\vdots&\vdots&\vdots&\vdots&\ddots&\vdots\\
C''_{n1}&C''_{n2}&0&0&\cdots&C_{nn}\\
\end{array}\right]
\end{displaymath}
where
\begin{equation}\label{e6}
C''_{i1}:=C'_{i1}-\Big(\frac{(\lambda_1-\lambda_i)x_i}{(\lambda_1-\lambda_2)x_2}\Big)C'_{21}=\frac{(\lambda_2-\lambda_i)x_i}{(\lambda_1-\lambda_2)x_1}(\lambda_1s-t),
\end{equation} 
\begin{displaymath}
C''_{i2}:=-\Big(\frac{(\lambda_1-\lambda_i)x_i}{(\lambda_1-\lambda_2)x_2}\Big)(\lambda_2s-t)
\end{displaymath}
for $i\geq3$,
\begin{equation}\label{e7} 
C'_{11}:=(\lambda_1s-t)-\Big(\frac{(\lambda_{n+2}-\lambda_1)(\lambda_{n+3}-\lambda_1)x_1}{(\lambda_{n+3}-\lambda_{n+2})(\lambda_1-\lambda_2)x_2}\Big)C'_{21}
\end{equation} 
\begin{displaymath}
=\frac{(\lambda_{n+3}-\lambda_1)(\lambda_{n+2}-\lambda_2)}{(\lambda_{n+3}-\lambda_{n+2})(\lambda_1-\lambda_2)}(\lambda_1s-t),
\end{displaymath}
and 
\begin{center}
$\displaystyle C'_{12}:=-\Big(\frac{(\lambda_{n+2}-\lambda_1)(\lambda_{n+3}-\lambda_1)x_1}{(\lambda_{n+3}-\lambda_{n+2})(\lambda_1-\lambda_2)x_2}\Big)(\lambda_2s-t)$.
\end{center}

Now we remove $A_{1j}$ and $B'_{2j}$ for $j\geq3$ by subtracting constant multiple of the $j$-th row from the first and second rows. The resulting matrix is
\begin{displaymath}
\mathbf M''':=\left[
\begin{array}{ccccccc}
P&Q&0&0&\cdots&0\\
R&S&0&0&\cdots&0\\
*&*&(\lambda_3s-t)&0&\cdots&0\\
*&*&0&(\lambda_4s-t)&\cdots&0\\
\vdots&\vdots&\vdots&\vdots&\ddots &\vdots\\
*&*&0&0&\cdots&(\lambda_ns-t)\\
\end{array}\right]
\end{displaymath}  
where
\begin{center}
$\displaystyle P:=(A_{11}+C'_{11})-\sum^{n}_{j=3}\frac{A_{1j}}{C_{jj}}C''_{j1}$,\\
\smallskip
$\displaystyle Q:=(A_{12}+C'_{12})-\sum^{n}_{j=3}\frac{A_{1j}}{C_{jj}}C''_{j2}$,\\
\smallskip
$\displaystyle R:=(B'_{21}+C'_{21})-\sum^{n}_{j=3}\frac{B'_{2j}}{C_{jj}}C''_{j1}$,\\
\end{center}
and
\begin{center}
$\displaystyle S:=(B'_{22}+C'_{22})-\sum^{n}_{j=3}\frac{B'_{2j}}{C_{jj}}C''_{j2}$.
\end{center} 
Then, by (\ref{e6}),
\begin{center}
$\displaystyle P=C'_{11}-\sum_{i=1}^n\frac{(\lambda_{n+3}-\lambda_1)(\lambda_{n+3}-\lambda_i)(\lambda_2-\lambda_i)(\lambda_1s-t)(\lambda_{n+2}s-t)x_i^2}{(\lambda_1-\lambda_2)(\lambda_{n+3}-\lambda_{n+2})^2(\lambda_is-t)x_{n+2}^2}$.
\end{center}
If we apply (\ref{e200}) to (\ref{e7}), then
\begin{displaymath}
\begin{split}
C'_{11} &= \frac{(\lambda_{n+3}-\lambda_1)(\lambda_{n+2}-\lambda_2)(\lambda_1s-t)}{(\lambda_1-\lambda_2)(\lambda_{n+3}-\lambda_{n+2})x_{n+2}^2}x_{n+2}^2\\
 &= -\sum_{i=1}^{n+1}\frac{(\lambda_{n+3}-\lambda_1)(\lambda_{n+3}-\lambda_i)(\lambda_{n+2}-\lambda_2)(\lambda_1s-t)}{(\lambda_1-\lambda_2)(\lambda_{n+3}-\lambda_{n+2})^2x_{n+2}^2}x_{i}^2.
\end{split}
\end{displaymath}
As a result,
\begin{center}
$\displaystyle P=-\sum_{i=1}^n\frac{(\lambda_{n+3}-\lambda_1)(\lambda_{n+3}-\lambda_i)(\lambda_{n+2}-\lambda_i)(\lambda_1s-t)(\lambda_2s-t)}{(\lambda_1-\lambda_2)(\lambda_{n+3}-\lambda_{n+2})^2(\lambda_is-t)x_{n+2}^2}x_i^2$\\
\smallskip
$\displaystyle \qquad -\frac{(\lambda_{n+3}-\lambda_1)(\lambda_{n+3}-\lambda_{n+1})(\lambda_{n+2}-\lambda_2)(\lambda_1s-t)}{(\lambda_1-\lambda_2)(\lambda_{n+3}-\lambda_{n+2})^2x_{n+2}^2}x_{n+1}^2$.
\end{center}
By similar computations, we obtain the following formulas:
\begin{center}
$\displaystyle P=\frac{-1}{(\lambda_{n+3}-\lambda_{n+2})^2x_{n+2}^2}\Big(\frac{(\lambda_{n+3}-\lambda_1)(\lambda_{n+3}-\lambda_{n+1})(\lambda_{n+2}-\lambda_2)(\lambda_1s-t)}{(\lambda_1-\lambda_2)}x_{n+1}^2+F\Big)$,\\
\medskip
$\displaystyle Q=\frac{x_1(x_2)^{-1}}{(\lambda_{n+3}-\lambda_{n+2})^2x_{n+2}^2}\Big(\frac{(\lambda_{n+3}-\lambda_1)(\lambda_{n+3}-\lambda_{n+1})(\lambda_{n+2}-\lambda_1)(\lambda_2s-t)}{(\lambda_1-\lambda_2)}x_{n+1}^2+F\Big)$,\\
\medskip
$\displaystyle R=\frac{-(x_1)^{-1}x_2}{(\lambda_{n+3}-\lambda_{n+2})x_{n+3}^2}\Big(\frac{(\lambda_{n+3}-\lambda_2)(\lambda_{n+2}-\lambda_{n+1})(\lambda_1s-t)}{(\lambda_{n+3}-\lambda_1)}x_{n+1}^2+G\Big)$,
\end{center}
and
\begin{center}
$\displaystyle S=\frac{1}{(\lambda_{n+3}-\lambda_{n+2})x_{n+3}^2}\Big({(\lambda_{n+2}-\lambda_{n+1})(\lambda_2s-t)}x_{n+1}^2+G\Big)$
\end{center}
where 
\begin{center}
$\displaystyle F=\sum_{i=1}^n\frac{(\lambda_{n+3}-\lambda_1)(\lambda_{n+3}-\lambda_i)(\lambda_{n+2}-\lambda_i)(\lambda_1s-t)(\lambda_2s-t)}{(\lambda_1-\lambda_2)(\lambda_is-t)}x_i^2$
\end{center}
and
\begin{center}
$\displaystyle G=\sum_{i=1}^n\frac{(\lambda_{n+3}-\lambda_i)(\lambda_{n+2}-\lambda_i)(\lambda_1s-t)(\lambda_2s-t)}{(\lambda_{n+3}-\lambda_1)(\lambda_is-t)}x_i^2$.
\end{center}

Since 
\begin{equation}\label{e20}
\det(\mathbf M)=\det(\mathbf M''')=(PS-QR)\prod_{j=3}^n(\lambda_js-t),
\end{equation}
we need to compute the term $(PS-QR)$.\\

On the other hand, let 
\begin{displaymath}
P':=(\lambda_{n+3}-\lambda_{n+2})^2x_{n+2}^2P,
\end{displaymath}
\begin{center} 
$\displaystyle Q':=\frac{(\lambda_{n+3}-\lambda_{n+2})^2x_2x_{n+2}^2}{x_1}Q$,\\ 
$\displaystyle R':=\frac{(\lambda_{n+3}-\lambda_{n+2})x_1x_{n+3}^2}{x_2}R$,\\ 
\end{center}
and
\begin{center}
$S':=(\lambda_{n+3}-\lambda_{n+2})x_{n+3}^2S$.  
\end{center}
Then
\begin{center}
$P'+Q'=-{(\lambda_{n+3}-\lambda_1)(\lambda_{n+3}-\lambda_{n+1})(\lambda_{n+2}s-t)}x_{n+1}^2$
\end{center}
and
\begin{center}
$R'+S'\displaystyle =-\frac{(\lambda_{1}-\lambda_2)(\lambda_{n+2}-\lambda_{n+1})(\lambda_{n+3}s-t)}{(\lambda_{n+3}-\lambda_1)}x_{n+1}^2$.
\end{center}
Since
\begin{equation}\label{e21}
P'S'-Q'R'=(P'+Q')S'-Q'(R'+S')={(\lambda_{n+3}-\lambda_{n+2})^3x_{n+2}^2x_{n+3}^2}(PS-QR),
\end{equation}\label{e22}
we compute $(P'+Q')S'-Q'(R'+S')$ instead of $(PS-QR)$. \\

By computing the coefficient of $x_i^2x_{n+1}^2$ in $(P'+Q')S'-Q'(R'+S')$,
\begin{equation}\label{e212}
\begin{split}
(&P'+Q')S'-Q'(R'+S')=\\
&-(\lambda_{n+3}-\lambda_{n+2})x_{n+1}^2\sum_{i=1}^{n+1}\Big((\lambda_{n+3}-\lambda_i)(\lambda_{n+2}-\lambda_i)\frac{(\lambda_1s-t)(\lambda_2s-t)(\lambda_{n+1}s-t)}{(\lambda_is-t)}x_i^2\Big).
\end{split}
\end{equation}
Then 
\begin{center}
$\det (\mathbf M)=\displaystyle\frac{-x_{n+1}^2}{(\lambda_{n+3}-\lambda_{n+2})^2x_{n+2}^2x_{n+3}^2}\sum_{i=1}^{n+1}\Big((\lambda_i-\lambda_{n+2})(\lambda_i-\lambda_{n+3})x_i^2\frac{\prod_{j=1}^{n+1}(\lambda_js-t)}{(\lambda_is-t)}\Big)$.
\end{center}
by (\ref{e20}), (\ref{e21}) and (\ref{e212}). Therefore, $\det(s \varphi_1|_{W_x}-t\varphi_2|_{W_x})=\det (\mathbf M)$ is a constant multiple of (\ref{t.1.1}) in the statement.
\end{proof}

\begin{proposition}\label{p5}
The binary form (\ref{t.1.1}) in Proposition \ref{p.5.1} is not identically zero for any $x\in X$. Hence, the map sending $x \in X$ to the discriminant \rm D\it$(II_{X,x}, \varphi_1, \varphi_2)\in \BP(\sB_n)$ defines a morphism 
\begin{center}
$\theta_{(\varphi_1,\varphi_2)}: X \to \BP(\sB_n)$ 
\end{center}
called the $discriminant$ $map$ on $X$ with respect to the pair $(\varphi_1,\varphi_2)$. The discriminant map $\theta_{(\varphi_1,\varphi_2)}$ satisfies the following properties.\\
(i) Let ${\rm sq}:\BP V \to\BP V$ be the morphism sending  $v=[v_1:v_2:\cdots:v_{n+3}] \in \BP V$ to $[v_1^2:v_2^2:\cdots:v_{n+3}^2] \in \BP V$. Then the image ${\rm sq}(X)$ is a linear subspace of codimension two in $\BP V$ and there is a linear isomorphism 
\begin{center}
$\theta'_{(\varphi_1,\varphi_2)}:{\rm sq}(X)\subset\BP V\to\BP(\sB_n)$ 
\end{center}
such that $\theta_{(\varphi_1,\varphi_2)}=\theta'_{(\varphi_1,\varphi_2)}\circ{\rm sq}|_X$. Therefore, $\theta_{(\varphi_1,\varphi_2)}$ is surjective as a set map.\\
(ii) For each class $[c_0s^0t^n+c_1s^1t^{n-1}+\cdots+c_ns^nt^0]\in\BP(\sB_n)$,
its inverse image 
\begin{center}
$\theta^{-1}_{(\varphi_1,\varphi_2)}([c_0s^0t^n+c_1s^1t^{n-1}+\cdots+c_ns^nt^0])\subset X$ 
\end{center}
through $\theta_{(\varphi_1,\varphi_2)}$ is the finite set
\begin{center}
$\{ [v_1:v_2:\cdots:v_{n+3}]\in \BP V\,|\,v_i^2=\frac{c_{0}\lambda_i^{n}+c_{1}\lambda_i^{n-1}+\cdots+c_{n}\lambda_i^{0}}{\prod_{j\neq i}(\lambda_i-\lambda_j)},\, 1\leq i \leq n+3\}$.
\end{center}
\end{proposition}

\begin{proof}
Recall (\ref{e99}) in Proposition \ref{p.5.1}, which is

\begin{equation}\label{e9}
\sum_{k=0}^n \Big( (-1)^{n-k} \big( \sum_{i=1}^{n+1}(\lambda_i-\lambda_{n+2})(\lambda_i-\lambda_{n+3}) x_i^2 \Gamma_i^k \big) s^kt^{n-k} \Big)
\end{equation}
where $\Gamma_i^0:=1$ and 
\begin{displaymath}
\Gamma_i^k:=\sum_{\tiny{
\begin{array}{c}
1\leq j_1<\cdots<j_k\leq n+1\\
i\notin\{j_1,\cdots,j_k\}
\end{array}
}}\lambda_{j_1}\cdots\lambda_{j_k}.
\end{displaymath}
When we denote by $a_k$ the coefficient of $s^kt^{n-k}$ in (\ref{e9}), we obtain an $(n+1)\times(n+1)$ matrix $\Lambda$ such that 
\begin{equation}\label{e10}
\Lambda\cdot\left[
\begin{array}{c}
x_1^2\\
x_2^2\\
\vdots\\
x_{n+1}^2\\
\end{array}\right]=\left[
\begin{array}{c}
a_0\\
a_{1}\\
\vdots\\
a_n\\
\end{array}\right],
\end{equation}
i.e., the $j$-th column of $\Lambda$ is
\begin{displaymath}
\left[
\begin{array}{c}
(-1)^{n}(\lambda_j-\lambda_{n+2})(\lambda_j-\lambda_{n+3})\Gamma_j^0\\
(-1)^{n-1}(\lambda_j-\lambda_{n+2})(\lambda_j-\lambda_{n+3})\Gamma_j^1\\
\vdots\\
(-1)^0(\lambda_j-\lambda_{n+2})(\lambda_j-\lambda_{n+3})\Gamma_j^n\\
\end{array}\right].
\end{displaymath}
Let $\Lambda'$ be an $(n+1)\times (n+1)$ matrix with $(i,j)$-entry
\begin{displaymath}
\Lambda'_{ij}=(-1)^n\lambda_i^{n+1-j}\cdot\prod_{k'\neq i}(\lambda_i-\lambda_{k'})^{-1}
\end{displaymath}
where $k'$ varies from 1 to $n+3$ except $i$. 
Note that the $(i,j)$-entry of the matrix $\Lambda'\Lambda$ is 
\begin{displaymath}
(\lambda_j-\lambda_{n+2})(\lambda_j-\lambda_{n+3})\cdot{\prod_{k''\neq j}(\lambda_i-\lambda_{k''})}\cdot{\prod_{k'\neq i}(\lambda_i-\lambda_{k'})^{-1}}
\end{displaymath}
where $k''$ varies from 1 to $n+1$ except $j$. Thus $\Lambda'\Lambda$ equals to the $(n+1)\times(n+1)$ identity matrix, and $\Lambda$ is invertible.

\smallskip
To see that (\ref{e9}) is not identically zero, suppose that all $c_j$'s are zero. Then each $x_i$ must be zero for $1\leq i\leq n+1$ since $\Lambda$ is invertible. This contradicts to the fact that $x=[x_1:x_2:\cdots:x_{n+3}]$ is a point in $X$. So the discriminant D$(II_{X,x},\varphi_1, \varphi_2)$ can not be identically zero, and the discriminant map $\theta_{(\varphi_1,\varphi_2)}$ is well-defined. 
To conclude, the assertion in $(i)$ is a direct interpretation of (\ref{e10}) and the assertion in $(ii)$ is obtained by multiplying both sides in (\ref{e10}) by $\Lambda^{-1}$.
\end{proof}

From $(ii)$ of Proposition \ref{p5}, we obtain a purely combinatorial proposition:
\begin{proposition}\label{l.1}
Let $l_1,l_2,\cdots,l_{n+3},$ $a_1,a_2,\cdots,a_{n}$ be distinct complex numbers with $n\geq3$. Then $\displaystyle \sum_{i=1}^{n+3}\frac{P_i}{Q_i}=0$ and $\displaystyle \sum_{i=1}^{n+3}l_i\frac{P_i}{Q_i}=0$ where $\displaystyle P_i={\prod_{l=1}^{n}(l_i-a_l)}$ and $\displaystyle Q_i={\prod_{j\neq i}(l_i-l_j)}$. 
\end{proposition}

\begin{proof}
Let $V$ be a complex vector space of dimension $n+3$ with coordinates $z_1, z_2, \ldots, z_{n+3}$. Consider a subvariety $X' \subset \BP V$ defined by
\begin{equation}\label{e.3.1}
\varphi_1'=l_1z_1^2+l_2z_2^2+\cdots+l_{n+3}z_{n+3}^2=0 \,\,{\rm and }\,\, \varphi_2'=z_1^2+z_2^2+\cdots+z_{n+3}^2=0,
\end{equation}
and a point $x'=[x_1':x_2':\ldots:x_{n+3}'] \in X'$ such that the roots of D$(II_{X',x'},\varphi'_1, \varphi'_2)$ are exactly $[1:a_1]$, $[1:a_2],\ldots,[1:a_n]\in\BP(\C^2)$. This is possible since the discriminant map 
\begin{center}
$\theta_{(\varphi_1',\varphi_2')} : X' \to \BP (\sB_n)$ 
\end{center}
of $X'$ is surjective as a set map, see $(i)$ of Proposition \ref{p5}. Then $(ii)$ of Proposition \ref{p5} says
\begin{equation}\label{e12}
\big[(x_1')^2:(x_2')^2:\cdots:(x_{n+3}')^2\big]=\big[\frac{P_1}{Q_1}:\frac{P_2}{Q_2}:\cdots:\frac{P_{n+3}}{Q_{n+3}}\big] \,\in\, \BP(\C^{n+3}).
\end{equation}
Combining (\ref{e.3.1}) and (\ref{e12}), we obtain $\displaystyle \sum_{i=1}^{n+3}\frac{P_i}{Q_i}=0$ and $\displaystyle \sum_{i=1}^{n+3}l_i\frac{P_i}{Q_i}=0$.
\end{proof}

\begin{proposition}\label{p.sl2change}
Let $\varphi_1'=a\varphi_1+b\varphi_2\in\widehat\Phi_X$ and $\varphi_2'=c\varphi_1+d\varphi_2\in\widehat\Phi_X$ with $ad-bc=1$, then the discriminant map $\theta_{(\varphi_1',\varphi_2')}: X \to \BP(\sB_n)$ with respect to $(\varphi_1',\varphi_2')$ sends $x\in X$ to 
\begin{displaymath}
\left[
\begin{array}{cc}
a&b\\
c&d
\end{array}\right].\,\theta_{(\varphi_1,\varphi_2)}(x)\,\in\,\BP(\sB_n)
\end{displaymath}
where the SL(2,$\C$)-action on $\BP(\sB_n)$ is as in Definition \ref{n.61}.
\end{proposition}

\begin{proof}
In Proposition \ref{p5}, if the discriminant map $\theta_{(\varphi_1,\varphi_2)}: X \to \BP(\sB_n)$ sends $x\in X$ to 
\begin{center}
$[(\tau_1s-\sigma_1t)(\tau_2s-\sigma_2t)\cdots(\tau_ns-\sigma_nt)]\in\BP(\sB_n)$, 
\end{center}
then the degenerate elements in $II_{X,x}$ are the quadrics on $T_x(X)$ induced by 
\begin{center}
$[\sigma_1\varphi_1-\tau_1\varphi_2],[\sigma_2\varphi_1-\tau_2\varphi_2], \ldots, [\sigma_n\varphi_1-\tau_n\varphi_2]\in\Phi_X$.
\end{center}
Since 
\begin{center}
$(c\tau_i+d\sigma_i)\varphi_1'-(a\tau_i+b\sigma_i)\varphi_2'=(ad-bc)(\sigma_i\varphi_1-\tau_i\varphi_2)=\sigma_i\varphi_1-\tau_i\varphi_2$, 
\end{center}
the discriminant map $\theta_{(\varphi_1',\varphi_2')}: X \to \BP(\sB_n)$ with respect to $(\varphi_1',\varphi_2')$ sends $x$ to
\begin{center}
$\big[\big((a\tau_1+b\sigma_1)s-(c\tau_1+d\sigma_1)t\big)\cdots\big((a\tau_n+b\sigma_n)s-(c\tau_n+d\sigma_n)t\big)\big]\in\BP(\sB_n)$, 
\end{center}
which is exactly
\begin{displaymath}
\left[
\begin{array}{cc}
a&b\\
c&d
\end{array}\right].\,\theta_{(\varphi_1,\varphi_2)}(x)\,\in\,\BP(\sB_n).
\end{displaymath}
\end{proof}

\begin{definition}\label{d.modulimap}
When $X\subset\BP^{n+2}$ is a nonsingular intersection of two quadric hypersurfaces, let $X^o\subset X$ be the Zariski open subset on which the second fundamental forms are nonsingular in the sense of Definition \ref{d.nonsingular}. Then we define a morphism 
\begin{center}
$\mu^X: X^o \to \sM^{\rm PQ}_n$
\end{center}
assigning the isomorphism class $[II_{X,x}]\in\sM^{\rm PQ}_n$ to $x\in X^o$ and call it the \textit{moduli map of second fundamental forms} on $X$. 
\end{definition}

\begin{proposition}\label{p.commdiag}
With notations in Definition \ref{d.modulimap}, there is a commutative diagram
\begin{displaymath}
\begin{array}{cccc}
X^o\quad& \xrightarrow{\quad\theta_{(\varphi_1,\varphi_2)}|_{X^o}} & \,\theta_{(\varphi_1,\varphi_2)}(X^o)&=\,\,\BP (\sB_{n})^o\subset\BP (\sB_{n})\\
\Bigg\downarrow\mu^X & & \quad\Bigg\downarrow q|_{\BP (\sB_{n})^o}\\
\sM^{\rm PQ}_n \,& \xrightarrow{\quad\quad\, \widetilde{\rm D} \,\quad\quad} & \sM^{\rm BF}_n\quad\\ 
\end{array}
\end{displaymath}
where $q$ is as in Definition \ref{n.6}.
\end{proposition}

Now we are ready to prove the main result in this section, which is Theorem \ref{t.1} in the introduction.
\begin{theorem} \label{theoremA} Let $X \subset \BP^{n+2}$ be a nonsingular intersection of two quadric hypersurfaces with $n\geq3$. Then the morphism $\mu^X: X^o \to \sM^{\rm PQ}_n$ is dominant.
In particular, given any two nonsingular varieties $X, X' \subset \BP^{n+2}$ defined as the intersections of two quadric hypersurfaces, we can always find a biholomorphic map $f:M \to M'$ between some Euclidean open subsets $M \subset X$ and $M' \subset X'$ such that $\mu^M = \mu^{M'} \circ f.$ \end{theorem}

\begin{proof}
By Proposition \ref{p5}, the discriminant map $\theta_{(\varphi_1,\varphi_2)}$ is a dominant morphism. Composing $\theta_{(\varphi_1,\varphi_2)}|_{X^o}$ with $q|_{\BP (\sB_{n})^o}$ in the commutative diagram in Proposition \ref{p.commdiag}, the map 
\begin{center}
$q|_{\BP (\sB_{n})^o}\circ\theta_{(\varphi_1,\varphi_2)}|_{X^o} = \widetilde{\rm D}\circ\mu^X : X^o \to \sM^{\rm BF}_n$ 
\end{center}
is immediately a dominant morphism. Since $\widetilde{\rm D}$ is an isomorphism, $\mu^X$ is also a dominant morphism. For the second assertion, we can find such a biholomorphic map since $\theta_{(\varphi_1,\varphi_2)}$ is finite. Consider a small Euclidean open subset $U \subset \BP (\sB_{n}) \setminus Z$. If $U$ is small enough, there are open subsets $M \subset X$  and $M' \subset X'$ on which the restrictions $\theta_{(\varphi_1,\varphi_2)}|_M$ and $\theta_{(\varphi_1',\varphi_2')}|_{M'}$ are biholomorphic maps onto $U$. Then the composition 
\begin{center}
$f:=(\theta_{(\varphi_1',\varphi_2')}|_{M'})^{-1}\circ(\theta_{(\varphi_1,\varphi_2)}|_M) : M\to M'$ 
\end{center}
satisfies $\mu^M = \mu^{M'} \circ f$.
\end{proof}

\begin{remark}
Recall that the moduli of nonsingular intersections of two quadric hypersurfaces in $\BP^{n+2}$ is equivalent to $\sM^{\rm PQ}_{n+3}$, which is isomorphic to $\sM^{\rm BF}_{n+3}$. Since the dimension of $\sM^{\rm BF}_{n+3}$ is $n$, these moduli spaces are nontrivial for positive $n$. So, Theorem \ref{t.1} gives a negative answer to Question \ref{q.GH}.
\end{remark}

\begin{remark}
Proposition \ref{p5} does not imply that every type of pencils of quadrics on a complex vector space of dimension $n$ appear on $X$. The result only says that all types of nonsingular pencils of quadrics can appear on $X$. By the language of Segre symbols, one can show that not every singular pencil of quadrics can be realized as the second fundamental form at some point of a given $X$. 
\end{remark}

\section{Three-dimensional subspaces poised by a pencil of quadrics}\label{poised}

\begin{definition}\label{d.4.2}
Let $W$ be a complex vector space of dimension $n\geq3$. We say that a pair $(\varphi_1, \varphi_2)$ of quadratic forms on $W$ is $nonsingular$ if $\varphi_1$ and $\varphi_2$ are linearly independent nondegenerate quadratic forms generating a nonsingular pencil of quadratic forms, which is denoted by $\Phi_{(\varphi_1, \varphi_2)}\subset\BP(\Sym^2(W^*))$. For a nonsingular pencil of quadratic forms $\Phi\subset\BP(\Sym^2(W^*))$, we also say the pair $(\varphi_1, \varphi_2)$ is a $good$ $pair$ of $\Phi$ if $\Phi_{(\varphi_1, \varphi_2)}=\Phi$.
\end{definition}

\begin{definition}\label{d.ma}
Let $(\varphi_1, \varphi_2)$ be a nonsingular pair of quadratic forms on $W$. Here we define some notions induced by the pair. Firstly, we fix a standard basis $B=\{\mathbf w_1, \mathbf w_2, \ldots, \mathbf w_n\}$ of $W$ with respect to $(\varphi_1, \varphi_2)$ such that
\begin{center}
$\varphi_1(\sum z_i \mathbf w_i)=\sum \alpha_iz_i^2$ and $\varphi_2(\sum z_i \mathbf w_i)=\sum z_i^2$ 
\end{center}
where $[1:\alpha_1],[1:\alpha_2],\ldots,[1:\alpha_n]\in\BP(\C^2)$ are the $n$ distinct roots of the discriminant $\det(s\varphi_1-t\varphi_2)$. (See Proposition \ref{p.1.6} and Definition \ref{d.standard}.) 
We say a vector $\mathbf u=\sum u_i\mathbf w_i$ in $W$ is $(\varphi_1, \varphi_2)$-$general$ if $u_i\neq 0$ for all $i$ and we say $z\in \BP (W)$ is $(\varphi_1, \varphi_2)$-$general$ if $z=[\mathbf u]\in\BP (W)$ for a $(\varphi_1, \varphi_2)$-general vector $\mathbf u\in W$. In addition, let
\begin{center}
$\alpha_{(\varphi_1,\varphi_2)}:W \to W$
\end{center}
be the linear map sending $\sum z_i\mathbf w_i\in W$ to $\sum \alpha_iz_i\mathbf w_i\in W$. 
\end{definition}  

\begin{remark}\label{p.4.2.1}
The notions defined in Definition \ref{d.ma} depend on neither of the order of $\alpha_i$'s nor the choice of a standard basis $B$. Moreover, for a nonsingular pencil $\Phi\subset\BP(\Sym^2(W^*))$ and its good pairs $(\varphi_1, \varphi_2)$ and $(\varphi_1', \varphi_2')$, a vector $\mathbf u\in W$ is $(\varphi_1, \varphi_2)$-general if and only if $\mathbf u$ is $(\varphi_1', \varphi_2')$-general. 
\end{remark}

\begin{definition}\label{phigeneral}
For a nonsingular pencil $\Phi\subset\BP(\Sym^2(W^*))$, we say $z\in\BP(W)$ is $\Phi$-$general$ if $z$ is $(\varphi_1, \varphi_2)$-general for a good pair $(\varphi_1, \varphi_2)$ of $\Phi$.
\end{definition}

\begin{proposition}
Let $(\varphi_1, \varphi_2)$ be a nonsingular pair of quadratic forms on $W$ and let $\mathbf u\in W$ be a $(\varphi_1, \varphi_2)$-general vector. Then $\mathbf u$, $\alpha_{(\varphi_1,\varphi_2)}(\mathbf u)$, and $(\alpha_{(\varphi_1,\varphi_2)})^2(\mathbf u)$ are linearly independent. 
\end{proposition}
\begin{proof}
Consider the basis $\{\mathbf w_1, \mathbf w_2, \ldots, \mathbf w_n\}$ of $W$ in Definition \ref{d.ma}. Then $\mathbf u=\sum u_i\mathbf w_i$ for some nonzero $u_i$'s in $\C$. It is enough to show $\det(\mathbf M)\neq0$ where $\mathbf M=(\mathbf M_{kl})$ is a $3\times3$ matrix given by $\mathbf M_{kl}=u_k(\alpha_k)^{l-1}$ for $k,l\in\{1,2,3\}$. By simple computing,  
\begin{displaymath}
\det(\mathbf M)=u_1u_2u_3(\alpha_1-\alpha_2)(\alpha_2-\alpha_3)(\alpha_3-\alpha_1) 
\end{displaymath}
and this is nonzero since $\alpha_k$'s are distinct and $u_i$'s are nonzero.
\end{proof}

\begin{notation}\label{d.pu}
Given a nonsingular pair $(\varphi_1, \varphi_2)$ of quadratic forms on $W$ and a $(\varphi_1, \varphi_2)$-general vector $\mathbf u\in W$, we denote by 
\begin{center}
$\sP_{(\varphi_1, \varphi_2)}(\mathbf u)\in\Gr(3,W)$ 
\end{center}
the vector subspace generated by $\mathbf u$, $\alpha_{(\varphi_1, \varphi_2)}(\mathbf u)$, and $(\alpha_{(\varphi_1, \varphi_2)})^2(\mathbf u)$. 
\end{notation}

\begin{proposition}\label{p.tt}
With notations in Notation \ref{d.pu}, $\sP_{(\varphi_1, \varphi_2)}(\mathbf u)=\sP_{(c\varphi_1, c\varphi_2)}(\mathbf u)\in\Gr(3,W)$ for any $c\in \C\setminus\{0\}$.
\end{proposition}

\begin{proof}
Since the discriminant polynomials $\det(s\varphi_1-t\varphi_2)$ and $\det\big(s(c\varphi_1)-t(c\varphi_2)\big)$ in $s$ and $t$ have the same roots in $\BP(\C^2)$, the linear maps $\alpha_{(\varphi_1,\varphi_2)}$ and $\alpha_{(c\varphi_1,c\varphi_2)}$ are same.
\end{proof}

\begin{definition} \label{d.poised} 
For a nonsingular pair $(\varphi_1, \varphi_2)$ of quadratic forms on $W$, we say $P \in \Gr(3,W)$ is \it poised by $(\varphi_1,\varphi_2)$ \rm if 
$P=\sP_{(\varphi_1, \varphi_2)}(\mathbf u)$ 
for some $(\varphi_1, \varphi_2)$-general vector $\mathbf u\in W$. Denote by 
\begin{center}
$S_{(\varphi_1, \varphi_2)}\subset \Gr(3,W)$ 
\end{center}
the set of three-dimensional subspaces poised by $(\varphi_1, \varphi_2)$. 
\end{definition}

\begin{proposition} \label{lm1}
Let $(\varphi_1, \varphi_2)$ be a nonsingular pair of quadratic forms on $W$ and let $\mathbf u,\mathbf u'\in W$ be two $(\varphi_1, \varphi_2)$-general vectors. When $\dim (W)=n\geq4$, $\sP_{(\varphi_1, \varphi_2)}(\mathbf u)=\sP_{(\varphi_1, \varphi_2)}(\mathbf u')$ if and only if $[\mathbf u]=[\mathbf u'] \in \BP(W)$. 
\end{proposition}
\begin{proof}
With notations in Definition \ref{d.ma}, $\mathbf u=\sum u_i \mathbf w_i$ and $\mathbf u'=\sum u_i'\mathbf w_i$ for some nonzero $u_i,u_i'\in \C$. Let $\alpha:=\alpha_{(\varphi_1,\varphi_2)}$ for simplicity. Let us consider an $n\times n$ matrix $\mathbf M=(\mathbf M_{ij})$ with $\mathbf M_{ij}=\alpha_j^{i-1}u_j$, which is close to an $n\times n$ Vandermonde matrix $\mathbf M'=(\mathbf M'_{ij})$ with $\mathbf M'_{ij}=\alpha_j^{i-1}$. Then 
\begin{center}
$\det(\mathbf M)=u_1u_2\cdots u_n\det(\mathbf M')$ 
\end{center}
and the determinant of the Vandermonde matrix $\mathbf M'$ is 
\begin{displaymath}
\prod_{1\leq i<j\leq n}(\alpha_i-\alpha_j) 
\end{displaymath}
up to sign. So $\mathbf M$ is invertible since $\alpha_i$'s are distinct and $u_i$'s are nonzero. Hence, the vectors $\mathbf u$, $\alpha(\mathbf u)$, $\alpha^2(\mathbf u)$, $\cdots$, $\alpha^{n-1}(\mathbf u)$ are linearly independent.
If $\mathbf u'$ is contained in $\sP_{(\varphi_1, \varphi_2)}(\mathbf u)$, then
\begin{center}
$\mathbf u'=c_0\mathbf u+c_1\alpha(\mathbf u)+c_2\alpha^2(\mathbf u) \in W$ 
\end{center}
for some $c_0, c_1, c_2\in\C$. Note that $\mathbf u$, $\alpha(\mathbf u)$, $\alpha^2(\mathbf u)$, and $\alpha^3(\mathbf u)$ must be linearly independent for $n\geq 4$. Thus, to make both of $\alpha(\mathbf u')=c_0\alpha(\mathbf u)+c_1\alpha^2(\mathbf u)+c_2\alpha^3(\mathbf u)$ and $\alpha^2(\mathbf u')=c_0\alpha^2(\mathbf u)+c_1\alpha^3(\mathbf u)+c_2\alpha^4(\mathbf u)$ belong to $\sP_{(\varphi_1,\varphi_2)}(\mathbf u)$, the coefficients $c_2$ and $c_1$ must be zero.
\end{proof}

Proposition \ref{lm1} says that the set $S_{(\varphi_1,\varphi_2)}$ is birational to $\BP(W)$ if dim($W)\geq4$.

\begin{notation}\label{n13}
Let $W$ be a complex vector space of dimension $n\geq4$. Given a nonsingular pair $(\varphi_1,\varphi_2)$ of quadratic forms on $W$,  
denote by 
\begin{center}
$v_{(\varphi_1,\varphi_2)}:S_{(\varphi_1,\varphi_2)}\to\BP(W)$ 
\end{center}
the birational morphism sending $\sP_{(\varphi_1,\varphi_2)}(\mathbf u)\in S_{(\varphi_1,\varphi_2)}$ to $[\mathbf u]\in\BP(W)$ where $\mathbf u\in W$ is a $(\varphi_1, \varphi_2)$-general vector. 
\end{notation}

\begin{proposition}\label{p.4.4}
Let $(\varphi_1,\varphi_2)$ and $(\varphi_1',\varphi_2')$ be two nonsingular pairs of quadratic forms on $W$ satisfying 
$\Phi_{(\varphi_1, \varphi_2)}=\Phi_{(\varphi_1', \varphi_2')}\in\Gr(1,\BP(\Sym^2(W^*)))$. 
Then  
$S_{(\varphi_1, \varphi_2)}=S_{(\varphi_1', \varphi_2')}$ in $\Gr(3,W)$.
\end{proposition}

\begin{proof}
It is enough to show $S_{(\varphi_1', \varphi_2')}\subset S_{(\varphi_1, \varphi_2)}$. 
Let $\{\mathbf w_1, \mathbf w_2, \ldots, \mathbf w_n\}$ be a standard basis for $(\varphi_1, \varphi_2)$ as in Definition \ref{d.ma}, i.e.
\begin{displaymath}
\varphi_1(\sum z_i \mathbf w_i)=\sum \alpha_iz_i^2 
\end{displaymath}
and 
\begin{displaymath}
\varphi_2(\sum z_i \mathbf w_i)=\sum z_i^2
\end{displaymath}
for distinct $\alpha_i$'s in $\C$. Since two pairs $(\varphi_1, \varphi_2)$ and $(\varphi_1', \varphi_2')$ generate the same pencil of quadrics, $\varphi_1'=a\varphi_1+b\varphi_2$ and $\varphi_2'=c\varphi_1+d\varphi_2$ for some complex numbers $a,b,c,d$ with $ad-bc\neq 0$. Then
\begin{equation}\label{p44.1}
\varphi_1'(\sum z_i \mathbf w_i)=\sum (a\alpha_i+b)z_i^2
\end{equation}
and
\begin{equation}\label{p44.2}
  \varphi_2'(\sum z_i \mathbf w_i)=\sum (c\alpha_i+d)z_i^2.
\end{equation}
The terms $(a\alpha_i+b)$ in (\ref{p44.1}) and $(c\alpha_i+d)$ in (\ref{p44.2}) are nonzero for all $i$ by the nondegeneracy of $\varphi_1'$ and $\varphi_2'$. \\

Let 
$\displaystyle\alpha_i':=\frac{a\alpha_i+b}{c\alpha_i+d}\neq0$. 
Then the roots of $\det(s\varphi_1'-t\varphi_2')$ are exactly 
\begin{displaymath}
[1:\alpha_1'],[1:\alpha_2'],\ldots,[1:\alpha_n']\in\BP(\C^2).
\end{displaymath}
If $\{\mathbf w_1', \mathbf w_2', \ldots, \mathbf w_n'\}$ is a standard basis for $(\varphi_1', \varphi_2')$ such that
\begin{displaymath}
\varphi_1'(\sum z_i' \mathbf w_i')=\sum \alpha_i'(z_i')^2 
\end{displaymath}
and 
\begin{center}
$\displaystyle\varphi_2'(\sum z_i' \mathbf w_i')=\sum (z_i')^2$, 
\end{center}
then $[\mathbf w_i']=[\mathbf w_i]\in\BP (W)$, see Corollary \ref{p.1.6.1}. So
\begin{center}
$\alpha_{(\varphi_1,\varphi_2)}:W\to W$ 
\end{center}
sends $\sum z_i' \mathbf w_i'\in W$ to $\sum \alpha_i z_i' \mathbf w_i'\in W$ while 
\begin{center}
$\alpha_{(\varphi_1',\varphi_2')}:W \to W$ 
\end{center}
sends $\sum z_i' \mathbf w_i'\in W$ to $\sum \alpha_i' z_i' \mathbf w_i'\in W$.
For convenience, let $\alpha':=\alpha_{(\varphi_1',\varphi_2')}$ and $\alpha:=\alpha_{(\varphi_1,\varphi_2)}$. \\

Now take $P\in S_{(\varphi_1', \varphi_2')}$. In other words, $P$ is a three-dimensional vector subspace generated by $\mathbf u, \alpha'(\mathbf u), (\alpha')^2(\mathbf u)\in W$ for some $\mathbf u=\sum u_i \mathbf w_i'\in W$ with nonzero $u_i$'s. Let 
\begin{center}
$\mathbf v:=\sum (c\alpha_i+d)^{-2} u_i \mathbf w_i'\in W$, 
\end{center}
i.e. $\mathbf v=\sum v_i w_i'$ with $v_i=(c\alpha_i+d)^{-2} u_i\neq0$. Then $\mathbf v$ is $(\varphi_1',\varphi_2')$-general and $(\varphi_1,\varphi_2)$-general. We will show  
\begin{equation}\label{sps}
P=\sP_{(\varphi_1, \varphi_2)}(\mathbf v)\in S_{(\varphi_1, \varphi_2)}.
\end{equation}
Note  
\begin{equation}\label{p44.3}
\mathbf u=\sum(c\alpha_i+d)^{2} v_i w_i'=c^2\alpha^2(\mathbf v)+2cd\alpha(\mathbf v)+d^2\mathbf v, 
\end{equation}
\begin{equation}\label{p44.4}
\alpha'(\mathbf u)=\sum(a\alpha_i+b)(c\alpha_i+d) v_i w_i'=ac\alpha^2(\mathbf v)+(ad+bc)\alpha(\mathbf v)+bd\mathbf v,
\end{equation}
and
\begin{equation}\label{p44.5}
(\alpha')^2(\mathbf u)=\sum(a\alpha_i+b)^2 v_i w_i'=a^2\alpha^2(\mathbf v)+2ab\alpha(\mathbf v)+b^2\mathbf v.
\end{equation}
Then (\ref{sps}) is implied by (\ref{p44.3}), (\ref{p44.4}), and (\ref{p44.5}).
\end{proof}

\begin{definition}\label{d.sphi}
For a nonsingular pencil $\Phi\subset\BP(\Sym^2(W^*))$, let 
\begin{center}
$S_{\Phi}:=S_{(\varphi_1,\varphi_2)}\subset\Gr(3,W)$ 
\end{center}
where $(\varphi_1, \varphi_2)$ is a good pair  
of $\Phi$. 
We say $P\in \Gr(3,W)$ is \it poised by \rm $\Phi$ if $P\in S_{\Phi}$.
(Note that $S_{\Phi}$ is well-defined by Proposition \ref{p.4.4}.) 
\end{definition}

\begin{definition}\label{d.refinedm} 
Recall Definition \ref{n.4} and consider the natural PGL($W$)-action on the product space 
\begin{displaymath}
\Gr(1, \BP(\Sym^2(W^*)))\times\Gr(3,W).
\end{displaymath}
For $(\Phi,P)\in\Gr(1, \BP(\Sym^2(W^*)))\times\Gr(3,W)$ and $T\in$ PGL($W$), 
\begin{center}
$T.(\Phi,P):=(T^*(\Phi),\,T^{-1}(P)) \in \Gr(1, \BP(\Sym^2(W^*)))\times\Gr(3,W)$.
\end{center}
When  
pr$_1:\Gr(1, \BP(\Sym^2(W^*)))\times\Gr(3,W) \to \Gr(1, \BP(\Sym^2(W^*)))$  
is the projection on the first component, the inverse image
\begin{center}
$(\textrm{pr}_1)^{-1}\big(\Gr(1, \BP(\Sym^2(W^*)))^{\rm s}\big)\subset\Gr(1, \BP(\Sym^2(W^*)))\times\Gr(3,W)$
\end{center}
is contained in the stable locus of the product space. So it has the orbit map
\begin{center}
$\widetilde{q}:(\textrm{pr}_1)^{-1}\big(\Gr(1, \BP(\Sym^2(W^*)))^{\rm s}\big)\to(\textrm{pr}_1)^{-1}\big(\Gr(1, \BP(\Sym^2(W^*)))^{\rm s}\big)\big/\textrm{PGL}(W)$.
\end{center}
As in Definition \ref{n.4}, the orbit space  
\begin{center}
(pr$_1)^{-1}\big(\Gr(1, \BP(\Sym^2(W^*)))^{\rm s}\big)$\big/PGL($W$) 
\end{center}
is defined independently of the choice of the $n$-dimensional vector space $W$. Let 
\begin{center}
$\widetilde{\sM}^{PQ}_n$:=(pr$_1)^{-1}\big(\Gr(1, \BP(\Sym^2(W^*)))^{\rm s}\big)$\big/PGL($W$).
\end{center}
We call it the \it refined moduli space \rm of $\sM^{PQ}_n$ and denote by $\pi_n$ the forgetful morphism from $\widetilde{\sM}^{PQ}_n$ to $\sM^{PQ}_n$ induced by pr$_1$.
\end{definition}

\begin{proposition}\label{pp5}
Let $\Phi$ be a nonsingular pencil of quadratic forms on $W$ and let $(\varphi_1,\varphi_2)$ be a good pair of $\Phi$. Then 
\begin{equation}\label{e.key}
T^{-1}(\sP_{(\varphi_1,\varphi_2)}(\mathbf u))=\sP_{(T^*(\varphi_1),T^*(\varphi_2))}(T^{-1}(\mathbf u))
\end{equation} 
for any $(\varphi_1,\varphi_2)$-general vector $\mathbf u\in W$ and $T\in$ GL$(W)$. Therefore, 
\begin{equation}\label{e.keyy}
\widetilde{q}(\{\Phi\}\times S_{\Phi})=\widetilde{q}(\{ \Phi' \}\times S_{\Phi'}) 
\end{equation}
if $[\Phi']=[\Phi]\in\sM^{\rm PG}_n$. 
\end{proposition}
\begin{proof}

Note that $(T^*(\varphi_1),T^*(\varphi_2))$ is a good pair of $T^*(\Phi)$ since $T^*(\varphi_1)$ and $T^*(\varphi_1)$ are nondegenerate quadratic forms and they generate $T^*(\Phi)$, which is a nonsingular pencil of quadratic forms on $W$. Then, by $(i)$ of Proposition \ref{p.2}, the polynomials 
$\det(T^*(\varphi_1)-tT^*(\varphi_2))$ and $\det(\varphi_1-t\varphi_2)$ in $t$ have the same roots, say 
$\alpha_1, \alpha_2, \ldots, \alpha_n\in\C$. 
As in Definition \ref{d.ma}, consider a standard basis 
$\{\mathbf w_1, \mathbf w_2, \ldots, \mathbf w_n\}$ of $W$ with respect to $(\varphi_1,\varphi_2)$ such that
\begin{center}
$\varphi_1(\sum z_i \mathbf w_i)=\sum \alpha_iz_i^2\,$ and $\,\varphi_2(\sum z_i \mathbf w_i)=\sum z_i^2$.
\end{center}
Let 
\begin{displaymath}
\mathbf w_i':=T^{-1}(\mathbf w_i) 
\end{displaymath}
for $1\leq i \leq n$. Then $\{\mathbf w_1', \mathbf w_2', \ldots, \mathbf w_n'\}$ is a standard basis with respect to $(T^*(\varphi_1),T^*(\varphi_2))$ such that 

\begin{center} 
$T^*(\varphi_1)(\sum z_i' \mathbf w_i')=\sum \alpha_i(z_i')^2\,$ and $\,T^*(\varphi_2)(\sum z_i' \mathbf w_i')=\sum (z_i')^2$.\\ 
\end{center}

Let $\mathbf u$ be a $(\varphi_1,\varphi_2)$-general vector, i.e. $\mathbf u=\sum u_i\mathbf w_i\in W$ with nonzero $u_i$'s. Then $\sP_{(\varphi_1,\varphi_2)}(\mathbf u)\subset W$ is generated by 
\begin{center}
$\sum u_i\mathbf w_i, \sum \alpha_iu_i\mathbf w_i, \sum \alpha_i^2u_i\mathbf w_i \in W$. 
\end{center}
Hence, $T^{-1}(\sP_{(\varphi_1,\varphi_2)}(\mathbf u))$ is generated by 
$T^{-1}(\mathbf u)=\sum u_i \mathbf w_i'$, $T^{-1}(\sum \alpha_iu_i\mathbf w_i)=\sum \alpha_iu_i\mathbf w_i'$, and $T^{-1}(\sum \alpha_i^2u_i\mathbf w_i)=\sum \alpha_i^2u_i\mathbf w_i'$. 
Therefore, 
\begin{center}
$T^{-1}(\sP_{(\varphi_1,\varphi_2)}(\mathbf u))=\sP_{(T^*(\varphi_1),T^*(\varphi_2))}(T^{-1}(\mathbf u))\subset S_{(T^*(\varphi_1),T^*(\varphi_2))}$.
\end{center}

If $[\Phi]=[\Phi']\in\sM^{\rm PQ}_n$, then $\Phi'=T^*(\Phi)$ 
for some $T\in\textrm{GL}(W)$. Since
\begin{displaymath}
\widetilde{q}(\{\Phi\}\times S_{\Phi})=\widetilde{q}(\{T^*(\Phi)\}\times \{T^{-1}(P)\mid P\in S_{\Phi}\})=\widetilde{q}(\{\Phi'\}\times \{T^{-1}(P)\mid P\in S_{\Phi}\}),
\end{displaymath}
the equality (\ref{e.keyy}) is implied by (\ref{e.key}).
\end{proof}

\begin{proposition}\label{p.1.10}
Let $\Phi$ be a general nonsingular pencil of quadratic forms on $W$ with dim($W)=n\geq4$. 
Consider the decomposition 
$W=W_1\oplus W_2\oplus\cdots\oplus W_n$ 
induced by $\Phi$ and a basis $\{\mathbf w_1', \mathbf w_2', \ldots, \mathbf w_n'\}$ of $W$ such that each $\mathbf w_i'$ generates $W_i\in\Gr(1,W)$. (See Corollary \ref{p.1.6.1}.) Then all projective automorphisms of $\BP(W)$ preserving $\Phi$ are given by orthogonal reflections in the vectors
$\mathbf w_1', \mathbf w_2', \ldots, \mathbf w_n'$. In particular, the number of such automorphisms is $2^{n-1}$.
\end{proposition}

\begin{proof}
Let $T$ be a projective automorphism of $\BP(W)$ that preserves $\Phi$. Such a morphism induces a projective automorphism of $\Phi\cong\BP^1\subset\BP(\Sym^2(W^*))$ that preserves the set of $n$ degenerate elements in $\Phi$. When $n$ points in $\BP^1$ are in general position, only the identity automorphism on $\BP^1$ can preserve the set of them as long as $n>3$. If $\Phi$ is general, then the degenerate elements in $\Phi$ are in general position, and $T$ must fix each element in $\Phi$.   
Thus, $T$ must take each $W_i$ to $W_i$ to fix each degenerate element in $\Phi_Y$. Moreover, $T$ must be given by an orthogonal reflection in the vectors $\mathbf w_1', \mathbf w_2', \ldots, \mathbf w_n'$ to fix every element in $\Phi_Y$. Therefore, the number of such automorphisms is $2^{n-1}$.
\end{proof}

\begin{notation}\label{n13sq}
For a nonsingular pair $(\varphi_1,\varphi_2)$ of quadratic forms on $W$, let $B = \{ \mathbf w_1$, $\mathbf w_2$, $\ldots, \mathbf w_n \}$ be a standard basis of $W$ with respect to $(\varphi_1,\varphi_2)$. We denote by 
\begin{center}
${\rm sq}^B:\BP(W)\to\BP(W)$ 
\end{center}
\smallskip
the finite morphism that sends $[\sum z_i\mathbf w_i]\in\BP(W)$ to $[\sum z_i^2\mathbf w_i]\in\BP(W)$. Note that ${\rm sq}^B$ depends on the choice of $B$.
\end{notation}

\begin{proposition}\label{pp3}
Let $\Phi$ be a general nonsingular pencil of quadratic forms on $W$, dim$(W)=n>3$. Let $(\varphi_1,\varphi_2)$ be a good pair of $\Phi$. 
When the orbit map $\widetilde{q}$ in Definition \ref{d.refinedm} is restricted on the set 
\begin{center}
$\{\Phi\}\times S_{(\varphi_1,\varphi_2)}\subset \{\Phi\}\times\Gr(3,W)$, 
\end{center}
it becomes a finite covering map over its image and there is a unique birational morphism
\begin{center}
$\widetilde v_{(\varphi_1,\varphi_2)}^B:\widetilde{q}(\{\Phi\}\times S_{\Phi})\to\BP(W)$
\end{center}
that makes the following diagram commutative for a standard basis $B=\{\mathbf w_1,\mathbf w_2,\ldots,\mathbf w_n\}$ of $W$ with respect to $(\varphi_1,\varphi_2)$: 
\begin{displaymath}
\begin{array}{cccc}
\{\Phi\}\times S_{\Phi}(\cong S_{\Phi})& \xrightarrow{\quad\widetilde{q}|_{\{\Phi\}\times S_{\Phi}}\quad} 
& \widetilde{q}(\{\Phi\}\times S_{\Phi})&\subset\,\,\widetilde\sM^{\rm PQ}_n\\
\quad v_{(\varphi_1,\varphi_2)}\Bigg\downarrow & & \qquad\quad\Bigg\downarrow  \widetilde v_{(\varphi_1,\varphi_2)}^B\\
\qquad\qquad\BP(W) & \xrightarrow{\quad\quad\, \rm sq^{\it B} \,\quad\quad} & \BP(W).\\ 
\end{array}
\end{displaymath}
\end{proposition}

\begin{proof}
Let $\sP_{(\varphi_1,\varphi_2)}(\mathbf u)\in S_{\Phi}$ for $u=\sum u_i\mathbf w_i \in W$. If such a morphism $\widetilde v_{(\varphi_1,\varphi_2)}^B$ exists, then $\widetilde v_{(\varphi_1,\varphi_2)}^B$ must send $\sP_{(\varphi_1,\varphi_2)}(\mathbf u)$ to $[\sum u_i^2 \mathbf w_i]\in \BP(W)$. So it is enough to show that $\widetilde v_{(\varphi_1,\varphi_2)}^B$ is well-defined and birational.

\smallskip
If $\Phi$ is general, then Proposition \ref{p.1.10} says that all projective automorphisms of $\BP(W)$ preserving $\Phi$ are given by orthogonal reflections in the vectors $\mathbf w_1, \mathbf w_2, \ldots, \mathbf w_n$ and the number of such automorphisms is $2^{n-1}$. Let $\Lambda\subset$ GL$(W)$ be the set of orthogonal reflections in $\mathbf w_1, \mathbf w_2, \ldots, \mathbf w_n$ and let $T\in\Lambda$. Then $T^*(\varphi_1)=\varphi_1$ and $T^*(\varphi_2)=\varphi_2$. 
By (\ref{e.key}) in Proposition \ref{pp5},
\begin{equation}
T^{-1}(\sP_{(\varphi_1,\varphi_2)}(\mathbf u))=\sP_{(T^*(\varphi_1),T^*(\varphi_2))}(T^{-1}(\mathbf u))=\sP_{(\varphi_1,\varphi_2)}(T^{-1}(\mathbf u)).
\end{equation} 
Moreover, if $T'=cT\in$ GL$(W)$ for some $c\in \C\setminus\{0\}$, i.e. $[T']=[T]\in$ PGL$(W)$, then  
\begin{equation}
(T')^{-1}(\sP_{(\varphi_1,\varphi_2)}(\mathbf u))=\sP_{(c^2\varphi_1,c^2\varphi_2)}(c^{-1}T^{-1}(\mathbf u))=\sP_{(\varphi_1,\varphi_2)}(T^{-1}(\mathbf u))
\end{equation}
by Propositions \ref{p.tt} and \ref{lm1}. Hence,
\begin{equation}\label{eset}
(\widetilde{q}|_{\{\Phi\}\times S_{\Phi}})^{-1}\big([(\Phi,\sP_{(\varphi_1,\varphi_2)}(\mathbf u)]\big)=\big\{\big(\Phi,\sP_{(\varphi_1,\varphi_2)}(T^{-1}(\mathbf u))\big)\mid T\in\Lambda\big\}.
\end{equation}
Since 
\begin{displaymath}
{\rm sq}^B\circ v_{(\varphi_1,\varphi_2)}\big( \sP_{(\varphi_1,\varphi_2)}(T^{-1}(\mathbf u)) \big)={\rm sq}^B\circ v_{(\varphi_1,\varphi_2)}\big( \sP_{(\varphi_1,\varphi_2)}(\mathbf u) \big)\in\BP(W)
\end{displaymath}
  for any $T\in \Lambda$, the morphism $\widetilde v_{(\varphi_1,\varphi_2)}^B$ is well-defined.

\smallskip
On the other hand, the number of elements in 
\begin{displaymath}
\{[T^{-1}(\mathbf u)]\in\BP(W)\mid T\in\Lambda\}\subset\BP(W)
\end{displaymath}
is exactly $2^{n-1}$, and the number of elements in (\ref{eset}) is also $2^{n-1}$ by Proposition \ref{lm1}. Therefore, $\widetilde v_{(\varphi_1,\varphi_2)}^B$ is a birational morphism since ${\rm sq}^B|_{v_{(\varphi_1,\varphi_2)}(S_{\Phi})}$ and $\widetilde{q}|_{\{\Phi\}\times S_{\Phi}}$ are $2^{n-1}$ to 1 morphisms and $v_{(\varphi_1,\varphi_2)}$ is a birational morphism.
\end{proof}

\begin{proposition}\label{p.isom}
Let $W$ and $W'$ be complex vector spaces of dimension $n$. For a good pair $(\varphi_1,\varphi_2)$ of a nonsingular pencil $\Phi\subset\BP(\Sym^2(W^*))$, let $B=\{\mathbf w_1,\mathbf w_2,\ldots,\mathbf w_n\}$ be a standard basis for $(\varphi_1,\varphi_2)$ such that $\varphi_1(\sum z_i\mathbf w_i)=\sum \alpha_iz_i^2$ and $\varphi_2(\sum z_i\mathbf w_i)=\sum z_i^2$ with distinct complex numbers $\alpha_1, \alpha_2, \ldots, \alpha_n$. Then
\begin{equation}\label{e.p.isom1}
\widetilde v_{(\varphi_1,\varphi_2)}^B\Big(\big[\big(\Phi,\sP_{(\varphi_1,\varphi_2)}(\mathbf u)\big)\big]\Big)=\Big[\sum u_i^2 \mathbf w_i\Big]\in\BP(W)
\end{equation} 
for $\mathbf u=\sum u_i \mathbf w_i\in W$ with nonzero $u_i$'s. Suppose that $(\varphi_1',\varphi_2')$ is a good pair of another nonsingular pencil $\Phi'\subset\BP(\Sym^2((W')^*))$ such that $\varphi_1'(\sum z_i'\mathbf w_i')=\sum \alpha_i(z_i')^2$ and $\varphi_2'(\sum z_i'\mathbf w_i')=\sum (z_i')^2$ for a standard basis $B'=\{\mathbf w_1', \mathbf w_2', \ldots, \mathbf w_n'\}$ for $(\varphi_1',\varphi_2')$. Then $[\Phi']=[\Phi]\in\sM^{\rm PQ}_n$ and
\begin{equation}\label{e.p.isom2}
\widetilde v_{(\varphi_1,\varphi_2)}^{B}\Big(\big[\big(\Phi',\sP_{(\varphi_1',\varphi_2')}(\mathbf u')\big)\big]\Big)=\widetilde v_{(\varphi_1,\varphi_2)}^{B}\Big(\big[\big(\Phi,\sP_{(\varphi_1,\varphi_2)}(\mathbf u)\big)\big]\Big)\in\BP(W)
\end{equation} 
where $\mathbf u'=\sum u_i \mathbf w_i'\in W'$. 

\end{proposition}

\begin{proof}
Firstly, (\ref{e.p.isom1}) is given by the definition of $\widetilde v_{(\varphi_1,\varphi_2)}^B$ in Proposition \ref{pp3}. Next, consider the isomorphism $T:W'\to W$ that sends each $\mathbf w_i'$ to $\mathbf w_i$. Then $T^*(\varphi_1)=\varphi_1'$, $T^*(\varphi_2)=\varphi_2'$, and $T^{-1}(\mathbf u)=\mathbf u'$ where $\mathbf u=\sum u_i \mathbf w_i\in W$ and $\mathbf u'=\sum u_i \mathbf w_i'\in W'$. So
\begin{equation}
T^{-1}(\sP_{(\varphi_1,\varphi_2)}(\mathbf u))=\sP_{(T^*(\varphi_1),T^*(\varphi_2))}(T^{-1}(\mathbf u))=\sP_{(\varphi_1',\varphi_2')}(\mathbf u')
\end{equation} 
by (\ref{e.key}) in Proposition \ref{pp5}. Therefore,
\begin{displaymath}
\big[\big(\Phi',\sP_{(\varphi_1',\varphi_2')}(\mathbf u')\big)\big]=\big[\big(T^*(\Phi),T^{-1}(\sP_{(\varphi_1,\varphi_2)}(\mathbf u))\big)\big]=\big[\big(\Phi,\sP_{(\varphi_1,\varphi_2)}(\mathbf u)\big)\big]\in\widetilde\sM^{\rm PQ}_n.
\end{displaymath}

\end{proof}

\begin{corollary}\label{cor4.18}
For a nonsingular pencil of quadrics $\Phi$ on a vector space $W$ of dimension $n>3$, let $S_{\Phi}$ be the set of three-dimensional subspaces poised by $\Phi$. Then there is a correspondence 
$S:=\{([\Phi],[\{\Phi\}\times S_{\Phi}])\in\sM^{\rm PQ}_n\times\widetilde\sM^{\rm PQ}_n\mid [\Phi]\in\sM^{\rm PQ}_n\}$ between $\sM^{\rm PQ}_n$ and $\widetilde\sM^{\rm PQ}_n$. Denote by $p_1$ and $p_2$ the natural projections:
 \begin{displaymath}
    \xymatrix{
      &  S \ar[dl]_{p_1} \ar[dr]^{p_2}       &   \\
        \sM^{\rm PQ}_n  &  & \widetilde\sM^{\rm PQ}_n }
\end{displaymath}
Then $p_2$ is an embedding and $p_1$ is a submersion having fibers birational to $\BP^{n-1}$.
\end{corollary}

\section{Fibers of the moduli map of second fundamental forms}

\begin{definition}\label{d.regular}
With Notation \ref{n.7}, we say a point $x=[x_1:x_2:\cdots:x_{n+3}]\in X \subset \BP V$ is $regular$ in $X$ if $x$ satisfies the following conditions.\\
($i$) The point $x$ is $\Phi_X$-general, i.e., $x_1,x_2,\ldots,x_{n+3}\in\C\setminus\{0\}$.\\ 
($ii$) $II_{X,x}\subset\BP(\Sym^2(T^*_x(X)))$ is a nonsingular pencil of quadrics on $T_x(X)$.\\ 
($iii$) $II_{X,x}\subset\BP(\Sym^2(T^*_x(X)))$ is general in the sense of Proposition \ref{p.1.10}, i.e., the trivial automorphism of $II_{X,x}\cong\BP^1$ is the only projective automorphism of $II_{X,x}$ that preserves the set of degenerate elements in $II_{X,x}$.\\
($iv$) The induced quadratic forms $\varphi_1|_{T_x(X)},\varphi_2|_{T_x(X)}\in \Sym^2(T^*_x(X))$ are nondegenerate.

\smallskip
Note that all the conditions are satisfied by general points in $X$. The first three conditions are formulated in terms of projective invariants of $X$ and only the last condition depends on the choice of $\varphi_1,\varphi_2\in\widehat\Phi_X$; $(iv)$ is added for the convenience of later work.  
We denote by 
\begin{displaymath}
X^{\rm reg} \subset X
\end{displaymath}
the set of regular points in $X$.

\end{definition}
\begin{proposition}\label{p.3.1}
Let $x=[x_1:x_2:\cdots:x_{n+3}]\in X$ with Notation \ref{n.7}. Then x satisfies the conditions $(i)$, $(ii)$, and $(iv)$ in Definition \ref{d.regular} if and only if the discriminant \rm D\it$(II_{X,x},\varphi_1,\varphi_2)$ of $II_{X,x}$ (defined in Definition \ref{d.theta}) has $n$ distinct roots and they are different from $[0:1],[1:0],[1:\lambda_1], [1:\lambda_2], \ldots, [1:\lambda_{n+3}] \in \BP(\C^2)$. 
\end{proposition}

\begin{proof}
By $(ii)$ and $(iv)$ of Definition \ref{d.regular}, the discriminant D$(II_{X,x},\varphi_1,\varphi_2)$ of $II_{X,x}$ has $n$ distinct roots different from $[0:1],[1:0]\in \BP(\C^2)$. 
Moreover, by $(ii)$ of Proposition \ref{p5}, $x_i$ vanishes if and only if $[1:\lambda_i]$ is a root of D$(II_{X,x},\varphi_1,\varphi_2)$. 
\end{proof}

\begin{proposition}\label{p.fiber}
Let $x\in X^{\rm reg}$. Then the fiber $X^{\mu}_x:=(\mu^X)^{-1}(\mu^X(x))$ of the moduli map $\mu^X:X\to\sM^{\rm PQ}_n$ at $x$ is of pure dimension three and smooth at $x$. In particular, $\Ker(\rd_x\mu^X)\subset T_x(X)$ has dimension three.
\end{proposition}

\begin{proof}
We will use the decomposition $\theta'_{(\varphi_1,\varphi_2)}\circ{\rm sq}|_X$ of $\theta_{(\varphi_1,\varphi_2)}$ in Proposition \ref{p5}.
Recall that the morphism $\theta'_{(\varphi_1,\varphi_2)}$ is an isomorphism between projective spaces of dimension $n$ and ${\rm sq}|_X$ is a finite morphism from $X$ onto a projective space of dimension $n$. By $(i)$ in Definition \ref{d.regular}, $x$ is not a critical point of ${\rm sq}|_X$. Hence, the derivative 
\begin{center}
$\rd_x\theta_{(\varphi_1,\varphi_2)}=\rd_x\theta'_{(\varphi_1,\varphi_2)}\circ\rd_x({\rm sq}|_X)$ 
\end{center}
of $\theta_{(\varphi_1,\varphi_2)}$ at $x$ is an isomorphism between tangent spaces. 

On the other hand, by $(ii)$ in Definition \ref{d.regular}, $\theta_{(\varphi_1,\varphi_2)}(x)\in\BP(\sB_n)$ is stable with respect to the SL(2,$\C$)-action on $\BP(\sB_n)$ (defined in Definition \ref{n.61}), and the orbit 
\begin{center}
$q^{-1}(\mu^X(x))\subset\BP(\sB_n)$ 
\end{center}
is irreducible and three-dimensional. Thus the fiber 
\begin{center}
$(\mu^X)^{-1}(\mu^X(x))=\theta_{(\varphi_1,\varphi_2)}^{-1}(q^{-1}(\mu^X(x)))\subset X$ 
\end{center}
is also of pure dimension three. 
Therefore, since the fiber $(\mu^X)^{-1}(\mu^X(x))$ is of pure dimension three and smooth at $x$, $\Ker(\rd_x\mu^X)$ is a three-dimensional vector subspace in $T_x(X)$.
\end{proof}


\begin{definition}
\label{d.plane}
Let $X\subset\BP^{n+2}$ be a nonsingular intersection of two quadric hypersurfaces with $n\geq3$. Given a regular point $x \in X$, denote by $\sP_x$ the kernel $\Ker(\rd_x\mu^X) \subset T_x(X)$. When we denote by $X^{\mu}_x$  the fiber $(\mu^X)^{-1}(\mu^X(x))$ of $\mu^X$ at $\mu^X(x)$, 
\begin{center}
$\sP_x=\Ker(\rd_x\mu^X)=T_x(X^{\mu}_x) \subset T_x(X)$.
\end{center}

\end{definition}


The goal of this section is to prove that $\sP_x\subset T_x(X)$ is poised by $II_{X,x}$ at every $x\in X^{\rm reg}$. To achieve this, we will describe $\sP_x$ at $x$ explicitly.

\begin{assumption}\label{ass}
 From Lemma \ref{l.2} to Lemma \ref{l.4}, we work with Notation \ref{n.7} and fix a point $x=[x_1:x_2:\cdots:x_{n+3}]\in X$ that satisfies the conditions $(i)$, $(ii)$, and $(iv)$ in Definition \ref{d.regular}; we don't need to assume the condition $(iii)$ here. As in Lemma \ref{p55}, take $W_x=x^{\perp}\cap H$ with its basis $\{\mathbf e_1', \mathbf e_2', \ldots, \mathbf e_n'\}$ and regard the second fundamental form $II_{X,x}$ as the linear system of quadratic forms on $W_x$ generated by $\varphi_1|_{W_x}$ and $\varphi_2|_{W_x}$. Let 
 \begin{center}
 $\mathbf e_{i,x}:=\mathbf e_i'\in W_x\subset V$ 
 \end{center}
 for $1\leq i\leq n$.
  Then the discriminant polynomial
$\det(s\varphi_1|_{W_x}-t\varphi_1|_{W_x})\in\sB_n$
 has $n$ distinct roots 
\begin{center}
 $[1:\alpha_1], [1:\alpha_2], \ldots, [1:\alpha_n] \in \BP(\C^2)$ 
 \end{center}
different from $[1:0],[1:\lambda_1], [1:\lambda_2], \ldots, [1:\lambda_{n+3}] \in \BP(\C^2)$, see Proposition \ref{p.3.1}. Then, by $(ii)$ of Proposition \ref{p5}, we assume
\begin{displaymath}
x_i^2=\frac{(\lambda_i-\alpha_1)(\lambda_i-\alpha_2)\cdots(\lambda_i-\alpha_n)}{(\lambda_i-\lambda_1)(\lambda_i-\lambda_2)\cdots(\lambda_i-\lambda_{\check i})\cdots(\lambda_i-\lambda_{n+3})}
\end{displaymath}
for $1 \leq i \leq n+3$ where the notation $\check{i}$ means that $i$ is excluded from the index set. 
\end{assumption}

In the following lemma, we diagonalize the pair of two quadratic forms $\varphi_1|_{W_x}$ and $\varphi_2|_{W_x}$.
\begin{lemma}\label{l.2}
With Assumption \ref{ass}, let $\{\mathbf e_{1,x}', \mathbf e'_{2,x}, \ldots, \mathbf e'_{n,x}\}$ be a standard basis of $W_x$ with respect to $(\varphi_1|_{W_x}, \varphi_2|_{W_x})$ such that 
\begin{displaymath}
\varphi_1|_{W_x}\big(\sum Z_i\mathbf e_{i,x}'\big)=\alpha_1Z_1^2+\alpha_2Z_2^2+\cdots+\alpha_{n}Z_{n}^2
\end{displaymath}
and 
\begin{displaymath}
\varphi_2|_{W_x}\big(\sum Z_i\mathbf e_{i,x}'\big)=Z_1^2+Z_2^2+\cdots+Z_{n}^2.
\end{displaymath}
Then each $Z_i=Z_i(\sum z_j \mathbf e_{j,x})$ satisfies $Z_i^2=-c_iF_i^2$ where
\begin{displaymath}
c_i=\frac{1}{\prod_{j \neq i}(\alpha_i-\alpha_j)}\cdot\frac{\prod_{k=1}^{n+1}(\alpha_i-\lambda_{k})}{(\alpha_i-\lambda_{n+2})(\alpha_i-\lambda_{n+3})}\in\C
\end{displaymath}
and
\begin{displaymath}
F_i\big(\sum_{j=1}^n z_j \mathbf e_{j,x}\big)=\sum_{j=1}^{n}\Big(\frac{(\lambda_{n+2}-\lambda_j)(\lambda_{n+3}-\lambda_j)}{(\alpha_i-\lambda_j)}x_jz_j\Big).
 \end{displaymath}
\end{lemma}

\begin{proof}
It suffices to show  
\begin{center}
$\varphi_1|_{W_x}=-\alpha_1c_1F_1^2-\alpha_2c_2F_2^2-\cdots-\alpha_nc_nF_n^2$
\end{center}
and
\begin{center}
$\varphi_2|_{W_x}=-c_1F_1^2-c_2F_2^2-\cdots-c_nF_n^2$.
\end{center}
Since $\varphi_1|_{W_x}$ is nondegenerate for $x\in X^{\rm reg}$, we don't need to show the linear independency of $F_i$'s. 
For the case $n=1$, we directly verify $\varphi_1|_{W_x}=-\alpha_1c_1F_1^2$ and $\varphi_2|_{W_x}=-c_1F_1^2$. And higher dimensional cases are proven inductively.
Let $z_i':=x_iz_i$ for simplicity. \\ 

Recall the formulas (\ref{e1}) and (\ref{e2}) in Lemma \ref{p55}. For $n=1$, 
\begin{displaymath}
\varphi_1|_{W_x}\big(\sum x_i^{-1}z_i'\mathbf e_{i,x}\big)=\Big(\lambda_3\frac{(\lambda_{4}-\lambda_1)^2}{(\lambda_{4}-\lambda_{3})^2x_{3}^2}+\lambda_4\frac{(\lambda_{3}-\lambda_1)^2}{(\lambda_{4}-\lambda_{3})^2x_{4}^2}+\lambda_1\frac{1}{x_1^2}\Big)(z_1')^2
\end{displaymath}
and
\begin{displaymath}
\varphi_2|_{W_x}\big(\sum x_i^{-1}z_i'\mathbf e_{i,x}\big)=\Big(\frac{(\lambda_{4}-\lambda_1)^2}{(\lambda_{4}-\lambda_{3})^2x_{3}^2}+\frac{(\lambda_{3}-\lambda_1)^2}{(\lambda_{4}-\lambda_{3})^2x_{4}^2}+\frac{1}{x_1^2}\Big)(z_1')^2.
\end{displaymath}
By Assumption \ref{ass},
\begin{displaymath}
x_3^2=\frac{(\lambda_3-\alpha_1)}{(\lambda_3-\lambda_1)(\lambda_3-\lambda_2)(\lambda_3-\lambda_4)},
\end{displaymath}
\begin{displaymath}
x_4^2=\frac{(\lambda_4-\alpha_1)}{(\lambda_4-\lambda_1)(\lambda_4-\lambda_2)(\lambda_4-\lambda_3)},
\end{displaymath}
and
\begin{displaymath}
x_1^2=\frac{(\lambda_1-\alpha_1)}{(\lambda_1-\lambda_2)(\lambda_1-\lambda_3)(\lambda_1-\lambda_4)}.
\end{displaymath}
Then direct calculation shows $\varphi_2|_{W_x}=-c_1F_1^2$ and $\varphi_1|_{W_x}=-\alpha_1c_1F_1^2$. 
More precisely,
\begin{center}
$\displaystyle \varphi_2|_{W_x}\big(\sum x_i^{-1}z_i'\mathbf e_{i,x}\big)=
\Big[\frac{(\lambda_{3}-\lambda_1)(\lambda_4-\lambda_1)}{(\lambda_4-\lambda_3)}\Big\{-\frac{(\lambda_{4}-\lambda_1)(\lambda_3-\lambda_2)}{(\lambda_3-\alpha_1)}+\frac{(\lambda_{3}-\lambda_1)(\lambda_4-\lambda_2)}{(\lambda_4-\alpha_1)}\Big\}$\\
$\displaystyle+\frac{(\lambda_1-\lambda_2)(\lambda_1-\lambda_3)(\lambda_1-\lambda_4)}{(\lambda_1-\alpha_1)}\Big](z_1')^2$\\
$\displaystyle =\Big[\frac{(\lambda_{3}-\lambda_1)(\lambda_4-\lambda_1)}{(\lambda_3-\alpha_1)(\lambda_4-\alpha_1)}\Big\{-\lambda_1\lambda_2+\lambda_1\alpha_1-\lambda_2\alpha_1
-\lambda_3\lambda_4+\lambda_2(\lambda_3+\lambda_4)\Big\}$\\
$\displaystyle+\frac{(\lambda_1-\lambda_2)(\lambda_1-\lambda_3)(\lambda_1-\lambda_4)}{(\lambda_1-\alpha_1)}\Big](z_1')^2$\\
$\displaystyle =\frac{(\lambda_{3}-\lambda_1)(\lambda_4-\lambda_1)}{(\lambda_3-\alpha_1)(\lambda_4-\alpha_1)(\lambda_1-\alpha_1)}\Big\{\alpha_1(\lambda_1-\lambda_3)(\lambda_1-\lambda_4)-\lambda_2(\lambda_1-\lambda_3)(\lambda_1-\lambda_4)\Big\}(z_1')^2$\\
$\displaystyle =-\frac{(\alpha_1-\lambda_2)}{(\alpha_1-\lambda_3)(\alpha_1-\lambda_4)}\frac{(\lambda_{3}-\lambda_1)^2(\lambda_4-\lambda_1)^2}{(\alpha_1-\lambda_1)}(z_1')^2=-c_1F_1^2$.
\end{center}
Similarly
\begin{center}
$\displaystyle \varphi_1|_{W_x}\big(\sum x_i^{-1}z_i'\mathbf e_{i,x}\big)=\Big[\frac{(\lambda_{3}-\lambda_1)(\lambda_4-\lambda_1)}{(\lambda_4-\lambda_3)}\Big\{-\lambda_3\frac{(\lambda_{4}-\lambda_1)(\lambda_3-\lambda_2)}{(\lambda_3-\alpha_1)}$
\end{center}
\begin{flushright}
$\displaystyle+\lambda_4\frac{(\lambda_{3}-\lambda_1)(\lambda_4-\lambda_2)}{(\lambda_4-\alpha_1)}\Big\}+\lambda_1\frac{(\lambda_1-\lambda_2)(\lambda_1-\lambda_3)(\lambda_1-\lambda_4)}{(\lambda_1-\alpha_1)}\Big](z_1')^2$
\end{flushright}
\begin{center}
$\displaystyle =\Big[\frac{(\lambda_{3}-\lambda_1)(\lambda_4-\lambda_1)}{(\lambda_3-\alpha_1)(\lambda_4-\alpha_1)}\Big\{-\alpha_1\lambda_1\lambda_2+\lambda_1\alpha_1(\lambda_3+\lambda_4)
-\alpha_1\lambda_3\lambda_4-\lambda_1\lambda_3\lambda_4+\lambda_2\lambda_3\lambda_4)\Big\}$\\
$\displaystyle+\lambda_1\frac{(\lambda_1-\lambda_2)(\lambda_1-\lambda_3)(\lambda_1-\lambda_4)}{(\lambda_1-\alpha_1)}\Big](z_1')^2$\\
$\displaystyle =\frac{(\lambda_{3}-\lambda_1)(\lambda_4-\lambda_1)}{(\lambda_3-\alpha_1)(\lambda_4-\alpha_1)(\lambda_1-\alpha_1)}\Big\{\alpha_1^2(\lambda_1-\lambda_3)(\lambda_1-\lambda_4)-\lambda_2\alpha_1(\lambda_1-\lambda_3)(\lambda_1-\lambda_4)\Big\}(z_1')^2$\\
$\displaystyle =-\alpha_1\frac{(\alpha_1-\lambda_2)}{(\alpha_1-\lambda_3)(\alpha_1-\lambda_4)}\frac{(\lambda_{3}-\lambda_1)^2(\lambda_4-\lambda_1)^2}{(\alpha_1-\lambda_1)}(z_1')^2=-\alpha_1c_1F_1^2$.
\end{center}

For $n\geq 2$, let  
\begin{displaymath}
\varphi:=(s\varphi_1|_{W_x}-t\varphi_2|_{W_x})-\sum_{i=1}^{n}(\alpha_is-t)c_iF_i^2.
\end{displaymath}
Then $\varphi$ is a homogeneous polynomial of degree two in variables $z_1', z_2', \ldots, z_n'$. Let $\mathbf M$ be the symmetric matrix corresponding to $\varphi$, i.e. $\mathbf M_{ij}:=\varphi(x_i^{-1}\mathbf e_{i,x}',x_j^{-1}\mathbf e_{j,x}')$. To prove this lemma, it suffices to show that $\mathbf M$ is identically zero for any $s$ and $t$. \\

From (\ref{e1}) and (\ref{e2}),
\begin{center}
$\mathbf M_{ij}=A_{ij}+B_{ij}+C_{ij}+\sum_{k=1}^{n} D_{ij}^k$
\end{center}
where
\begin{displaymath}
A_{ij}:=(\lambda_{n+2}s-t)\frac{(\lambda_{n+3}-\lambda_i)(\lambda_{n+3}-\lambda_j)}{(\lambda_{n+3}-\lambda_{n+2})^2x_{n+2}^2},
\end{displaymath}
\begin{displaymath}
B_{ij}:=(\lambda_{n+3}s-t)\frac{(\lambda_{n+2}-\lambda_i)(\lambda_{n+2}-\lambda_j)}{(\lambda_{n+3}-\lambda_{n+2})^2x_{n+3}^2},
\end{displaymath}
\begin{center}
$\displaystyle C_{ii}:=(\lambda_is-t)\frac{1}{x_i^2}$ \,\, ($C_{ij}=0$ if $i\neq j$),
\end{center}
and 
\begin{center}
$\displaystyle D_{ij}^k:=(\alpha_ks-t)\cdot c_k\cdot\frac{(\lambda_{n+2}-\lambda_i)(\lambda_{n+3}-\lambda_i)(\lambda_{n+2}-\lambda_j)(\lambda_{n+3}-\lambda_j)}{(\alpha_k-\lambda_i)(\alpha_k-\lambda_j)}$.
\end{center}
To show $\mathbf M=0$, we apply elementary row operations to $\mathbf M$. Let $\mathbf r_i$ be the $i$-th row of $\mathbf M$. 
We firstly remove $D_{ij}^1$ (related to $\alpha_1$) for $i,j\geq2$ by subtracting \begin{displaymath}
\frac{(\lambda_{n+2}-\lambda_i)(\lambda_{n+3}-\lambda_i)(\alpha_1-\lambda_1)}{(\lambda_{n+2}-\lambda_1)(\lambda_{n+3}-\lambda_1)(\alpha_1-\lambda_i)}\mathbf r_1 
\end{displaymath}
from $\mathbf r_i$ for $i\geq2$. Then 
\begin{center}
$\displaystyle \mathbf M_{ij}-\frac{(\lambda_{n+2}-\lambda_i)(\lambda_{n+3}-\lambda_i)(\alpha_1-\lambda_1)}{(\lambda_{n+2}-\lambda_1)(\lambda_{n+3}-\lambda_1)(\alpha_1-\lambda_i)}\mathbf M_{1j}=$\\
\smallskip
$\displaystyle A_{ij}\Big(1-\frac{(\lambda_{n+2}-\lambda_i)(\alpha_1-\lambda_1)}{(\lambda_{n+2}-\lambda_1)(\alpha_1-\lambda_i)}\Big)+
B_{ij}\Big(1-\frac{(\lambda_{n+3}-\lambda_i)(\alpha_1-\lambda_1)}{(\lambda_{n+3}-\lambda_1)(\alpha_1-\lambda_i)}\Big)+ C_{ij}$\\
\smallskip
$\displaystyle +D_{ij}^2\Big(1-\frac{(\alpha_2-\lambda_i)(\alpha_1-\lambda_1)}{(\alpha_2-\lambda_1)(\alpha_1-\lambda_i)}\Big)+\cdots
+D_{ij}^{n}\Big(1-\frac{(\alpha_{n}-\lambda_i)(\alpha_1-\lambda_1)}{(\alpha_{n}-\lambda_1)(\alpha_1-\lambda_i)}\Big)$\\
\smallskip
$\displaystyle =A_{ij}\frac{(\lambda_{n+2}-\alpha_1)(\lambda_1-\lambda_i)}{(\lambda_{n+2}-\lambda_1)(\alpha_1-\lambda_i)}+
B_{ij}\frac{(\lambda_{n+3}-\alpha_1)(\lambda_1-\lambda_i)}{(\lambda_{n+3}-\lambda_1)(\alpha_1-\lambda_i)}+C_{ij}$\\
\smallskip
$\displaystyle +D_{ij}^2\frac{(\alpha_2-\alpha_1)(\lambda_1-\lambda_i)}{(\alpha_2-\lambda_1)(\alpha_1-\lambda_i)}+\cdots
+D_{ij}^{n}\frac{(\alpha_{n}-\alpha_1)(\lambda_1-\lambda_i)}{(\alpha_{n}-\lambda_1)(\alpha_1-\lambda_i)}$\\
\smallskip
$\displaystyle =\frac{(\lambda_1-\lambda_i)}{(\alpha_1-\lambda_i)}\Big(A_{ij}\frac{(\lambda_{n+2}-\alpha_1)}{(\lambda_{n+2}-\lambda_1)}+
B_{ij}\frac{(\lambda_{n+3}-\alpha_1)}{(\lambda_{n+3}-\lambda_1)}+C_{ij}\frac{(\alpha_1-\lambda_i)}{(\lambda_1-\lambda_i)}$\\
\smallskip

$\displaystyle +D_{ij}^2\frac{(\alpha_2-\alpha_1)}{(\alpha_2-\lambda_1)}+\cdots
+D_{ij}^{n}\frac{(\alpha_{n}-\alpha_1)}{(\alpha_{n}-\lambda_1)}\Big)$\\
\end{center}
for $i,j\geq2$. Put 
\begin{center}
$\displaystyle N_{ij}:=A_{ij}\frac{(\lambda_{n+2}-\alpha_1)}{(\lambda_{n+2}-\lambda_1)}+
B_{ij}\frac{(\lambda_{n+3}-\alpha_1)}{(\lambda_{n+3}-\lambda_1)}+C_{ij}\frac{(\alpha_1-\lambda_i)}{(\lambda_1-\lambda_i)}$\\
$\displaystyle +D_{ij}^2\frac{(\alpha_2-\alpha_1)}{(\alpha_2-\lambda_1)}+\cdots
+D_{ij}^{n}\frac{(\alpha_{n}-\alpha_1)}{(\alpha_{n}-\lambda_1)}$ 
\end{center}
for $i,j\geq2$, then
\begin{equation}\label{e15}
\mathbf M_{ij}-\frac{(\lambda_{n+2}-\lambda_i)(\lambda_{n+3}-\lambda_i)(\alpha_1-\lambda_1)}{(\lambda_{n+2}-\lambda_1)(\lambda_{n+3}-\lambda_1)(\alpha_1-\lambda_i)}\mathbf M_{1j}=\frac{(\lambda_1-\lambda_i)}{(\alpha_1-\lambda_i)}N_{ij}.
  \end{equation} 

Consider $X'= X \cap H \subset \BP V$ where $H \subset \BP V$ is the hyperplane defined by $z_1=0$. Then $X' \subset H (\cong \BP^{n+1})$ is also a  nonsingular intersection of two quadrics in $H$, which is defined by 
\begin{displaymath}
\varphi_1|_H=\lambda_2z_2^2+\lambda_3z_3^2+\cdots+\lambda_{n+3}z_{n+3}^2=0
\end{displaymath}
and 
\begin{displaymath}
\varphi_2|_H=z_2^2+z_3^2+\cdots+z_{n+3}^2=0. 
\end{displaymath}
Let $[x_2':x_3':\cdots:x_{n+3}'] \in X'$ be a point in $X'$ such that 
\begin{displaymath}
x'\in(\theta_{\varphi_1|_H,\varphi_1|_H})^{-1}([(\alpha_2s-t)(\alpha_3s-t)\cdots(\alpha_ns-t)]).
\end{displaymath}
Then $x'$ satisfies the conditions $(i)$, $(ii)$, and $(iv)$ in Definition \ref{d.regular} by Proposition \ref{p.3.1}. So we can apply induction hypothesis to $x'\in X'$. For $i,j\geq2$,
\begin{center}
$\displaystyle A_{ij}\frac{(\lambda_{n+2}-\alpha_1)}{(\lambda_{n+2}-\lambda_1)}
=(\lambda_{n+2}s-t)\frac{(\lambda_{n+3}-\lambda_i)(\lambda_{n+3}-\lambda_j)}{(\lambda_{n+3}-\lambda_{n+2})^2}\cdot\frac{1}{x_{n+2}^2}\cdot\frac{(\lambda_{n+2}-\alpha_1)}{(\lambda_{n+2}-\lambda_1)}$\\
$\displaystyle=(\lambda_{n+2}s-t)\frac{(\lambda_{n+3}-\lambda_i)(\lambda_{n+3}-\lambda_j)}{(\lambda_{n+3}-\lambda_{n+2})^2}\cdot\frac{1}{(x_{n+2}')^2}$
\end{center}
and this is $A_{ij}'$ of $X'$. Similarly 
\begin{center}
$N_{ij}=A_{ij}'+B_{ij}'+C_{ij}'+(D_{ij}^{2})'+\cdots+(D_{ij}^{n-3})'=:\mathbf M_{ij}'$ 
\end{center}
for $i,j\geq2$ and $\mathbf M_{ij}'$ is zero by induction hypothesis. Hence, by (\ref{e15}) above,
\begin{center}
$\displaystyle \mathbf M_{ij}-\frac{(\alpha_1-\lambda_1)}{(\alpha_1-\lambda_i)}\cdot\frac{(\lambda_{n+2}-\lambda_i)(\lambda_{n+3}-\lambda_i)}{(\lambda_{n+2}-\lambda_1)(\lambda_{n+3}-\lambda_1)}\cdot\mathbf M_{1j}=0$,
\end{center}
that is, 
\begin{equation}\label{e23}
\frac{(\alpha_1-\lambda_i)}{(\alpha_1-\lambda_1)}\cdot \mathbf M_{ij}=\frac{(\lambda_{n+2}-\lambda_i)(\lambda_{n+3}-\lambda_i)}{(\lambda_{n+2}-\lambda_1)(\lambda_{n+3}-\lambda_1)}\cdot\mathbf M_{1j}.
\end{equation}
If this process was to remove $D_{ij}^2$ instead of $D_{ij}^1$ (by subtracting constant multiple of the first row $\mathbf r_1$ of $\mathbf M$ from other rows) at the beginning of this process, we obtain
\begin{equation}\label{e24}
 \frac{(\alpha_2-\lambda_i)}{(\alpha_2-\lambda_1)}\cdot \mathbf M_{ij}=\frac{(\lambda_{n+2}-\lambda_i)(\lambda_{n+3}-\lambda_i)}{(\lambda_{n+2}-\lambda_1)(\lambda_{n+3}-\lambda_1)}\cdot\mathbf M_{1j}
\end{equation}
instead of (\ref{e23}). 
Since
\begin{displaymath}
\frac{(\alpha_1-\lambda_i)}{(\alpha_1-\lambda_1)}-\frac{(\alpha_2-\lambda_i)}{(\alpha_2-\lambda_1)}=\frac{(\alpha_1-\alpha_2)(\lambda_i-\lambda_1)}{(\alpha_1-\lambda_1)(\alpha_2-\lambda_1)} \neq 0
\end{displaymath}
for $i\geq 2$, $\mathbf M_{ij}$ vanishes for $i\geq1$ and $j\geq2$ by (\ref{e23}) and (\ref{e24}). Furthermore, if we use another row (instead of the first row) to remove $D_{ij}^1$ or $D_{ij}^2$, the first column of $\mathbf M$ have to vanish, too. 
\end{proof}

From now on, we will compute generators for $\sP_x=\Ker({\rm d}_x\mu^X)\subset W_x$. 
We firstly do this with the basis $\{\mathbf e_{1,x},\mathbf e_{2,x},\ldots,\mathbf e_{n,x}\}$ of $W_x$ defined in Assumption \ref{ass} and later we will consider a standard basis $\{\mathbf e_{1,x}',\mathbf e_{2,x}',\ldots,\mathbf e_{n,x}'\}$ for $(\varphi_1|_{W_x},\varphi_2|_{W_x})$ as in Lemma \ref{l.2}.
\begin{lemma}
\label{l.3}
With the basis $\{\mathbf e_{1,x},\mathbf e_{2,x},\ldots,\mathbf e_{n,x}\}$ of $W_x$ as in Assumption \ref{ass}, $\sP_x\subset W_x$ is generated by three vectors $\mathbf u^{(0)}=\sum_{i=1}^n u_i^{(0)} \mathbf e_{i,x}$, $\mathbf u^{(1)}=\sum_{i=1}^n u^{(1)}_i \mathbf e_{i,x}$, and $\mathbf u^{(2)}=\sum_{i=1}^n u^{(2)}_i \mathbf e_{i,x}$ in $W_x$ where $\displaystyle u^{(l)}_i=\sum_{j=1}^{n}\frac{(\lambda_{n+1}-\lambda_i)}{(\lambda_i-\alpha_j)(\lambda_{n+1}-\alpha_j)}(\alpha_j)^{l}x_i$ for $i\in\{1,2,\cdots,n\}$, $l\in\{0,1,2\}$.
\end{lemma}
\begin{proof}
The three vectors for $\sP_x$ are computed infinitesimally from the SL$(2,\C)$-orbit of  the discriminant 
\begin{center}
D$(II_{X,x},\varphi_1,\varphi_2)=[(\alpha_1s-t)(\alpha_2s-t)\cdots(\alpha_ns-t)] \in \BP (\sB_n)$.  
\end{center}
(See Definition \ref{n.61} and Definition \ref{d.theta}.) From $(ii)$ of Proposition \ref{p.2}, an element 
\begin{displaymath}
\left[
\begin{array}{cc}
a&b\\
c&d
\end{array}\right] \in {\rm SL}(2,\C)
\end{displaymath}
sends each root $[1:\alpha_i]\in\BP(\C^2)$ to 
\begin{center}
$\displaystyle\Big[1:\frac{a\alpha_i+b}{c\alpha_i+d}\Big]\in\BP(\C^2)$.
\end{center}
Let
\begin{displaymath}
S:=\Big\{\left[
\begin{array}{cc}
1&p\\
0&1
\end{array}\right],\,
\left[
\begin{array}{cc}
q&0\\
0&q^{-1}
\end{array}\right],\,
\left[
\begin{array}{cc}
1&0\\
r&1
\end{array}\right] \in {\rm SL}(2,\C)\mid p,r\in\C, q\in\C\setminus\{0\} \Big\}.
\end{displaymath}
Then ${\rm SL}(2,\C)$ is generated by S. Note that 
\begin{displaymath}
\left[
\begin{array}{cc}
a&b\\
c&d
\end{array}\right]=
\left[
\begin{array}{cc}
d^{-1}&0\\
0&d
\end{array}\right]
\left[
\begin{array}{cc}
1&bd\\
0&1
\end{array}\right]
\left[
\begin{array}{cc}
1&0\\
d^{-1}c&1
\end{array}\right]\in {\rm SL}(2,\C) 
\end{displaymath}
if $d\neq 0$. If $d=0$ and $bc=-1$, then
\begin{displaymath}
\left[
\begin{array}{cc}
a&b\\
c&0
\end{array}\right]=
\left[
\begin{array}{cc}
-b&0\\
0&c
\end{array}\right]
\left[
\begin{array}{cc}
1&-\frac{a+b}{b}\\
0&1
\end{array}\right]
\left[
\begin{array}{cc}
1&0\\
1&1
\end{array}\right]
\left[
\begin{array}{cc}
1&-1\\
0&1
\end{array}\right]\in {\rm SL}(2,\C). 
\end{displaymath}
So there are three natural infinitesimal generators of the SL($2,\C$)-action:
\begin{center}
$\alpha_i \to \alpha_i+\epsilon_0=:\alpha_i^{(0)}$,\\ 
$\alpha_i \to \alpha_i+\alpha_i\epsilon_1=:\alpha_i^{(1)}$,
\end{center}
and 
\begin{center}
$\displaystyle \alpha_i \to \frac{\alpha_i}{1-\alpha_i\epsilon_2}\approx\alpha_i+\alpha_i^2\epsilon_2=:\alpha_i^{(2)}$.
\end{center}
These generators define three linearly independent vectors in $W_x$ and they generate $\mathcal{P}_x$.

\smallskip
To be more precise, we firstly compute the infinitesimal difference 
\begin{center}
$(x_1^{(l)}-x_1, x_2^{(l)}-x_2,\ldots,x_{n+3}^{(l)}-x_{n+3}) \in V$ 
\end{center}
for $l\in\{0,1,2\}$ where  
\begin{displaymath}
(x_i)^2=\frac{(\lambda_i-\alpha_1)(\lambda_i-\alpha_2)\cdots(\lambda_i-\alpha_n)}{(\lambda_i-\lambda_1)(\lambda_i-\lambda_2)\cdots(\lambda_i-\lambda_{\check i})\cdots(\lambda_i-\lambda_{n+3})}
\end{displaymath}
and
\begin{displaymath}
(x_i^{(l)})^2=\frac{(\lambda_i-\alpha_1^{(l)})(\lambda_i-\alpha_2^{(l)})\cdots(\lambda_i-\alpha_n^{(l)})}{(\lambda_i-\lambda_1)(\lambda_i-\lambda_2)\cdots(\lambda_i-\lambda_{\check i})\cdots(\lambda_i-\lambda_{n+3})}.
\end{displaymath}
In case of $l=0$,
\begin{center}
$\displaystyle (x_i^{(0)})^2=\frac{(\lambda_i-\alpha_1-\epsilon_0)(\lambda_i-\alpha_2-\epsilon_0)\cdots(\lambda_i-\alpha_{n}-\epsilon_0)}{(\lambda_i-\lambda_1)(\lambda_i-\lambda_2)\cdots(\lambda_i-\lambda_{\check i})\cdots(\lambda_i-\lambda_{n+3})}$
\end{center}
\begin{center}
$\displaystyle\qquad=(x_i)^2\cdot\big((1-\frac{\epsilon_0}{\lambda_i-\alpha_1})(1-\frac{\epsilon_0}{\lambda_i-\alpha_2})\cdots(1-\frac{\epsilon_0}{\lambda_i-\alpha_{n}})\big)$
\end{center}
\begin{center}
$\displaystyle\approx(x_i)^2\cdot\big(1-\epsilon_0\sum_{l=1}^{n}\frac{1}{\lambda_i-\alpha_l}\big).\qquad$
\end{center}
So
\begin{center}
$\displaystyle x_i^{(0)}-x_i=\frac{(x_i^{(0)})^2-(x_i)^2}{x_i^{(0)}+x_i}\approx\frac{x_i}{2}\cdot\frac{(x_i^{(0)})^2-(x_i)^2}{(x_i)^2}
\approx-\frac{1}{2}\big(x_i\sum_{l=1}^{n}\frac{1}{\lambda_i-\alpha_l}\big)\epsilon_1$
\end{center}
and 
\begin{center}
$\displaystyle \Delta^{(0)}:=\big(x_1\sum_{l=1}^{n}\frac{1}{\lambda_1-\alpha_l},x_2\sum_{l=1}^{n}\frac{1}{\lambda_2-\alpha_l},\ldots,x_{n+3}\sum_{l=1}^{n}\frac{1}{\lambda_n-\alpha_l}\big) \in x^{\perp} \subset V$.
\end{center}
If we project $\Delta^{(0)} \in x^{\perp}$ onto $W_x(=x^{\perp}\cap H)$, we obtain the vector $\mathbf u^{(0)} \in W_x$. Other two vectors $\mathbf u^{(1)}$ and $\mathbf u^{(2)}$ are also obtained similarly.
\end{proof}

\begin{lemma}\label{l.4}
With a standard basis $\{\mathbf e_{1,x}',\mathbf e_{2,x}',\ldots,\mathbf e_{n,x}'\}$ for $(\varphi_1|_{W_x},\varphi_2|_{W_x}) $ as in Lemma \ref{l.2}, each vector $\mathbf u^{(l)}=\sum_{i=1}^n u^{(l)}_i \mathbf e_{i,x}$ in Lemma \ref{l.3} is expressed in $\sum_{i=1}^n v^{(l)}_i \mathbf e_{i,x}'$ with $v^{(l)}_i=Z_i(\mathbf u^{(l)})\in\C$. Then
$v_i^{(1)}=\alpha_i v_i^{(0)}$, $v_i^{(2)}=\alpha_i^2 v_i^{(0)}$, and
\begin{center}
$\displaystyle (v_i^{(0)})^2=-\frac{(\alpha_i-\alpha_1)(\alpha_i-\alpha_2)\cdots(\alpha_i-\alpha_{\check i})\cdots(\alpha_i-\alpha_{n})}{(\alpha_i-\lambda_1)(\alpha_i-\lambda_2)\cdots(\alpha_i-\lambda_{n+3})}$
\end{center}
for $i\in\{1,2,\cdots,n\}$. Therefore, $\sP_x=\sP_{(\varphi_1|_{W_x},\varphi_2|_{W_x})}(\mathbf u^{(0)})\subset W_x$.
\end{lemma}

\begin{proof}
By Lemma \ref{l.2}, 
\begin{displaymath}
(v_i^{(l)})^2=\big(Z_i(\mathbf u^{(l)})\big)^2=c_i\big(F_i(\mathbf u^{(l)})\big)^2
\end{displaymath}
where
\begin{displaymath}
c_i=\frac{1}{\prod_{j \neq i}(\alpha_i-\alpha_j)}\cdot\frac{\prod_{k=1}^{n+1}(\alpha_i-\lambda_{k})}{(\alpha_i-\lambda_{n+2})(\alpha_i-\lambda_{n+3})}\in\C.
\end{displaymath}
So it suffices to show
\begin{equation}\label{e110} 
F_i(\mathbf u^{(l)})=\frac{(\alpha_i-\alpha_1)(\alpha_i-\alpha_2)\cdots(\alpha_i-\alpha_{\check i})\cdots(\alpha_i-\alpha_{n})}{(\alpha_i-\lambda_1)(\alpha_i-\lambda_2)\cdots(\alpha_i-\lambda_{n+1})}\cdot(\alpha_i)^{l}.
\end{equation}
Firstly, recall the definitions
\begin{displaymath} 
F_i\big(\sum_{j=1}^n z_j \mathbf e_{j,x}\big)=-\sum_{j=1}^{n}\frac{(\lambda_{n+2}-\lambda_j)(\lambda_{n+3}-\lambda_j)}{(\lambda_j-\alpha_i)}x_jz_j
\end{displaymath}
and 
\begin{displaymath}
 z_j(\mathbf u^{(l)})=u^{(l)}_j=x_j\sum_{k=1}^{n}\frac{(\lambda_{n+1}-\lambda_j)(\alpha_k)^{l}}{(\lambda_j-\alpha_k)(\lambda_{n+1}-\alpha_k)}
\end{displaymath}
in Lemma \ref{l.2} and \ref{l.3}. For simplicity, let $i=1$. Then
\begin{center}
$\displaystyle F_1(\mathbf u^{(l)})=-\sum_{j=4}^{n}\Big(\frac{(\lambda_{n+2}-\lambda_j)(\lambda_{n+3}-\lambda_j)}{(\lambda_j-\alpha_1)}x_j\cdot z_j(\mathbf u^{(l)})\Big)\qquad\qquad$\\
$\displaystyle\qquad\qquad=-\sum_{j=1}^{n}\Big(\frac{(\lambda_{n+2}-\lambda_j)(\lambda_{n+3}-\lambda_j)}{(\lambda_j-\alpha_1)}x_j\big(x_j\sum_{k=1}^{n}\frac{(\lambda_{n+1}-\lambda_j)(\alpha_k)^{l}}{(\lambda_j-\alpha_k)(\lambda_{n+1}-\alpha_k)} \big)\Big)$,
\end{center}
so
\begin{equation}\label{e25}
F_1(\mathbf u^{(l)})=-\sum_{k=1}^{n}\Big( (\alpha_k)^{l} \sum_{j=1}^{n}\big(\frac{(\lambda_{n+1}-\lambda_j)(\lambda_{n+2}-\lambda_j)(\lambda_{n+3}-\lambda_j)}{(\lambda_j-\alpha_1)(\lambda_j-\alpha_k)(\lambda_{n+1}-\alpha_k)}x_j^2\big)\Big).
\end{equation}
Let $S_k$ be the sum $\displaystyle\sum_{j=1}^{n}\frac{(\lambda_{n+1}-\lambda_j)(\lambda_{n+2}-\lambda_j)(\lambda_{n+3}-\lambda_j)}{(\lambda_j-\alpha_1)(\lambda_j-\alpha_k)(\lambda_{n+1}-\alpha_k)}x_j^2$ in (\ref{e25}), i.e.
\begin{equation}\label{e26}
F_1(\mathbf u^{(l)})=-\sum_{k=1}^{n}\big( (\alpha_k)^{l} S_k\big).
\end{equation}
We claim that $S_k$ vanishes  $k\neq 1$ (i.e. $k\neq i$) and $S_1$ gives (\ref{e110}). For $k=2$,
\begin{center}
$\displaystyle S_2:=\sum_{j=1}^{n}\frac{(\lambda_{n+1}-\lambda_j)(\lambda_{n+2}-\lambda_j)(\lambda_{n+3}-\lambda_j)}{(\lambda_j-\alpha_1)(\lambda_j-\alpha_2)(\lambda_{n+1}-\alpha_2)}x_j^2\qquad$\\
$=\displaystyle\sum_{j=1}^{n+1}\Big(\frac{(\lambda_{n+1}-\lambda_j)}{(\lambda_{n+1}-\alpha_2)}\cdot\frac{(\lambda_j-\lambda_{n+2})(\lambda_j-\lambda_{n+3})}{(\lambda_j-\alpha_1)(\lambda_j-\alpha_2)}x_j^2\Big)$\\
$=\displaystyle\sum_{j=1}^{n+1}\Big(\frac{(\lambda_{n+1}-\lambda_j)}{(\lambda_{n+1}-\alpha_2)}\cdot\frac{(\lambda_j-\lambda_{n+2})(\lambda_j-\lambda_{n+3})}{(\lambda_j-\alpha_1)(\lambda_j-\alpha_2)}\cdot\frac{(\lambda_j-\alpha_1)(\lambda_j-\alpha_2)\cdots(\lambda_j-\alpha_{n})}{(\lambda_j-\lambda_1)\cdots(\lambda_j-\lambda_{\check j})\cdots(\lambda_j-\lambda_{n+3})}\Big)$\\
$=\displaystyle\sum_{j=1}^{n+1}\Big(\frac{(\lambda_{n+1}-\lambda_j)}{(\lambda_{n+1}-\alpha_2)}\cdot\frac{(\lambda_j-\alpha_3)(\lambda_j-\alpha_4)\cdots(\lambda_j-\alpha_{n})}{(\lambda_j-\lambda_1)(\lambda_j-\lambda_2)\cdots(\lambda_j-\lambda_{\check j})\cdots(\lambda_j-\lambda_{n+1})}\Big)$\\
$=\displaystyle\sum_{j=1}^{n+1}\Big(\big(\frac{\lambda_{n+1}}{(\lambda_{n+1}-\alpha_2)}-\frac{\lambda_j}{(\lambda_{n+1}-\alpha_2)}\big)\cdot\frac{(\lambda_j-\alpha_3)(\lambda_j-\alpha_4)\cdots(\lambda_j-\alpha_{n})}{(\lambda_j-\lambda_1)(\lambda_j-\lambda_2)\cdots(\lambda_j-\lambda_{\check j})\cdots(\lambda_j-\lambda_{n+1})}\Big)$.
\end{center}
So
\begin{equation}\label{e1000}
S_2=\frac{\lambda_{n+1}}{(\lambda_{n+1}-\alpha_2)}S_2'-\frac{1}{(\lambda_{n+1}-\alpha_2)}S_2''
\end{equation}
where
\begin{center}
$S_2':=\displaystyle\sum_{j=1}^{n+1}\frac{(\lambda_j-\alpha_3)(\lambda_j-\alpha_4)\cdots(\lambda_j-\alpha_{n})}{(\lambda_j-\lambda_1)(\lambda_j-\lambda_2)\cdots(\lambda_j-\lambda_{\check j})\cdots(\lambda_j-\lambda_{n+1})}$
\end{center}
and
\begin{center}
$S_2'':=\displaystyle\sum_{j=1}^{n+1}\Big(\lambda_j\cdot\frac{(\lambda_j-\alpha_3)(\lambda_j-\alpha_4)\cdots(\lambda_j-\alpha_{n})}{(\lambda_j-\lambda_1)(\lambda_j-\lambda_2)\cdots(\lambda_j-\lambda_{\check j})\cdots(\lambda_j-\lambda_{n+1})}\Big)$.
\end{center}
Let $n':=n-2$, $\lambda_j':=\lambda_{j}$, and $\alpha_i':=\alpha_{i+2}$ for $1\leq j \leq n'+3$ and $1\leq i\leq n'$. Then $S_2'=0=S_2''$ by Proposition \ref{l.1}, and the sum $S_2$ in (\ref{e1000}) is also zero. Similarly $S_k$ is zero for $k\neq 1$, and (\ref{e26}) turns into 
\begin{equation}\label{e27}
F_1(\mathbf u^{(l)})=-(\alpha_k)^{l} S_1.
\end{equation}

On the other hand,
\begin{center}
$\displaystyle S_1:=\sum_{j=1}^{n}\frac{(\lambda_{n+1}-\lambda_j)(\lambda_{n+2}-\lambda_j)(\lambda_{n+3}-\lambda_j)}{(\lambda_j-\alpha_1)(\lambda_j-\alpha_1)(\lambda_{n+1}-\alpha_1)}x_j^2\qquad$\\
$=\displaystyle\sum_{j=1}^{n+1}\Big(\frac{(\lambda_{n+1}-\lambda_j)}{(\lambda_{n+1}-\alpha_1)}\cdot\frac{(\lambda_j-\lambda_{n+2})(\lambda_j-\lambda_{n+3})}{(\lambda_j-\alpha_1)(\lambda_j-\alpha_1)}x_j^2\Big)$\\
$=\displaystyle\sum_{j=1}^{n+1}\Big(\frac{(\lambda_{n+1}-\lambda_j)}{(\lambda_{n+1}-\alpha_1)}\cdot\frac{(\lambda_j-\lambda_{n+2})(\lambda_j-\lambda_{n+3})}{(\lambda_j-\alpha_1)(\lambda_j-\alpha_1)}\cdot\frac{(\lambda_j-\alpha_1)(\lambda_j-\alpha_2)\cdots(\lambda_j-\alpha_{n})}{(\lambda_j-\lambda_1)\cdots(\lambda_j-\lambda_{\check j})\cdots(\lambda_j-\lambda_{n+3})}\Big)$\\
$=\displaystyle\sum_{j=1}^{n+1}\Big(\frac{(\lambda_{n+1}-\lambda_j)}{(\lambda_{n+1}-\alpha_1)}\cdot\frac{(\lambda_j-\alpha_2)(\lambda_j-\alpha_3)\cdots(\lambda_j-\alpha_{n})}{(\lambda_j-\alpha_1)\cdot(\lambda_j-\lambda_1)\cdots(\lambda_j-\lambda_{\check j})\cdots(\lambda_j-\lambda_{n+1})}\Big)$\\
$=\displaystyle\sum_{j=1}^{n+1}\Big(\big(\frac{\lambda_{n+1}}{(\lambda_{n+1}-\alpha_1)}-\frac{\lambda_j}{(\lambda_{n+1}-\alpha_1)}\big)\cdot\frac{(\lambda_j-\alpha_2)(\lambda_j-\alpha_3)\cdots(\lambda_j-\alpha_{n})}{(\lambda_j-\alpha_1)\cdot(\lambda_j-\lambda_1)\cdots(\lambda_j-\lambda_{\check j})\cdots(\lambda_j-\lambda_{n+1})}\Big)$.
\end{center}
So
\begin{equation}\label{e28}
S_1=\frac{\lambda_{n+1}}{(\lambda_{n+1}-\alpha_1)}S_1'-\frac{1}{(\lambda_{n+1}-\alpha_1)}S_1''
\end{equation}
where
\begin{center}
$S_1':=\displaystyle\sum_{j=1}^{n+1}\frac{(\lambda_j-\alpha_2)(\lambda_j-\alpha_3)\cdots(\lambda_j-\alpha_{n})}{(\lambda_j-\alpha_1)\cdot(\lambda_j-\lambda_1)\cdots(\lambda_j-\lambda_{\check j})\cdots(\lambda_j-\lambda_{n+1})}$
\end{center}
and
\begin{center}
$S_1'':=\displaystyle\sum_{j=1}^{n+1}\Big(\lambda_j\cdot\frac{(\lambda_j-\alpha_2)(\lambda_j-\alpha_3)\cdots(\lambda_j-\alpha_{n})}{(\lambda_j-\alpha_1)\cdot(\lambda_j-\lambda_1)\cdots(\lambda_j-\lambda_{\check j})\cdots(\lambda_j-\lambda_{n+1})}\Big)$.
\end{center}

Let $n'':=n-1$, $\lambda_1'':=\alpha_1$, $\lambda_j'':=\lambda_{j-1}$ for $2\leq j\leq n''+3$, and $\alpha_i'':=\alpha_{i+1}$ for $1\leq i\leq n''$. Then
\begin{center}
$S_1'=\displaystyle\sum_{j=2}^{n'+3}\frac{(\lambda_j''-\alpha_1'')(\lambda_j''-\alpha_2'')\cdots(\lambda_j''-\alpha_{n''}'')}{(\lambda_j''-\lambda_1'')(\lambda_j''-\lambda_2'')\cdots(\lambda_j''-\lambda_{\check j}'')\cdots(\lambda_j''-\lambda_{n''+3}'')}$
\end{center}
and
\begin{center}
$S_1''=\displaystyle\sum_{j=2}^{n'+3}\Big(\lambda_j''\cdot\frac{(\lambda_j''-\alpha_1'')(\lambda_j''-\alpha_2'')\cdots(\lambda_j''-\alpha_{n''}'')}{(\lambda_j''-\lambda_1'')(\lambda_j''-\lambda_2'')\cdots(\lambda_j''-\lambda_{\check j}'')\cdots(\lambda_j''-\lambda_{n''+3}'')}\Big)$.
\end{center}
By Proposition \ref{l.1} again, 
\begin{center}
$S_1'=\displaystyle-\frac{(\lambda_1''-\alpha_1'')(\lambda_1''-\alpha_2'')\cdots(\lambda_1''-\alpha_{n''}'')}{(\lambda_1''-\lambda_2'')(\lambda_1''-\lambda_3'')\cdots(\lambda_1''-\lambda_{n''+3}'')}$
\end{center}
\begin{center}
$=\displaystyle-\frac{(\alpha_1-\alpha_2)(\alpha_1-\alpha_3)\cdots(\alpha_1-\alpha_{n})}{(\alpha_1-\lambda_1)(\alpha_1-\lambda_1)\cdots(\alpha_1-\lambda_{n+1})}$
\end{center}
and
\begin{center}
$S_1''=\displaystyle-\lambda_1''\frac{(\lambda_1''-\alpha_1'')(\lambda_1''-\alpha_2'')\cdots(\lambda_1''-\alpha_{n''}'')}{(\lambda_1''-\lambda_2'')(\lambda_1''-\lambda_3'')\cdots(\lambda_1''-\lambda_{n''+3}'')}\qquad\quad$
\end{center}
\begin{center}
$\qquad=\displaystyle-\alpha_1\frac{(\alpha_1-\alpha_2)(\alpha_1-\alpha_3)\cdots(\alpha_1-\alpha_{n})}{(\alpha_1-\lambda_1)(\alpha_1-\lambda_2)\cdots(\alpha_1-\lambda_{n+1})}\,(\,=\alpha_1S_1')\,$.
\end{center}
Hence, (\ref{e28}) implies
$S_1=\displaystyle\frac{\lambda_{n+1}}{(\lambda_{n+1}-\alpha_1)}S_1'-\frac{1}{(\lambda_{n+1}-\alpha_1)}(\alpha_1S_1')=S_1'$. Then, by (\ref{e26}),   
\begin{center} 
$\displaystyle F_i(\mathbf u^{(l)})=-(\alpha_1)^{l}S_1'=\frac{(\alpha_i-\alpha_1)(\alpha_i-\alpha_2)\cdots(\alpha_i-\alpha_{\check i})\cdots(\alpha_i-\alpha_{n})}{(\alpha_i-\lambda_1)(\alpha_i-\lambda_2)\cdots(\alpha_i-\lambda_{n+1})}\cdot(\alpha_i)^{l}$
\end{center}
for $i=1$. In the same way, we obtain (\ref{e110}) for any $i$.\\
\end{proof}

\begin{theorem}\label{t.71}
Let $X \subset \BP^{n+2}$ be a nonsingular intersection of two quadric hypersurfaces with $n\geq3$. Let $\mu^X: X^o \to \sM^{\rm PQ}_n$ be the moduli map of second fundamental forms on $X$. Then $\sP_x=\Ker(\rd_x\mu^X)\subset T_x(X)$ at $x\in X^{\rm reg}$ is a three-dimensional vector subspace poised by the second fundamental form $II_{X,x}$.  
\end{theorem}

\begin{proof}
Proposition \ref{p.fiber} says that $\sP_x=\Ker(\rd_x\mu^X)\subset T_x(X)$ is a three-dimensional vector subspace at any $x\in X^{\rm reg}$ and Lemma \ref{l.4} implies that $\sP_x$ is poised by $II_{X,x}$.  
\end{proof}


\section{Refined moduli map of second fundamental forms}\label{mainthm}

The final goal of this section is to prove Theorem \ref{t.3}. We define the refined moduli map of second fundamental forms first. 

\begin{definition}\label{d.refmu}
When $X\in \BP^{n+2}$ is an intersection of two quadric hypersurfaces, define the \textit{refined moduli map of second fundamental forms} 
\begin{center}
$\widetilde\mu^X:X^{\rm reg} \to \widetilde{\sM}^{PQ}_n$ 
\end{center}
as a map assigning the isomorphism class of the pair 
\begin{center}
$({II}_{X,x},\sP_x) \in \Gr(1, \BP(\Sym^2(W^*)))\times\Gr(3,W)$ 
\end{center}
for $W=T_x(X)$ to each $x\in X^{\rm reg}$.
\end{definition}

\begin{definition}\label{d.tangentx}
Let $X\subset\BP^{n+2}$ be a nonsingular intersection of two quadric hypersurfaces with $n>3$. Recall Definition \ref{d.plane}. For $x\in X^{\rm reg}$,  we denote by $\sT_x$ the image of 
\begin{center}
$\sP_x=\Ker(\rd_x\mu^X)\subset T_x(X)$ 
\end{center}
through the derivative of $\widetilde\mu^X$ at $x$, i.e.,
\begin{center}
$\sT_x:=\rd_{x}\widetilde\mu^X(\sP_x)\subset T_{\widetilde\mu^X(x)}(\widetilde\sM^{\rm PQ}_n)$.
\end{center} 
\end{definition}

\begin{remark}
Later we will show ${\rm dim}(\sT_x)={\rm dim}(\sP_x)=3$ at general $x\in X$. Then the derivative of $\,\widetilde\mu^X$ at general $x\in X$ is injective, see Theorem \ref{t.inj}. 
\end{remark}

Now we verify a few lemmas to prove Theorem \ref{t.inj}, which is an essential ingredient for the proof of Theorem \ref{theoremB}(=Theorem \ref{t.3}). 

\begin{assumption}\label{asss}
From Lemma \ref{l.7} to Lemma \ref{l.9.1}, we work with Notation \ref{n.7} and assume ${\rm dim}(X)=n>3$. Moreover, we fix a point $x=[x_1:x_2:\cdots:x_{n+3}]\in X^{\rm reg}$.
As in Assumption \ref{ass}, take $W_x=x^{\perp}\cap H$ with its basis $\{\mathbf e_1', \mathbf e_2', \ldots, \mathbf e_n'\}$ and identify $W_x$ with $T_x(X)$. Let 
 \begin{center}
 $\mathbf e_{i,x}:=\mathbf e_i'\in W_x\subset V$ 
 \end{center}
 for $1\leq i\leq n$.
  Then the discriminant polynomial
$\det(s\varphi_1|_{W_x}-t\varphi_1|_{W_x})\in\sB_n$
 has $n$ distinct roots 
\begin{center}
 $[1:\alpha_1], [1:\alpha_2], \ldots, [1:\alpha_n] \in \BP(\C^2)$ 
 \end{center}
different from $[1:0], [1:\lambda_1], [1:\lambda_2], \ldots, [1:\lambda_{n+3}] \in \BP(\C^2)$. 
Moreover we fix a standard basis $B=\{\mathbf e_{1,x}',\mathbf e_{1,x}', \ldots, \mathbf e_{1,x}'\}$ of $W_x$ with respect to the pair $(\varphi_1|_{W_x},\varphi_2|_{W_x})$ such that \begin{displaymath}
\varphi_1|_{W_x}\big(\sum Z_i\mathbf e_{i,x}'\big)=\alpha_1Z_1^2+\alpha_2Z_2^2+\cdots+\alpha_{n}Z_{n}^2
\end{displaymath}
and 
\begin{displaymath}
\varphi_2|_{W_x}\big(\sum Z_i\mathbf e_{i,x}'\big)=Z_1^2+Z_2^2+\cdots+Z_{n}^2.
\end{displaymath}
as in Lemma \ref{l.2}. 
\end{assumption}

\begin{notation}\label{p.5.9}
Recall the birational morphism
\begin{displaymath}
v_{(\varphi_1|_{W_x},\varphi_2|_{W_x})}^{B}:\widetilde{q}(\{II_{X,x}\}\times S_{II_{X,x}})\subset\widetilde\sM^{\rm PQ}_n\to\BP(W_x) 
\end{displaymath}
in Proposition \ref{pp3} for $W=W_x$. When we denote by $X^{\mu}_x$ the fiber $(\mu^X)^{-1}(\mu^X(x))$ of the moduli map $\mu^X:X^o \to \sM^{\rm PQ}_n$, the composition of $v_{(\varphi_1|_{W_x},\varphi_2|_{W_x})}^{B}$ with the restriction of $\,\widetilde\mu^X:X\to\widetilde\sM^{\rm PQ}_n$ on ${X^{\mu}_x\cap X^{\rm reg}}$ gives a morphism 
\begin{center}
$v^{B}_{x}:X^{\mu}_x \cap X^{\rm reg} \to \BP(W_x)$. 
\end{center}
Note that $v^B_x$ depends on the choice of $\varphi_1, \varphi_2\in\widehat\Phi_X$ and the choice of a standard basis $B$. Furthermore, we regard $\sT_x$ in Definition \ref{d.tangentx} as 
\begin{displaymath}
\rd_{x}v^{B}_{x}(\sP_x)=\rd_{x}v^{B}_{x}(T_{x}(X^{\mu}_x))\subset T_{v^{B}_{x}(x)}(\BP (W_x))
\end{displaymath}
since $v_{(\varphi_1|_{W_x},\varphi_2|_{W_x})}^{B}$ is a birational morphism. 
\end{notation}

\begin{lemma}\label{l.7}
Let $n>3$. For $x'\in X^{\mu}_x \cap X^{\rm reg}$, the image $v^B_x(x')$ is expressed as follows. Let $[1:\alpha_1'], [1:\alpha_2'], \ldots, [1:\alpha_n']\in\BP(\C^2)$ be the roots of $\theta_{(\varphi_1,\varphi_2)}(x')$, i.e., 
\begin{displaymath}
\alpha_i'=\frac{a\alpha_i+b}{c\alpha_i+d} \rm{\quad for\,\,some\,\,}
\left[
\begin{array}{cc}
a&b\\
c&d
\end{array}\right]\in\rm{SL}(2,\C).
\end{displaymath}
Put 
$\displaystyle\lambda_j':={\frac{d\lambda_j-b}{-c\lambda_j+a}}$ for $1\leq j\leq n+3$. 
Then $v^B_x(x')=[\sum z^{x'}_i \mathbf e_{i,x}']\in\BP(W_x)$ where 
\begin{equation}\label{e.700}
z^{x'}_i=-\frac{(\alpha_i-\alpha_1)(\alpha_i-\alpha_2)\cdots(\alpha_i-\alpha_{\check i})\cdots(\alpha_i-\alpha_{n})}{(\alpha_i-\lambda_1')(\alpha_i-\lambda_2')\cdots(\alpha_i-\lambda_{n+3}')}\in\C.
\end{equation}
\end{lemma}

\begin{proof}
By Lemma \ref{l.4}, $v^B_x(x)=\big[\sum (v^{(0)}_i)^2 \,\mathbf e_{i,x}'\big]\in\BP(W_x)$ where 
\begin{displaymath}
(v^{(0)}_i)^2=-\frac{(\alpha_i-\alpha_1)(\alpha_i-\alpha_2)\cdots(\alpha_i-\alpha_{\check i})\cdots(\alpha_i-\alpha_{n})}{(\alpha_i-\lambda_1)(\alpha_i-\lambda_2)\cdots(\alpha_i-\lambda_{n+3})}\in\C.
\end{displaymath}
If $x'\in X^{\mu}_x \cap X^{\rm reg}$, then
\begin{equation}\label{equ3}
\theta_{(\varphi_1,\varphi_2)}(x')=\left[
\begin{array}{cc}
a&b\\
c&d
\end{array}\right].\,\theta_{(\varphi_1,\varphi_2)}(x)\in\BP(\sB_n)
\end{equation}
for some $a,b,c,d\in\C$ satisfying $ad-bc=1$. So let $\alpha_i'$ and $\lambda_j'$ be as in the statement.
On the other hand, by Proposition \ref{p.sl2change},
\begin{equation}\label{equ4}
\theta_{(\varphi_1',\varphi_2')}(x')
=\left[
\begin{array}{cc}
d&-b\\
-c&a
\end{array}\right].\,\theta_{(\varphi_1,\varphi_2)}(x')
\in\BP(\sB_n)
\end{equation}
where $\varphi_1'=d\varphi_1-b\varphi_2$ and $\varphi_2'=-c\varphi_1+a\varphi_2$. 
By (\ref{equ4}) and (\ref{equ3}),
\begin{equation}\label{equ1}
\theta_{(\varphi_1',\varphi_2')}(x')
=\theta_{(\varphi_1,\varphi_2)}(x)=\big[\prod_{i=1}^n (\alpha_is-t)\big]\in\BP(\sB_n).
\end{equation}
Moreover, 
\begin{equation}\label{equ2}
\big[\det(s\varphi_1'-t\varphi_2')\big]=\big[\prod_{j=1}^{n+3}(\lambda_j's-t)\big]\in\BP(\sB_{n+3}).
\end{equation}
By (\ref{equ1}) and (\ref{equ2}), if we apply Lemma \ref{l.4} to $x'$ with $(\varphi_1',\varphi_2')$, then 
\begin{center}
$\sP_{x'}=\sP_{(\varphi_1'|_{W_{x'}},\varphi_2'|_{W_{x'}})}(\sum z_i\mathbf w_i')$ 
\end{center}
for $\sum z_i\mathbf w_i'\in W_x'$ such that
\begin{displaymath}
(z_i)^2=-\frac{(\alpha_i-\alpha_1)(\alpha_i-\alpha_2)\cdots(\alpha_i-\alpha_{\check i})\cdots(\alpha_i-\alpha_{n})}{(\alpha_i-\lambda_1')(\alpha_i-\lambda_2')\cdots(\alpha_i-\lambda_{n+3}')}\in\C
\end{displaymath}
where $B'=\{\mathbf w_1', \mathbf w_2', \ldots, \mathbf w_n'\}$ is a standard basis for $(\varphi_1'|_{W_{x'}},\varphi_2'|_{W_{x'}})$. Then, by Proposition \ref{p.isom},  
\begin{displaymath}
v_{(\varphi_1|_{W_x},\varphi_2|_{W_x})}^B:\widetilde{q}(\{II_{X,x}\}\times S_{II_{X,x}})\subset\widetilde\sM^{\rm PQ}_n\to\BP(W_x)
\end{displaymath}
sends $[(II_{X,x'},\sP_{x'})]$ 
to $[\sum (z_i)^2 \mathbf e_{i,x}']=[\sum z^{x'}_i \mathbf e_{i,x}']\in\BP(W_x)$. 
\end{proof}

\begin{notation}\label{l.71}
In Proposition \ref{l.7}, $v^B_x(x)=[\sum z^{x}_i \mathbf e_{i,x}']\in\BP(W_x)$ where 
\begin{displaymath}
z^{x}_i=-\frac{(\alpha_i-\alpha_1)(\alpha_i-\alpha_2)\cdots(\alpha_i-\alpha_{\check i})\cdots(\alpha_i-\alpha_{n})}{(\alpha_i-\lambda_1)(\alpha_i-\lambda_2)\cdots(\alpha_i-\lambda_{n+3})}\in\C.
\end{displaymath}
We denote the vector $\sum z^{x}_i \mathbf e_{i,x}'\in W_x$ by $\mathbf v^x$, i.e.,
\begin{displaymath}
  [\mathbf v^x]=v^B_x(x)\in \BP(W_x).
  \end{displaymath} 
\end{notation}

\begin{lemma}\label{l.6}
When $\mathbf v^x=\sum z^{x}_i \mathbf e_{i,x}'\in W_x$ is as in Notation \ref{l.71}, consider 
three vectors $\mathbf w^{(0)}=\sum w^{(0)}_i \mathbf e_{i,x}'$, $\mathbf w^{(1)}=\sum w^{(1)}_i \mathbf e_{i,x}'$, $\mathbf w^{(2)}=\sum w^{(2)}_i \mathbf e_{i,x}'\in W_x$ with
\begin{equation}\label{e.72}
w^{(l)}_i=v^{x}_i\sum_{j=1}^{n+3}\lambda_j^l(\alpha_i-\lambda_j)^{-1}.
\end{equation}
Then $\sT_x$ corresponds to a vector subspace $\sT'_x \subset W_x$ generated by $\mathbf w^{(0)}$, $\mathbf w^{(1)}$, $\mathbf w^{(2)}$, and $\mathbf v^{x}$. 
\end{lemma}
\begin{proof}
Let $x'\in X^{\mu}_x \cap X^{\rm reg}$.
By Lemma \ref{l.7},  
$v^B_x(x')=[\sum z^{x'}_i \mathbf e_{i,x}']\in \BP(W_x)$ with
\begin{displaymath}
z_i^{x'}=-\frac{(\alpha_i-\alpha_1)(\alpha_i-\alpha_2)\cdots(\alpha_i-\alpha_{\check i})\cdots(\alpha_i-\alpha_{n})}{(\alpha_i-\lambda_1')(\alpha_i-\lambda_2')\cdots(\alpha_i-\lambda_{n+3}')}\in\C
\end{displaymath}
where
\begin{displaymath}
\lambda_i'=\frac{a\lambda_i+b}{c\lambda_i+d} \rm{\quad for\,\,some\,\,}
\left[
\begin{array}{cc}
a&b\\
c&d
\end{array}\right]\in\rm{SL}(2,\C).
\end{displaymath}
Then, as in the proof of Lemma \ref{l.3}, there are natural infinitesimal generators of SL($2,\mathbb{C}$)-action on $\lambda_i$'s: $\lambda_i\rightarrow\lambda_i+\epsilon_0$, $\lambda_i\rightarrow\lambda_i+\lambda_i\epsilon_1$, and $\lambda_i\rightarrow\frac{\lambda_i}{1-\lambda_i\epsilon_3}\approx\lambda_i+\lambda_i^2\epsilon_2$.

\smallskip
In case of $\lambda_i'=\lambda_i+\epsilon_0$, 
\begin{displaymath}
z^{x'}_i=-\frac{(\alpha_i-\alpha_1)(\alpha_i-\alpha_2)\cdots(\alpha_i-\alpha_{\check i})\cdots(\alpha_i-\alpha_{n})}{(\alpha_i-\lambda_1-\epsilon_0)(\alpha_i-\lambda_2-\epsilon_0)\cdots(\alpha_i-\lambda_{n+3}-\epsilon_0)}\approx v^{x}_i\Big(1+\big(\sum_{j=1}^{n+3}\frac{1}{\alpha_{i}-\lambda_j}\big)\epsilon_0\Big).
\end{displaymath}
Then 
\begin{equation}\label{e.71}
z_i^{x'}-z_i^{x}\approx
\Big(v^{x}_i\sum_{j=1}^{n+3}\frac{1}{\alpha_{i}-\lambda_j}\Big)\epsilon_0,
\end{equation}
and (\ref{e.71}) gives (\ref{e.72}) for $l=0$. Other two vectors $\mathbf w^{(1)}$, $\mathbf w^{(2)}$ are also obtained similarly.
\end{proof}

\begin{notation}\label{n.alpha}
Let $\alpha:W_x\to W_x$ be the linear map sending each vector $\sum z_i \mathbf e_{i,x}'\in W_x$ to $\sum \alpha_i z_i \mathbf e_{i,x}'\in W_x$. 
\end{notation}

\begin{lemma}\label{l.8}
In Lemma \ref{l.6}, consider a bigger subspace $\sT_x''\subset W_x$ generated by five vectors $\mathbf w^{(0)}$, $\mathbf w^{(1)}$, $\mathbf w^{(2)}$, $\mathbf v^{x}$, $\alpha(\mathbf v^{x}) \in W_x$. Then $\sT_x''$ is generated by $\mathbf w^{(0)}$, $\alpha(\mathbf w^{(0)})$, $\alpha^2(\mathbf w^{(0)})$, $\mathbf v^{x}$, $\alpha(\mathbf v^{x}) \in W_x$.
\end{lemma}
\begin{proof}
This lemma follows from two relations 
\begin{center}
$\alpha(\mathbf w^{(0)})-\mathbf w^{(1)}=(n+3)\mathbf v^{x}$
and
$\alpha(\mathbf w^{(1)})-\mathbf w^{(2)}=(\sum_{j=1}^{n+3}\lambda_j)\mathbf v^{x}$.
\end{center}
\end{proof}

\begin{notation}\label{tu}
Given a vector $\mathbf u=\sum u_i \mathbf e_{i,x}' \in W_x$, denote by $\sT(\mathbf u)$ the vector subspace in $W_x$ generated by 5 vectors $\mathbf I:=\sum \mathbf e_{i,x}'$, $\alpha(\mathbf I)$, $\mathbf u$, $\alpha(\mathbf u)$, $\alpha^2(\mathbf u)\in W_x$.
\end{notation}

\begin{lemma}\label{l.9}
If $n>4$, then $\sT(\mathbf u)\in\Gr(5,W_x)$ for general $\mathbf u=\sum u_i \mathbf e_{i,x}' \in W_x$.
\end{lemma}
\begin{proof}
Consider the following $5\times5$ matrix
\begin{displaymath}
\left[
\begin{array}{ccccc}
1&1&1&1&1\\
\alpha_1&\alpha_2&\alpha_3&\alpha_4&\alpha_5\\
u_1&u_2&u_3&u_4&u_5\\
u_1\alpha_1&u_2\alpha_2&u_3\alpha_3&u_4\alpha_4&u_5\alpha_5\\
u_1(\alpha_1)^2&u_2(\alpha_2)^2&u_3(\alpha_3)^2&u_4(\alpha_4)^2&u_5(\alpha_5)^2\\
\end{array}\right].
\end{displaymath}
The determinant of this matrix is a nonzero polynomial in $u_i$'s. For example, the determinant polynomial contains the term $u_1u_2u_3$ with nonzero coefficient. So the matrix has rank five for general $\mathbf u\in W_x$. 
\end{proof} 

\begin{notation}\label{n.w}
In Lemma \ref{l.6}, we denote by $\mathbf w^{x}$ the vector $\sum w^{x}_i \mathbf e_{i,x}' \in W_x$ with
\begin{displaymath}
w^x_i=(w_i^{(0)})/(v_i^x)=\sum_{j=1}^{n+3}(\alpha_i-\lambda_j)^{-1}\in\C.
\end{displaymath}
Then the dimensions of $\sT(\mathbf w^{x})$ and $\sT''_x$ are equal since $v_1^x, v_2^x, \ldots, v_n^x$ are all nonzero for $x\in X^{\rm reg}$.
\end{notation}

\begin{lemma}\label{l.9.1}
Let $x$ be a general point in $X^{\rm reg}$ with $n>4$. Then $\sT''_x\in\Gr(5,W_x)$. 
\end{lemma}
\begin{proof}
It is enough to show $\sT(\mathbf w^{x})\in\Gr(5,W_x)$ for general $x\in X^{\rm reg}$. (See Notation \ref{n.w}.)
Let $\sigma_{0}\lambda^{n+3}+\sigma_{1}\lambda^{n+2}+\cdots+\sigma_{n+2}\lambda+\sigma_{n+3}$  be a polynomial in $\lambda$ of which the roots are $\lambda_1, \lambda_2, \ldots, \lambda_{n+3}\in\C$. Then 
\begin{center}
$\displaystyle w^x_i=\sum_{j=1}^{n+3}\frac{1}{\alpha_i-\lambda_j}=\frac{(n+3)\sigma_0\alpha_i^{n+2}+(n+2)\sigma_1\alpha_i^{n+1}+\cdots+\sigma_{n+2}}{\sigma_0\alpha_i^{n+3}+\sigma_1\alpha_i^{n+2}+\cdots+\sigma_{n+2}\alpha_i+\sigma_{n+3}}$.
\end{center}
Since $\lambda_j$'s are distinct complex numbers, each $w^x_i$ is a nonzero rational function in $\alpha_i$. Hence, for general $\alpha_i$'s in $\C$, the vector $\mathbf w_x\in W_x$ is general in the sense of Lemma \ref{l.9}. In other words, for general $x\in X^{\rm reg}$, the vector subspace $\sT(\mathbf w^x)$ has dimension five if $n>4$.
\end{proof}

By the definition of $\sT''_x$, Lemma \ref{l.9.1} gives the following theorem.

\begin{theorem}\label{t.inj}
Let $X \subset \BP^{n+2}$ be a nonsingular intersection of two quadric hypersurfaces with $n>4$. Let $\widetilde{\mu}^X: X^{\rm reg} \to \widetilde{\sM}_n^{\rm PQ}$ be the refined moduli map of second fundamental forms on $X$. Denote by $X^{\rm good}$ the subset $\{x\in X^{\rm reg} \mid \Ker(d_x\widetilde{\mu}^X)=0\}\subset X^{\rm reg}$. Then $X^{\rm good}$ is nonempty. 
\end{theorem}
\begin{proof}
Note that $\mu^X|_{X^{\rm reg}}=\pi_n\circ\widetilde\mu^X:X^{\rm reg}\to\sM^{\rm PQ}_n$ where $\pi_n$ is the forgetful morphism from $\widetilde{\sM}_n^{\rm PQ}$ to ${\sM}_n^{\rm PQ}$. Hence, at $x\in X^{\rm reg}$, 
\begin{center}
$\Ker(\rd_x\widetilde\mu^X)\subset\Ker(\rd_x\mu^X)=\sP_x$, 
\end{center}
and $\Ker(d_x\widetilde{\mu}^X)=0$ if and only if 
${\rm dim}(\rd_x\widetilde\mu^X(\sP_x))={\rm dim}(\sP_x)=3$.
For general $x\in X^{\rm reg}$, the subspace $\sT''_x\subset W_x$ has dimension five by Lemma \ref{l.9.1}. Moreover, by the definitions of $\sT''_x$ and $\sT'_x$ in Lemma \ref{l.6} and \ref{l.8}, the subspace $\sT'_x\subset W_x$ has dimension four and $\sT_x=\rd_x\widetilde\mu^X(\sP_x)$ has dimension three. Therefore, $d_x\widetilde\mu^X$ at general $x\in X$ is injective. 
\end{proof}


Now we are ready to prove our main result, which is Theorem \ref{t.3} in the introduction:

\begin{theorem}\label{theoremB}
Let $X, X'$ be two nonsingular varieties in $\BP^{n+2}$ ($n>4$), each of them defined as an intersection of two quadric hypersurfaces. Let $\widetilde{\mu}^X: X^{\rm reg} \to \widetilde{\sM}_n^{\rm PQ}$ and $\widetilde{\mu}^{X'}: (X')^{\rm reg} \to \widetilde{\sM}_n^{\rm PQ}$ be their refined moduli maps of second fundamental forms.  Suppose there exists a biholomorphic map $f:M \to M'$ between connected Euclidean open subsets $M \subset X^{\rm reg}$ and $M' \subset (X')^{\rm reg}$ such that $\widetilde{\mu}^X|_M = \widetilde{\mu}^{X'}|_{M'} \circ f$. Then $f$ comes from a projective automorphism of $\BP^{n+2}$.
\end{theorem}

 \begin{proof} 
By Theorem 5, the derivative $\rd_x\widetilde{\mu}^X$ at general $x\in X^{\rm reg}$ is injective and 
\begin{center}
$X^{\rm good}=\{x\in X^{\rm reg} \mid \Ker(d_x\widetilde{\mu}^X)=0\}\subset X^{\rm reg}$ 
\end{center}
is a dense open subset in $X^{\rm reg}$. Hence, suppose that the restrictions $\widetilde{\mu}^X|_{M}$ and $\widetilde{\mu}^{X'}|_{M'}$ are injective and $\widetilde{\mu}^X(M) = \widetilde{\mu}^{X'}(M')$.
Then 
\begin{equation}\label{e.mt0}
f=(\widetilde{\mu}^{X'}|_{M'})^{-1}\circ\widetilde{\mu}^X|_{M} : M \to M'.
\end{equation}

Let $V$ be a complex vector space of dimension $n+3$ with a basis $\{\mathbf e_1, \mathbf e_2, \ldots, \mathbf e_{n+3}\}$ and regard $X$ and $X'$ as subvarieties in $\BP V$. Up to projective transformations on $\BP V$, we assume that $X\subset \BP V$ is given as the intersection of two quadrics  $\varphi_1(\sum z_i\mathbf e_i)=\sum \lambda_iz_i^2=0$ and $\varphi_2(\sum z_i\mathbf e_i)=\sum z_i^2=0$ 
with distinct complex numbers $\lambda_1, \lambda_2, \ldots, \lambda_{n+3}$.

\smallskip
By (\ref{e.mt0}), 
\begin{equation}\label{e.mt00}
\mu^X(x)=\mu^{X'}(f(x)) 
\end{equation}
for each $x\in M$ since $\mu^{X}=\pi_n\circ\widetilde\mu^{X}$ and $\mu^{X'}=\pi_n\circ\widetilde\mu^{X'}$ 
where $\pi_n$ is the natural forgetful morphism from $\widetilde{\sM}^{PQ}_n$ to $\sM^{PQ}_n$. (See Definition \ref{d.refinedm}.) 
Then by Proposition \ref{p.sl2change} there exist $\varphi'_{1,x},\varphi'_{2,x}\in\widehat\Phi_{X'}$ such that they are linearly independent and the discriminant map 
\begin{center}
$\theta_{(\varphi'_{1,x},\varphi'_{2,x})}: X' \to \BP(\sB_n)$ 
\end{center}
sends $f(x)\in X'$ to $\theta_{(\varphi_{1},\varphi_{2})}(x)$. So there is a correspondence 
\begin{center}
$C:=\{(x,\varphi'_{1,x},\varphi'_{2,x})\in M\times\widehat\Phi_{X'}\times \widehat\Phi_{X'} \mid \theta_{(\varphi'_{1,x},\varphi'_{2,x})}(f(x))=\theta_{(\varphi_{1},\varphi_{2})}(x) \}  $ 
\end{center}
between $M$ and $E:=\widehat\Phi_{X'}\times \widehat\Phi_{X'}$:
 \begin{displaymath}
    \xymatrix{
      &  C \ar[dl]_{p_1} \ar[dr]^{p_2}       &   \\
        M  &  &  \quad E \cong \C^4. }
\end{displaymath}
When we denote by $p_1$ and $p_2$ the natural projections, Proposition \ref{p.sl2change} implies that $p_1$ is dominant and has fibers of positive dimension since $\theta_{(c\varphi'_{1,x},c\varphi'_{2,x})}=\theta_{(\varphi'_{1,x},\varphi'_{2,x})}$ for every nonzero $c\in\C$. Hence, $C$ has dimension at least $n+1$, and general fibers of $p_2$ have dimension at least $n-3$.

\smallskip
For $\alpha\in\C$, consider a hyperplane $H_{\alpha}\subset\BP(\sB_n)$ consisting of binary forms that vanish at $[1:\alpha]\in\BP(\C^2)$. Since the discriminant map $\theta_{(\varphi_{1},\varphi_{2})}:X\to \BP(\sB_n)$ is dominant by Proposition \ref{p5}, we can take $\alpha_1\in\C$ so that the dimension of $M_{\alpha_1}:=M \cap (\theta_{(\varphi_{1},\varphi_{2})})^{-1}(H_{\alpha_1})$ is $n-1$. Then the inverse image $(p_1)^{-1}(M_{\alpha_1})$ has dimension at least $n$ and general fibers of $p_2|_{(p_1)^{-1}(M_{\alpha_1})}$ have dimension at least $n-4$($>0$ as long as $n>4$). When we denote ${(p_1)^{-1}(M_{\alpha_1})}\subset C$ by ${C_{\alpha_1}}$, there is a pair of linearly independent quadratic forms $\varphi_1',\varphi_2'\in\widehat\Phi_{X'}$ such that the dimension of the fiber $(p_2|_{C_{\alpha_1}})^{-1}((\varphi_1',\varphi_2'))$ is positive. Then its image
\begin{equation}\label{e.mt00}
M_{\alpha_1}^{(\varphi'_{1},\varphi'_{2})}:=p_1\big( (p_2|_{C_{\alpha_1}})^{-1}((\varphi_1',\varphi_2')) \big) = \{x \in M_{\alpha_1} \mid \theta_{(\varphi'_{1},\varphi'_{2})}(f(x))=\theta_{(\varphi_{1},\varphi_{2})}(x)\} 
\end{equation}
through $p_1$ also has positive dimension because $p_1$ is injective on each fiber of $p_2$. Let $x$ be a point in the set. Then
\begin{equation}\label{e.mt1}
\theta_{(\varphi'_{1},\varphi'_{2})}(f(x))=\theta_{(\varphi_{1},\varphi_{2})}(x) 
=[(\alpha_1s-t)(\alpha_2s-t)\cdots(\alpha_ns-t)]\in\BP(\sB_n)
\end{equation}
for some distinct complex numbers $\alpha_2, \alpha_3, \ldots, \alpha_{n}$ different from $\alpha_1$. (Here we need the properties $(ii)$ and $(iv)$ of $X^{\rm reg}$ in Definition \ref{d.regular}.) Since $x$ is in $M$, 
\begin{equation}\label{e.mt6}
{\widetilde{\mu}^X(x)}={\widetilde{\mu}^{X'}(f(x))}\end{equation}
by (\ref{e.mt0}).

\smallskip
On the other hand, let $\{\mathbf e_1', \mathbf e_2', \ldots, \mathbf e_{n+3}'\}$ be another basis of $V$ that is standard with respect to the pair $(\varphi_1',\varphi_2')$, i.e., $X'$ is given as the intersection of two quadrics $\varphi_1'(\sum z_i'\mathbf e_i')=\sum \lambda_i'(z_i')^2=0$ and $\varphi_2(\sum z_i'\mathbf e_i')=\sum (z_i')^2=0$ 
with distinct complex numbers $\lambda_1', \lambda_2', \ldots, \lambda_{n+3}'$. If 
\begin{equation}\label{e.mt3}
\{\lambda_1', \lambda_2', \ldots, \lambda_{n+3}'\}=\{\lambda_1, \lambda_2, \ldots, \lambda_{n+3}\},
\end{equation}
then $X$ and $X'$ must be biregular to each other. 
Let $\sigma_{0}\lambda^{n+3}+\sigma_{1}\lambda^{n+2}+\cdots+\sigma_{n+2}\lambda+\sigma_{n+3}$ (resp. $\sigma_{0}'\lambda^{n+3}+\sigma_{1}'\lambda^{n+2}+\cdots+\sigma_{n+2}'\lambda+\sigma_{n+3}'$) be a polynomial of which roots are $\lambda_1, \lambda_2, \ldots, \lambda_{n+3}$ (resp. $\lambda_1', \lambda_2', \ldots, \lambda_{n+3}'$). Then 
\begin{equation}\label{e.mt4}
[\sigma_0:\sigma_1:\cdots:\sigma_{n+3}] = [\sigma_0':\sigma_1':\cdots:\sigma_{n+3}'] \in\BP (\C^{n+4}) 
\end{equation}
is equivalent to (\ref{e.mt3}). So we will finally show (\ref{e.mt4}) to conclude that $X$ and $X'$ are biregular to each other. 

\smallskip
Recall Proposition \ref{pp3}, Notation \ref{p.5.9}, and Notation \ref{l.71}. 
Then, since $x\in X^{\rm reg}$, (\ref{e.mt1}) and (\ref{e.mt6}) imply 
\begin{equation}\label{e.mt5}
\big[v^{x}_1:v^{x}_2:\cdots:v^{x}_n\big]=\big[v^{f(x)}_1:v^{f(x)}_2:\cdots:v^{f(x)}_n\big]\in\BP(\C^n)
\end{equation}
where
\begin{equation}\label{e.mt8}
v^{x}_i=-\frac{(\alpha_{i}-\alpha_1)(\alpha_{i}-\alpha_2)\cdots(\alpha_{i}-\alpha_{\check i})\cdots(\alpha_{i}-\alpha_n)}{(\alpha_{i}-\lambda_1)(\alpha_{i}-\lambda_2)\cdots(\alpha_{i}-\lambda_{n+3})}
\end{equation}
and
\begin{equation}\label{e.mt9}
v^{f(x)}_i=-\frac{(\alpha_{i}-\alpha_1)(\alpha_{i}-\alpha_2)\cdots(\alpha_{i}-\alpha_{\check i})\cdots(\alpha_{i}-\alpha_n)}{(\alpha_{i}-\lambda_1')(\alpha_{i}-\lambda_2')\cdots(\alpha_{i}-\lambda_{n+3}')}.
\end{equation}
\smallskip
To be more precise, consider the subspace $\sP_{x}=\Ker(\rd_{x}\mu^X)\subset T_{x}(X)$, which is poised by $II_{X,{x}}$ since $x\in X^{\rm reg}$. (See Theorem \ref{t.71}.) With a standard basis $B=\{\mathbf w_{1}^{x},\mathbf w_{2}^{x},\ldots,\mathbf w_{n}^{x}\}$ of 
$T_x(X)$ with respect to the pair $(\varphi_{1}^{x}:=\varphi_1|_{T_x(X)},\varphi_{2}^{x}:=\varphi_2|_{T_x(X)})$ (that is a nonsingular pair of quadratic forms on $T_{x}(X)$ by the properties $(ii)$ and $(iv)$ of $X^{\rm reg}$ in Definition \ref{d.regular}), 
\begin{displaymath}
\sP_{x}=\sP_{(\varphi_1^{x},\varphi_2^{x})}(\mathbf v)\subset T_{x}(X)
\end{displaymath}
for some vector $\mathbf v=\sum v_i\mathbf w_{i}^{x}\in T_{x}{X}$. As mentioned in Proposition \ref{pp3}, for $W=T_{x}(X)$, the morphism 
\begin{displaymath}
{\rm sq}^{B}\circ v_{(\varphi_1^{x},\varphi_2^{x})}:S_{(\varphi_1^{x},\varphi_2^{x})}\to\BP(W)
\end{displaymath}
sends 
\begin{displaymath}
\sP_{x}=\sP_{(\varphi_1^{x},\varphi_2^{x})}(\mathbf v)\in S_{(\varphi_1^{x},\varphi_2^{x})}
\end{displaymath}
to $[\sum v_i^2 \mathbf w_i^{x}]\in\BP(W)$. (To apply Proposition \ref{pp3}, we need the generality of $II_{X,x}$, which is the property $(iii)$ of $X^{\rm reg}$ in Definition \ref{d.regular}). Then Notation \ref{p.5.9} and \ref{l.71} say
\begin{displaymath}
\big[v_1^2:v_2^2:\cdots:v_n^2\big]=\big[v^{x}_1:v^{x}_2:\cdots:v^{x}_n\big]\in\BP(\C^n)
\end{displaymath}
where
\begin{displaymath}
v^{x}_i=-\frac{(\alpha_{i}-\alpha_1)(\alpha_{i}-\alpha_2)\cdots(\alpha_{i}-\alpha_{\check i})\cdots(\alpha_{i}-\alpha_n)}{(\alpha_{i}-\lambda_1)(\alpha_{i}-\lambda_2)\cdots(\alpha_{i}-\lambda_{n+3})}
\end{displaymath}
with $\alpha_i$'s in (\ref{e.mt1}). 
Note that (\ref{e.mt6}) says
\begin{displaymath}
{\widetilde{\mu}^X(x)}={\widetilde{\mu}^{X'}(f(x))}\in\widetilde{q}(\{II_{X,x}\}\times S_{(\varphi_1^{x},\varphi_2^{x})})\subset\,\widetilde\sM^{\rm PQ}_n\end{displaymath}
where $\widetilde{q}$ is the orbit map in Definition \ref{d.refinedm}. Recall Proposition \ref{p.isom}. Then (\ref{e.mt1}) and (\ref{e.mt6}) give (\ref{e.mt5}). 

\smallskip
If we translate (\ref{e.mt8}) and (\ref{e.mt9}) in terms of $\sigma_j$ and $\sigma_j'$, then 
\begin{displaymath}
v^{x}_i=-\frac{(\alpha_{i}-\alpha_1)(\alpha_{i}-\alpha_2)\cdots(\alpha_{i}-\alpha_{\check i})\cdots(\alpha_{i}-\alpha_n)}{\sigma_0\alpha_i^{n+3}+\sigma_1\alpha_i^{n+2}+\cdots+\sigma_{n+2}\alpha_i+\sigma_{n+3}}
\end{displaymath}
and
\begin{displaymath}
v^{f(x)}_i=-\frac{(\alpha_{i}-\alpha_1)(\alpha_{i}-\alpha_2)\cdots(\alpha_{i}-\alpha_{\check i})\cdots(\alpha_{i}-\alpha_n)}{\sigma_0'\alpha_i^{n+3}+\sigma_1'\alpha_i^{n+2}+\cdots+\sigma_{n+2}'\alpha_i+\sigma_{n+3}'}.
\end{displaymath}
Here, $\alpha_i$'s are constants, so (\ref{e.mt5}) is equivalent to
\begin{equation}\label{e.mt12}
\big[u^{x}_1:u^{x}_2:\cdots:u^{x}_n\big]=\big[u^{f(x)}_1:u^{f(x)}_2:\cdots:u^{f(x)}_n\big]
\in\BP(\C^n)
\end{equation}
where
\begin{displaymath}
u^{x}_i=\sigma_0\alpha_i^{n+3}+\sigma_1\alpha_i^{n+2}+\cdots+\sigma_{n+3}\end{displaymath}
and
\begin{displaymath}
u^{f(x)}_i=\sigma_0'\alpha_i^{n+3}+\sigma_1'\alpha_i^{n+2}+\cdots+\sigma_{n+3}'.
\end{displaymath}

\smallskip
Since $M^{(\varphi_1',\varphi_2')}_{\alpha_1}$ in (\ref{e.mt00}) has positive dimension and the discriminant map $\theta_{(\varphi_{1},\varphi_{2})}$ is finite,
there is $y\in M^{(\varphi_1',\varphi_2')}_{\alpha_1}$ satisfying $\theta_{(\varphi_{1},\varphi_{2})}(y)\neq \theta_{(\varphi_{1},\varphi_{2})}(x)\in \BP(\sB_n)$. Then 
\begin{displaymath}
\theta_{(\varphi'_{1},\varphi'_{2})}(f(y))=\theta_{(\varphi_{1},\varphi_{2})}(y)
=[(\alpha_1s-t)(\beta_2s-t)\cdots(\beta_ns-t)]\in\BP(\sB_n)
\end{displaymath}
for some distinct complex numbers $\{\beta_2, \beta_3, \ldots, \beta_n\}\neq\{\alpha_2, \alpha_3, \ldots, \alpha_n\}\subset\C$, and
\begin{equation}\label{e.mt15}
\big[u^{y}_1:u^{y}_2:\cdots:u^{y}_n\big]=\big[u^{f(y)}_1:u^{f(y)}_2:\cdots:u^{f(y)}_n\big]
\in\BP(\C^n)
\end{equation}
where 
\begin{equation}\label{e.mt18}
u^{y}_1=u^{x}_1, \quad u^{f(y)}_1=u^{f(x)}_1,
\end{equation}
\begin{displaymath}
u^{y}_{i>1}=\sigma_0\beta_i^{n+3}+\sigma_1\beta_i^{n+2}+\cdots+\sigma_{n+3},\end{displaymath}
and
\begin{displaymath}
u^{f(y)}_{i>1}=\sigma_0'\beta_i^{n+3}+\sigma_1'\beta_i^{n+2}+\cdots+\sigma_{n+3}'.
\end{displaymath}
Then 
\begin{equation}\label{e.mt19}
\big[u^{x}_1:\cdots:u^{x}_n:u^{y}_2:\cdots:u^{y}_n\big]=\big[u^{f(x)}_1:\cdots:u^{f(x)}_n:u^{f(y)}_2:\cdots:u^{f(y)}_n\big]
\in\BP(\C^{2n-1})
\end{equation}
by (\ref{e.mt12}), (\ref{e.mt15}), and (\ref{e.mt18}). Suppose $\alpha_{n+i-1}:=\beta_i\notin\{\alpha_2,\alpha_3,\ldots,\alpha_n\}$ for $i\in\{2,3,4,5\}$. (This is possible by considering a few more points in $M^{(\varphi_1',\varphi_2')}_{\alpha_1}$ if necessary.) 
Let $\mathbf M$ be the following $(n+4)\times(n+4)$ matrix:
\begin{displaymath}
\left[
\begin{array}{cccccccc}
(\alpha_1)^{n+3}&(\alpha_1)^{n+2}&\cdots&\alpha_1&1\\
(\alpha_{2})^{n+3}&(\alpha_{2})^{n+2}&\cdots&\alpha_{2}&1\\
\vdots&\vdots&&\vdots&\vdots\\
(\alpha_{n})^{n+3}&(\alpha_{n})^{n+2}&\cdots&\alpha_n&1\\
(\alpha_{n+1})^{n+3}&(\alpha_{n+1})^{n+2}&\cdots&\alpha_{n+1}&1\\
\vdots&\vdots&&\vdots&\vdots\\
(\alpha_{n+4})^{n+3}&(\alpha_{n+4})^{n+2}&\cdots&\alpha_{n+4}&1\\
\end{array}\right].
\end{displaymath}
Note that $\mathbf M$ is a Vandermonde matrix since $\alpha_1, \alpha_2, \ldots, \alpha_{n+4}$ are distinct complex numbers. Hence, $\mathbf M$ is invertible, and (\ref{e.mt19}) implies (\ref{e.mt4}). Therefore, $X$ and $X'$ are biregular to each other.
         \end{proof}


 
 \end{document}